\documentclass{amsart}
\usepackage{amsfonts, amsbsy, amsmath, amssymb}
\newtheorem{thm}{Theorem}[section]
\newtheorem{lem}[thm]{Lemma}
\newtheorem{cor}[thm]{Corollary}

\newtheorem{exmp}[thm]{Example}

\numberwithin{equation}{section}

\theoremstyle{definition}
\newtheorem{defn}[thm]{Definition}
\begin{document}

\title 
{Enumeration of ${\rm AGL}(\frac m3,\,{\Bbb F}_{p^3})$-Invariant Extended Cyclic Codes}

\author{Xiang-dong Hou}

\address{Department of Mathematics, University of South Florida, Tampa, Florida 33620}
\curraddr{}
\email{xhou@math.usf.edu}



\begin{abstract}
Let $p$ be a prime and let $r,\ e,\ m$ be positive integers such that $r|e$ and $e|m$. 
The enumeration of linear codes of length $p^m$ over ${\Bbb F}_{p^r}$ which are invariant under the 
affine linear group ${\rm AGL}(\frac me,\,{\Bbb F}_{p^e})$ is equivalent to the enumeration of
certain ideals in a partially ordered set $({\mathcal U},\,\prec)$ where
${\mathcal U}=\{0,1,\cdots,\frac me(p-1)\}^e$ and $\prec$ is defined by an $e$-dimensional simplicial cone.
When $e=2$, the enumeration problem was solved in an earlier paper. In the present paper, we consider the cases
$e=3$. We describe methods for enumerating all ${\rm AGL}(\frac m3,\,{\Bbb F}_{p^3})$-invariant linear codes of length $p^m$ over ${\Bbb F}_{p^r}$
\end{abstract}
\maketitle

%%%%%%%%%%%%%%%%%%%%%%%%%%%%%%%%%%%%%%%%%%
%   Section 1
%%%%%%%%%%%%%%%%%%%%%%%%%%%%%%%%%%%%%%%%%%
\section{Introduction}

Extended cyclic codes which are invariant under a certain affine linear group were first
studied by Kasami, Lin and Peterson \cite{Kas68} and by  Delsarte \cite{Del70}. 
These codes were further investigated by 
Charpin \cite{Cha90} \cite{Cha98}, by Berger \cite{Ber96-1}, Berger and Charpin \cite{Ber96} 
\cite{Ber99} in the context of permutation groups, and by
Charpin and Levy-Dit-Vehel \cite{Cha94}
in conjunction with self-duality. 
Extended cyclicity follows from affine invariance except when the code is the full ambient space; see later in the introduction.
Affine-invariant codes are interesting 
because of the large automorphism groups they possess. Examples of affine-invariant codes include the $q$-ary Reed-Muller codes which are precisely $\text{AGL}(m,\Bbb F_q)$-invariant codes of length $q^m$ over $\Bbb F_q$.

The interest of affine-invariant codes is not limited to coding theory. As we will
see below, such codes are precisely submodule of a certain module over the group algebra $\Bbb K[\text{AGL}(n,\Bbb F)]$ where $\Bbb F$ and $\Bbb K$ are two finite fields of the same characteristic. Therefore, affine-invariant codes provide concrete examples of modular representations of the affine linear group $\text{AGL}(n,\Bbb F)$. 

The present paper and its predecessor \cite{Hou} deal with the enumeration of 
affine-invariant codes. Delsarte's characterization of affine--invariant extended cyclic codes in terms of defining sets \cite{Del70} is the foundation of our work. The starting point of our approach is a reformulation (Theorem~\ref{T1.1}) of Delsarte's characterization; the reformulation changes the enumeration problem from an algebraic one to a combinatorial and geometric one.

A comprehensive introduction to affine-invariant extended cyclic codes can be found in 
\cite{Ber96}. A detailed introduction to our approach was given in \cite{Hou}. Thus in the
present introduction, we only give the essential facts to be used in the paper. 

Let $p$ be a prime and $r,\ m,\ e$ positive integers such that $e|m$. Identify
${\Bbb F}_{p^m}$ with ${\Bbb F}_{p^e}^{m/e}$.
Then the affine linear group ${\rm AGL}(\frac me, {\Bbb F}_{p^e})$ acts on ${\Bbb F}_{p^m}$
hence also acts on the group algebra ${\Bbb F}_{p^r}[({\Bbb F}_{p^m},\,+)]$. Put 
\[
G_{m,e}=
{\rm AGL}(\frac me,\, {\Bbb F}_{p^e}).
\] 
Then ${\Bbb F}_{p^r}[({\Bbb F}_{p^m},\,+)]$ 
is an ${\Bbb F}_{p^r}[G_{m,e}]$-module. Define
\[
{\mathcal M}=\Bigl\{\sum_{g\in{\Bbb F}_{p^m}}a_gX^g\in{\Bbb F}_{p^r}[({\Bbb F}_{p^m},\,+)]:
\sum_{g\in{\Bbb F}_{p^m}}a_g=0\Big\}.
\]
${\Bbb F}_{p^r}[G_{m,e}]$-submodules of ${\Bbb F}_{p^r}[({\Bbb F}_{p^m},\,+)]$ are $G_{m,e}$-{\em invariant
codes} over ${\Bbb F}_{p^r}$; ${\Bbb F}_{p^r}[G_{m,e}]$-submodules of ${\mathcal M}$ are $G_{m,e}$-{\em invariant
extended cyclic codes} over ${\Bbb F}_{p^r}$. In fact, every proper  
${\Bbb F}_{p^r}[G_{m,e}]$-submodules of ${\Bbb F}_{p^r}[({\Bbb F}_{p^m},\,+)]$ must be contained in
${\mathcal M}$ (\cite{Hou}).

As pointed out in \cite{Hou}, in order to determine ${\Bbb F}_{p^r}[G_{m,e}]$-submodules of ${\mathcal M}$ for all 
$r$, it suffices to determine those with $r|e$. Thus we always assume $r|e$.

Let
\[
P=\left[
\begin{matrix}
p^0&p^{e-1}&\cdots&p^1\cr
p^1&p^0&\cdots&p^2\cr
\vdots&\vdots&\ddots&\vdots\cr
p^{e-1}&p^{e-2}&\cdots&p^0\cr
\end{matrix}\right].
\]
For $u,v\in{\Bbb R}^e$, we say $u\prec v$ if $(u-v)P$ has all the coordinates $\le 0$.
Let $\Delta\subset {\Bbb R}^e$ be the set of all linear combinations of the rows of 
\[
(1-p^e)P^{-1}=\left[
\begin{matrix}
1&0&0&\cdots&0&-p\cr
-p&1&0&\cdots&0&0\cr
0&-p&1&\cdots&0&0\cr
\vdots&\vdots&\vdots&\ddots&\vdots&\vdots\cr
0&0&0&\cdots&1&0\cr
0&0&0&\cdots&-p&1\cr
\end{matrix}\right]
\]
with nonnegative coefficients. Namely, $\Delta$ is the $e$-dimensional simplicial cone
spanned by the rows of $(1-p^e)P^{-1}$. It is clear that $u\prec v$ if and only if 
$u\in v+\Delta$. The relation $\prec$ is a partial order in ${\Bbb R}^e$.

 Let
\[
A=\left[
\begin{matrix}
0&&&&&1\cr
1&0\cr
&1&0\cr
&&&\ddots\cr
&&&&0\cr
&&&&1&0\cr
\end{matrix}\right]_{e\times e}
\]
be the circulant permutation matrix. Since $AP=PA$, the matrix $A$ preserves the partial order $\prec$,
i.e., $u\prec v$ if and only if $uA\prec vA$.

For any subset $\Omega\subset{\Bbb R}^e$, $(\Omega,\,\prec)$ is a partially ordered set.
An {\em ideal} of $(\Omega,\,\prec)$ is a subset $I\subset\Omega$ such that for each $u\in I$
and $v\in \Omega$, $v\prec u$ implies $v\in I$.

Let
\[
{\mathcal U}=\Bigl\{0, 1, \cdots, \frac me(p-1)\Bigr\}^e.
\]
For each $s\in\{0, 1,\cdots,p^m-1\}$, write
\[
s=s_0p^0+\cdots+s_{m-1}p^{m-1},\quad 0\le s_i\le p-1,
\]
and define
\[
\sigma(s)=\Bigl[
\sum_{i\equiv 0\!\!\!\pmod e}s_i,\ \sum_{i\equiv 1\!\!\! \pmod e}s_i,\ \dots,\ \sum_{i\equiv e-1\!\!\! \pmod e}s_i
\Bigr]\in{\mathcal U}.
\]
The following is a reformulation of Delsarte's characterization of affine-invariant
extended cyclic codes \cite{Del70}:

\begin{thm}\label{T1.1} {\rm (\cite{Hou})} There is a one-to-one correspondence between the 
${\Bbb F}_{p^r}[G_{m,e}]$-submodules of ${\Bbb F}_{p^r}[({\Bbb F}_{p^m},\,+)]$ and the $A^r$-invariant ideals of
$({\mathcal U},\,\prec)$. If $I$ is an $A^r$-invariant ideal of $({\mathcal U},\,\prec)$, the corresponding ${\Bbb F}_{p^r}[G_{m,e}]$-submodules of ${\Bbb F}_{p^r}[({\Bbb F}_{p^m},\,+)]$ is
\begin{equation}\label{MI}
\begin{split}
M(I):=&\Bigl\{\sum_{g\in{\Bbb F}_{p^m}}a_gX^g\in{\Bbb F}_{p^r}[({\Bbb F}_{p^m},\,+)]:
\sum_{g\in{\Bbb F}_{p^m}}a_g g^s=0\cr
&\text{for all $s\in\{0, 1, \cdots, p^m-1\}$ with $\sigma(s)\in I$}\Bigr\}.
\end{split}
\end{equation}
In {\rm (\ref{MI})}, $0^0$ is defined as $1$. Moreover, $M(I)\subset{\mathcal M}$ if and only if
$I\ne\emptyset$.
\end{thm}

\noindent{\bf Note}. When $e=m$, i.e., when ${\mathcal U}=\{0,1,\dots,p-1\}^e$, the partial
order $\prec$ in ${\mathcal U}$ is the cartesian product of linear orders. Namely, $(x_1,\dots,x_e)\prec(y_1,\dots,y_e)$ in ${\mathcal U}$ if and only if $x_i\le y_i$ for
all $1\le i\le e$. However, this is not the case when $1<e<m$. 

\begin{exmp}\label{E1.2}
\rm
Let $p=3$, $m=6$, $e=3$, $r=1$ and
\[
\begin{split}
I=\bigl\{&(0,0,0),\ (1,0,0),\ (0,1,0),\ (0,0,1),\cr
&(0,1,1),\ (1,0,1),\ (1,1,0),\ (1,1,1),\cr
&(2,0,0),\ (0,2,0),\ (0,0,2),\cr
&(3,0,0),\ (0,3,0),\ (0,0,3)\bigr\}.
\end{split}
\]

%%%%%%%%%%%% Figure 1 %%%%%%%%%%%%%%%%%%%%%%%%%%%%%%%
\vskip5mm
\setlength{\unitlength}{8mm}
\[
\begin{picture}(8,8)
\put(3,3){\vector(1,0){5}}
\put(3,3){\vector(0,1){5}}
\put(3,3){\vector(-1,-1){3}}
\put(8,2.8){\makebox(0,0)[t]{$\scriptstyle x$}}
\put(3.1,8){\makebox(0,0)[l]{$\scriptstyle y$}}
\put(0,0.2){\makebox(0,0)[b]{$\scriptstyle z$}}
\put(6,2.8){\makebox(0,0)[t]{$\scriptstyle 3$}}
\put(3.1,6){\makebox(0,0)[l]{$\scriptstyle 3$}}
\put(1.5,1.7){\makebox(0,0)[b]{$\scriptstyle 3$}}
\multiput(3,3)(1,0){4}{{\makebox(0,0){$\scriptstyle \bullet$}}}
\multiput(3,4)(0,1){3}{{\makebox(0,0){$\scriptstyle \bullet$}}}
\multiput(2.5,2.5)(-0.5,-0.5){3}{{\makebox(0,0){$\scriptstyle \bullet$}}}
\put(4,4){{\makebox(0,0){$\scriptstyle \bullet$}}}
\put(2.5,3.5){{\makebox(0,0){$\scriptstyle \bullet$}}}
\put(3.5,2.5){{\makebox(0,0){$\scriptstyle \bullet$}}}
\put(3.5,3.5){{\makebox(0,0){$\scriptstyle \bullet$}}}
\put(2.5,2.5){\line(1,0){1}}
\put(2.5,3.5){\line(1,0){1}}
\put(3,4){\line(1,0){1}}
\put(2.5,2.5){\line(0,1){1}}
\put(3.5,2.5){\line(0,1){1}}
\put(4,3){\line(0,1){1}}
\put(2.5,3.5){\line(1,1){0.5}}
\put(3.5,3.5){\line(1,1){0.5}}
\put(3.5,2.5){\line(1,1){0.5}}
\end{picture}
\]
\[
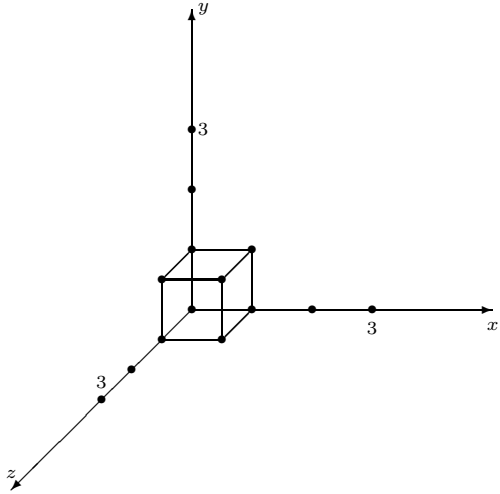
\text{Figure 1. The $A$-invariant ideal $I$}
\]
\vskip5mm
\noindent
It is easy to see that $I$ is an $A$-invariant ideal of $({\mathcal U},\,\prec)$. We have
\[
\begin{split}
\sigma^{-1}(I)=\bigl\{&0,\ 1,\ 2,\ 3,\ 4,\ 6,\ 9,\ 10,\ 12,\ 13,\ 18,\ 27,\ 28,\ 29,\ 30,\ 36,\ 39,\ 54,\cr
&55,\ 81,\ 82,\ 84,\ 87,\ 90,\ 91,\ 108,\ 117,\ 162,\ 165,\ 243,\ 244,\cr
&246,\ 247,\ 252,\ 261,\ 270,\ 273,\ 324,\ 325,\ 351,\ 486,\ 495\bigr\}.\cr
\end{split}
\]
The ${\Bbb F}_3[G_{6,3}]$-submodule of ${\mathcal M}$ corresponding to $I$ is 
\[
M(I)=\Bigl\{\sum_{g\in{\Bbb F}_{3^6}}a_gX^g\in{\Bbb F}_{3}[({\Bbb F}_{3^6},\,+)]:
\sum_{g\in{\Bbb F}_{3^6}}a_g g^s=0\
\text{for all $s\in\sigma^{-1}(I)$}\Bigr\}.
\]
\end{exmp}

Therefore, the essential problem is how to enumerate the $A^r$-invariant ideals of $({\mathcal U},\,\prec)$.
When $e=1$, the problem is trivial. When $e=2$, the problem has been solved in \cite{Hou}. The present paper deals with the case $e=3$. 
We will describe methods for enumerating all $A^r$-invariant ideals of $({\mathcal U},\,\prec)$ for $e=3$.

%%%%%%%%%%%%%%%%%%%%%%%%%%%%%%%%%%%%%%%%%%%%
%    section 2
%%%%%%%%%%%%%%%%%%%%%%%%%%%%%%%%%%%%%%%%%%%%
\section{Description of the Approach}

For simplicity, an ideal of $(\Omega,\,\prec)$, where $\Omega\subset {\Bbb R}^e$, is called an ideal of $\Omega$.

\begin{lem}\label{L2.1}
{\rm (i)} Let $\Omega\subset\Gamma\subset{\Bbb R}^e$ such that $\Omega$ and $\Gamma$ are $A^r$-invariant.
If $I$ is an $A^r$-invariant ideal of $\Omega$, then there  is an $A^r$-invariant ideal 
$J$ of $\Gamma$ such that $J\cap \Omega=I$.

{\rm (ii)} Let $\Omega\subset{\Bbb R}^e$ and $\Gamma\subset{\Bbb R}^e$. Let $I$ be an ideal of $\Omega$ and
$J$ an ideal of  $\Gamma$ such that $I\cap \Gamma= J\cap\Omega$. Then $I\cup J$ is an ideal of $\Omega\cup\Gamma$ if and only if 
\begin{equation}\label{Iff}
(I+\Delta)\cap\Gamma\subset J\quad\text{and}\quad (J+\Delta)\cap\Omega\subset I.
\end{equation}
\end{lem}

\begin{proof}
(i) Let $J=(I+\Delta)\cap\Gamma$. Then $J$ is an $A^r$-invariant ideal of $\Gamma$. Since $I$ is an ideal of
$\Omega$, we have $J\cap\Omega=(I+\Delta)\cap\Omega=I$.

(ii) $(\Rightarrow$) Since $(I+\Delta)\cap(\Omega\cup\Gamma)$ is the ideal of $\Omega\cup\Gamma$ generated by $I$, i.e., the smallest
ideal of $\Omega\cup\Gamma$ containing $I$, and since $I\cup J$ is an ideal of $\Omega\cup \Gamma$, we have $(I+\Delta)\cap(\Omega\cup \Gamma)\subset
I\cup J$. Hence
\[
\begin{split}
(I+\Delta)\cap\Gamma\,&= (I+\Delta)\cap(\Omega\cup\Gamma)\cap\Gamma\cr
&\subset(I\cup J)\cap \Gamma\cr
&=(I\cap\Gamma)\cup J\cr
&=(J\cap\Omega)\cup J\cr
&=J.\cr
\end{split}
\]
In the same way, $(J+\Delta)\cap\Omega\subset I$.

($\Leftarrow$) We have
\[
(I+\Delta)\cap(\Omega\cup\Gamma)=\bigl[(I+\Delta)\cap\Omega\bigr]\cup\bigl[(I+\Delta)\cap
\Gamma\bigr]\subset I\cup J
\]
since $(I+\Delta)\cap\Omega=I$ and, by (\ref{Iff}), $(I+\Delta)\cap \Gamma\subset J$. In the same way,
$(J+\Delta)\cap(\Omega\cup\Gamma)\subset I\cup J$. Therefore,
\[
\bigl[(I\cup J)+\Delta\bigr]\cap(\Omega\cup\Gamma)\subset I\cup J,
\]
which makes $I\cup J$ an ideal of $\Omega\cup \Gamma$.
\end{proof}

In general, all $A^r$-invariant ideals of ${\mathcal U}$ can be constructed using the following inductive strategy.
Partition ${\mathcal U}$ into $A^r$-invariant subsets ${\mathcal U}_1,\dots,{\mathcal U}_k$. 
Let $1\le i\le k$ and 
assume that for each 
$j$ with $j<i$, an $A^r$-invariant ideal $I_j$ of ${\mathcal U}_j$ has been constructed such that $\bigcup_{j<i}I_j$ is an ideal
of $\bigcup_{j<i}{\mathcal U}_j$. Construct an $A^r$-invariant ideal $I_i$ of ${\mathcal U}_i$ such that for all $j<i$,
\begin{equation}\label{Iff1}
(I_i+\Delta)\cap{\mathcal U}_j\subset I_j\quad\text{and}\quad (I_j+\Delta)\cap{\mathcal U}_i\subset I_i.
\end{equation}
Then by Lemma~\ref{L2.1} (ii), $\bigcup_{j\le i}I_j$ is an $A^r$-invariant ideal of $\bigcup_{j\le i}{\mathcal U}_j$.
Eventually, $I=\bigcup_{j\le k}{\mathcal U}_j$ is an $A^r$-invariant ideal of ${\mathcal U}$ with $I\cap{\mathcal U}_i=I_i$ for all $1\le i\le k$.
We shall call an ideal $I_i$ of ${\mathcal U}_i$ satisfying (\ref{Iff1}) {\em compatible} with $I_j$ ($j<i$).

\vskip2mm

\noindent{\bf Remarks}. (i) Constructing an $A^r$-invariant ideal $I$ in ${\mathcal U}$ is an $e$-dimensional 
geometric problem. By partitioning ${\mathcal U}$
suitably, constructing an $A^r$-invariant ideal $I_i$ in ${\mathcal U}_i$ becomes an $(e-1)$-dimensional geometric problem. 

(ii) Since for each $A^r$-invariant ideal $I$ of ${\mathcal U}$, $I\cap {\mathcal U}_i$ ($1\le i\le k$) is $A^r$-invariant ideal of 
${\mathcal U}_i$, the above strategy does enumerate all $A^r$-invariant ideals of ${\mathcal U}$.

(iii) The existence of an $A^r$-invariant ideal $I_i$ of ${\mathcal U}_i$ compatible with $I_j$ ($j<i$) is guaranteed by
Lemma~\ref{L2.1}. Hence the inductive construction can always be completed.
\vskip2mm

To turn the above strategy into an enumeration algorithm, what we essentially need are effective methods
for enumerating all $A^r$-invariant ideals $I_i$ which are compatible with an existing sequence of
$A^r$-invariant ideals $I_j$ ($j<i$). The main purpose of this paper is to provide such effective methods in
the case $e=3$.

Form now on, we assume $e=3$. Put
\[
n=\frac m3(p-1).
\]
Since $r|e$, there are two possibilities for $r$: $r=1$ or 3. When $r=3$, we partition ${\mathcal U}$ as
\begin{equation}\label{Ptn}
{\mathcal U}=\bigcup_{i=0}^n{\mathcal U}_i
\end{equation}
where
\[
{\mathcal U}_i=\bigl\{(x,y,z)\in{\mathcal U}:z=i\bigr\}.
\]
When $r=1$, we partition ${\mathcal U}$ as
\begin{equation}\label{Ptn1}
{\mathcal U}=\bigcup_{i=0}^n{\mathcal V}_i
\end{equation}
where
\[
{\mathcal V}_i=\bigl\{(x,y,z)\in{\mathcal U}:x\le i,\,y\le i,\, z\le i,\ \text{and at least one of $x,y,z$ is $i$}\bigr\}.
\]

Section 4 deals with the case $r=3$. We describe two methods for enumerating compatible ideals $I_i$ of ${\mathcal U}_i$. The method of forward slicing enumerates all ideals $I_i$ of ${\mathcal U}_i$ 
which are compatible with
ideals $I_j$ of ${\mathcal U}_j$ where $0\le j<i$;
the method of backward slicing enumerates all ideals $I_i$ of ${\mathcal U}_i$ 
which are compatible with
ideals $I_j$ of ${\mathcal U}_j$ where $i< j\le n$.
Section 5 deals with the case $r=1$.
We describe a method for enumerating all $A$-invariant ideals $I_i$ of ${\mathcal V}_i$ compatible with
$A$-invariant ideals $I_j$ of ${\mathcal V}_j$ where $0\le j<i$. In preparation for these attempts, in the next section, we first take a close look of the cross section of an ideal in ${\mathcal U}$ on a plane parallel to a coordinate plane. We also introduce the notion of walk in the next section.

%%%%%%%%%%%%%%%%%%%%%%%%%%%%%%%%%%%%%%%%%%%%
%    section 3
%%%%%%%%%%%%%%%%%%%%%%%%%%%%%%%%%%%%%%%%%%%%
\section{Cross Sections and Walks}

Let $c\in {\Bbb R}$. Observe that $\Delta\cap({\Bbb R}^2\times\{c\})$ consists of 
points $(x,y,c)\in{\Bbb R}^3$ satisfying
\[
\begin{cases}
x+py+p^2c\le 0,\cr
p^2x+y+pc\le 0,\cr
px+p^2y+c\le 0,\cr
\end{cases}
\]
i.e.,
\begin{equation}\label{Eq3.1}
\begin{cases}
x+py\le\min\{-p^2c,\,-\frac 1pc\},\cr
p^2x+y\le -pc.\cr
\end{cases}
\end{equation}
The solution set of (\ref{Eq3.1}) is depicted in Figure 2 when $c\ge 0$ and in
Figure 3 when $c<0$.

%%%%%%%%%%%% Figure 2 %%%%%%%%%%%%%%%%%%%%%%%%%%%%%%%
\vskip5mm
\setlength{\unitlength}{6mm}
\[
\begin{picture}(10,10)
\put(0,7){\vector(1,0){10}}
\put(5,0){\vector(0,1){10}}

\put(9,3){\line(-2,1){8.5}}
\put(4,9){\line(1,-4){2}}

\put(0,6){\line(1,1){1}}
\put(0,5){\line(1,1){1.65}}
\put(0,4){\line(1,1){2.3}}
\put(0,3){\line(1,1){3}}
\put(0,2){\line(1,1){3.65}}
\put(0,1){\line(1,1){4.3}}
\put(1,1){\line(1,1){4}}
\put(2,1){\line(1,1){3.2}}
\put(3,1){\line(1,1){2.4}}
\put(4,1){\line(1,1){1.6}}
\put(5,1){\line(1,1){0.75}}

\put(10,6.8){\makebox(0,0)[t]{$\scriptstyle x$}}
\put(5.2,10){\makebox(0,0)[l]{$\scriptstyle y$}}
\put(5.1,5.1){\makebox(0,0)[bl]{$\scriptstyle -pc$}}
\put(6,1.5){\makebox(0,0)[l]{$\scriptstyle \text{slope}=-p^2$}}
\put(7,4.2){\makebox(0,0)[bl]{$\scriptstyle \text{slope}=-\frac 1p$}}
\end{picture}
\]
\[
\text{Figure 2. The cross section of $\Delta$ on the plane $z=c$, $c\ge 0$}
\]
%%%%%%%%%%%%%%%% End Figure 2 %%%%%%%%%%%%%%%%%%%%%%%%%%%%%%%%%%%%%%

%%%%%%%%%%%% Figure 3 %%%%%%%%%%%%%%%%%%%%%%%%%%%%%%%
\vskip5mm
\setlength{\unitlength}{6mm}
\[
\begin{picture}(10,10)
\put(0,6){\vector(1,0){10}}
\put(5,0){\vector(0,1){10}}

\put(8,5){\line(-2,1){7}}
\put(5.5,8){\line(1,-4){1.75}}

\put(1,7){\line(1,1){1}}
\put(1,6){\line(1,1){1.65}}
\put(1,5){\line(1,1){2.3}}
\put(1,4){\line(1,1){3}}
\put(1,3){\line(1,1){3.65}}
\put(1,2){\line(1,1){4.3}}
\put(1,1){\line(1,1){5}}
\put(2,1){\line(1,1){4.2}}
\put(3,1){\line(1,1){3.4}}
\put(4,1){\line(1,1){2.6}}
\put(5,1){\line(1,1){1.8}}
\put(6,1){\line(1,1){1}}

\put(10,5.8){\makebox(0,0)[t]{$\scriptstyle x$}}
\put(5.2,10){\makebox(0,0)[l]{$\scriptstyle y$}}
\put(6,6.1){\makebox(0,0)[bl]{$\scriptstyle -\frac cp$}}
\put(6,8){\makebox(0,0)[l]{$\scriptstyle \text{slope}=-p^2$}}
\put(2,8){\makebox(0,0)[bl]{$\scriptstyle \text{slope}=-\frac 1p$}}
\end{picture}
\]
\[
\text{Figure 3. The cross section of $\Delta$ on the plane $z=c$, $c< 0$}
\]
%%%%%%%%%%%%%%%% End Figure 3 %%%%%%%%%%%%%%%%%%%%%%%%%%%%%%%%%%%%%%

%%%%%%%%%%%% Figure 4 %%%%%%%%%%%%%%%%%%%%%%%%%%%%%%%
\vskip5mm
\setlength{\unitlength}{6mm}
\[
\begin{picture}(10,10)
\put(0,6){\vector(1,0){10}}
\put(6,0){\vector(0,1){10}}

\put(8,5){\line(-2,1){7}}
\put(5.5,8){\line(1,-4){1.75}}

\put(1,7){\line(1,1){1}}
\put(1,6){\line(1,1){1.65}}
\put(1,5){\line(1,1){2.3}}
\put(1,4){\line(1,1){3}}
\put(1,3){\line(1,1){3.65}}
\put(1,2){\line(1,1){4.3}}
\put(1,1){\line(1,1){5}}
\put(2,1){\line(1,1){4.2}}
\put(3,1){\line(1,1){3.4}}
\put(4,1){\line(1,1){2.6}}
\put(5,1){\line(1,1){1.8}}
\put(6,1){\line(1,1){1}}

\put(10,5.8){\makebox(0,0)[t]{$\scriptstyle x$}}
\put(6.2,10){\makebox(0,0)[l]{$\scriptstyle y$}}
\put(7,3){\makebox(0,0)[l]{$\scriptstyle \text{slope}=-p^2$}}
\put(2,8){\makebox(0,0)[bl]{$\scriptstyle \text{slope}=-\frac 1p$}}
\end{picture}
\]
\[
\text{Figure 4. The region $D$}
\]
%%%%%%%%%%%%%%%% End Figure 3 %%%%%%%%%%%%%%%%%%%%%%%%%%%%%%%%%%%%%%
\vskip4mm

Let 
\[
D=\bigl\{(x,y)\in{\Bbb R}^2: x+py\le 0,\ p^2x+y\le 0\bigr\}
\]
(Figure 4). We can write
\begin{equation}\label{Cro}
\Delta\cap({\Bbb R}^2\times\{c\})=
\begin{cases}
\bigl(D-c(0,p)\bigr)\times\{c\},&\text{if}\ c\ge 0,\cr
\bigl(D-c(\frac 1p,0)\bigr)\times\{c\},&\text{if}\ c< 0.\cr
\end{cases}
\end{equation}
Given $(x_1,y_1,z_1)$ and $(x_2,y_2,z_2)$ in ${\Bbb R}^3$, $(x_1,y_1,z_1)\prec (x_2,y_2,z_2)$ if and
only if $(x_1,y_1,z_1)-(x_2,y_2,z_2)\in\Delta\cap\bigl({\Bbb R}^2\times\{z_1-z_2\}\bigr)$. By (\ref{Cro}), this happens
if and only if
\begin{equation}\label{Iff2}
(x_1,y_1)\in(x_2,y_2)+
\begin{cases}
D-(z_1-z_2)(0,p),&\text{if}\ z_1\ge z_2,\cr
D-(z_1-z_2)(\frac 1p,0),&\text{if}\ z_1< z_2.\cr
\end{cases}
\end{equation}
Thus
\begin{equation}\label{x2y2z2}
\begin{split}
&\bigl[(x_2,y_2,z_2)+\Delta\bigr]\cap\bigl[{\Bbb R}^2\times\{z_1\}\bigr]\cr
=\,&
\begin{cases}
\bigl[(x_2,y_2)+D-(z_1-z_2)(0,p)\bigr]\times\{z_1\},&\text{if}\ z_1\ge z_2,\cr
\bigl[(x_2,y_2)+D-(z_1-z_2)(\frac1p,0)\bigr]\times\{z_1\},&\text{if}\ z_1< z_2.\cr
\end{cases}
\end{split}
\end{equation}
By symmetry, we also see that $(x_1,y_1,z_1)\prec (x_2,y_2,z_2)$ if and only if
\begin{equation}\label{Iff3}
(y_1,z_1)\in(y_2,z_2)+
\begin{cases}
D-(x_1-x_2)(0,p),&\text{if}\ x_1\ge x_2,\cr
D-(x_1-x_2)(\frac 1p,0),&\text{if}\ x_1< x_2,\cr
\end{cases}
\end{equation}
which is equivalent to 
\begin{equation}\label{Iff4}
(z_1,x_1)\in(z_2,x_2)+
\begin{cases}
D-(y_1-y_2)(0,p),&\text{if}\ y_1\ge y_2,\cr
D-(y_1-y_2)(\frac 1p,0),&\text{if}\ y_1< y_2.\cr
\end{cases}
\end{equation}

\begin{lem}\label{L3.1}
Let $c$ be an integer written in the form $c=ap+b$ where $a,b\in{\Bbb Z}$,
$0\le b\le p-1$. Then
\[
\Bigl[c(\frac 1p,0)+D\Bigr]\cap{\Bbb Z}^2=\Bigl[\bigl\{(a,0),\ (a+1,-p^2+pb)\bigr\}+D\Bigr]\cap{\Bbb Z}^2.
\]
\end{lem}

\begin{proof}
Note that 
\[
\Bigl[c(\frac 1p,0)+D\Bigr]\setminus\Bigl[\bigl\{(a,0),\ (a+1,-p^2+pb)\bigr\}+D\Bigr]
\]
is the indicated region in Figure 5. Obvious, this region does not contain any points in ${\Bbb Z}^2$
\end{proof}

%%%%%%%%%%%% Figure  %%%%%%%%%%%%%%%%%%%%%%%%%%%%%%%
\vskip5mm
\setlength{\unitlength}{6mm}
\[
\begin{picture}(14,12)
\put(0,7){\vector(1,0){14}}
\put(2,0){\vector(0,1){12}}

\put(8,7){\line(-2,1){7}}
\put(10,7){\line(-2,1){7}}
\put(11,3){\line(-2,1){4}}
\put(8,7){\line(1,-4){1.5}}
\put(10,7){\line(1,-4){1.75}}

\multiput(8,7)(-0.66,0.33){9}{\line(1,1){0.66}}
\put(8.2,6.3){\line(1,1){1.1}}
\put(8.4,5.5){\line(1,1){1.4}}
\put(8.6,4.7){\line(1,1){1.4}}
\put(9,4){\line(1,1){1.3}}
\put(9.7,3.6){\line(1,1){0.9}}
\put(11,7){\line(0,1){0.1}}

\put(14,6.8){\makebox(0,0)[t]{$\scriptstyle x$}}
\put(2.2,12){\makebox(0,0)[l]{$\scriptstyle y$}}
\put(1.8,6.8){\makebox(0,0)[tr]{$\scriptstyle 0$}}
\put(7.9,6.8){\makebox(0,0)[tr]{$\scriptstyle a$}}
\put(9.5,7.2){\makebox(0,0)[bl]{$\scriptstyle a+\frac bp=\frac cp$}}
\put(11,6.8){\makebox(0,0)[t]{$\scriptstyle a+1$}}
\put(11.2,3){\makebox(0,0)[l]{$\scriptstyle (a+1, -p^2+pb)$}}
\put(4.2,10){\makebox(0,0)[bl]{$\scriptstyle \text{slope}=-\frac 1p$}}
\put(11.6,1){\makebox(0,0)[l]{$\scriptstyle \text{slope}=-p^2$}}

\end{picture}
\]

\[
\text{Figure 5. Proof of Lemma~\ref{L3.1}}
\]
%%%%%%%%%%  End Figure %%%%%%%%%%%%%%%%%%%%%%%%%%%%%%
\vskip4mm

The restriction of $\prec$ on the $xy$-plane, still denoted by $\prec$, is defined by the 2-dimensional cone $D$:
$(x_1,y_1)\prec (x_2,y_2)$ if and only if $(x_1,y_1)\prec (x_2,y_2)+D$. It is clear that for $I\subset\Omega\subset
{\Bbb R}^2$ and $c\in{\Bbb R}$, $I\times\{c\}$ is an ideal of $\Omega\times \{c\}$ if and only if $I$ is an ideal of 
$\Omega$.

For integers $a\le b$, let
\[
[a,b]=\{x\in{\Bbb Z}:a\le x\le b\}.
\]
Following the approach in \cite{Hou}, we can characterize ideals of a rectangle in ${\Bbb Z}^2$ by their boundaries. Such boundaries are called walks.

\begin{defn}\label{D3.1}
Let $a\le b$ and $c\le d$ be integers. A {\em walk} in $[a,b]\times [c,d]$ is a sequence
\begin{equation}\label{Wlk}
(x_0,y_0),\ (x_1,y_1),\ \dots,\ (x_k,y_k)
\end{equation}
in $[a,b]\times [c,d]$ satisfying the following conditions.

\begin{enumerate}
\item[(i)] 
$x_0=a$ or $y_0=d$; $x_k=b$ or $y_k=c$.

\item[(ii)] 
For each $0< i\le k$, either $(x_i,y_i)=(x_{i-1}+h,\,y_{i-1})$ for some $1\le h\le p$ or $(x_i,y_i)=(x_{i-1},\,y_{i-1}-v)$ for some $1\le v\le p^2$. In the first case,
$\bigl((x_{i-1},y_{i-1}),\ (x_i,y_i)\bigr)$ is called a {\em horizontal step of length} $h$;
in the second case,
$\bigl((x_{i-1},y_{i-1}),\ (x_i,y_i)\bigr)$ is called a {\em vertical step of length} $v$.

\item[(iii)]
The steps in the sequence (\ref{Wlk}) alternate between horizontal and vertical.

\item[(iv)]
If $a\le x_0<b$ and $y_0=d$, the first step is vertical; if $x_k=b$ and $c\le y_k<d$, the last step in horizontal.

\item[(v)]
If the first step is horizontal of length $h$, then $1\le h\le p-1$; if the last step is vertical of length $v$, then $1\le v\le p^2-1$
\end{enumerate}
\end{defn}

Let $U=[a,b]\times[c,d]$. For each walk $W=\bigl((x_0,y_0),\dots,(x_k,y_k)\bigr)$ in $U$, denote by $\iota_U(W)$ 
the lower left part of $U$ bounded by $W$ (see Figure 6), i.e.,
\[
\iota_U(W)=\bigl\{(x,y)\in U: x\le x_i\ \text{and}\ y\le y_i\ \text{for some}\ 0\le i\le k\bigr\}.
\]

%%%%%%%%%%%% Figure 5 %%%%%%%%%%%%%%%%%%%%%%%%%%%%%%%
\vskip5mm
\setlength{\unitlength}{5mm}
\[
\begin{picture}(12,12)
\put(0,0){\vector(1,0){12}}
\put(0,0){\vector(0,1){12}}
\put(0,10){\line(1,0){10}}
\put(10,0){\line(0,1){10}}
\put(0,9){\line(1,0){1}}
\put(1,9){\line(0,-1){2}}
\put(1,7){\line(1,0){2}}
\put(3,7){\line(0,-1){4}}
\put(3,3){\line(1,0){2}}
\put(5,3){\line(0,-1){3}}

\put(0,9){\makebox(0,0){$\scriptstyle \bullet$}}
\multiput(1,9)(0,-1){3}{\makebox(0,0){$\scriptstyle \bullet$}}
\put(2,7){\makebox(0,0){$\scriptstyle \bullet$}}
\multiput(3,7)(0,-1){5}{\makebox(0,0){$\scriptstyle \bullet$}}
\put(4,3){\makebox(0,0){$\scriptstyle \bullet$}}
\multiput(5,3)(0,-1){4}{\makebox(0,0){$\scriptstyle \bullet$}}

\put(-0.1,-0.1){\makebox(0,0)[tr]{$\scriptstyle 0$}}
\put(-0.1,10){\makebox(0,0)[r]{$\scriptstyle 10$}}
\put(10,-0.1){\makebox(0,0)[t]{$\scriptstyle 10$}}
\put(-0.1,9){\makebox(0,0)[r]{$\scriptstyle (x_0,y_0)$}}
\put(5,-0.1){\makebox(0,0)[t]{$\scriptstyle (x_6,y_6)$}}
\put(2,2){\makebox(0,0){$\scriptstyle \iota(W)$}}

\multiput(1,0)(1,0){10}{\line(0,1){0.1}}
\multiput(0,1)(0,1){10}{\line(1,0){0.1}}

\end{picture}
\]
\vskip5mm
\[
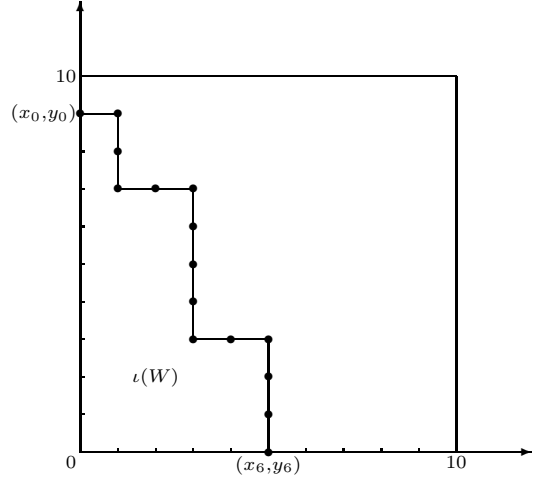
\text{Figure 6. A walk $W$ in $[0,10]\times[0,10]$ and its corresponding ideal $\iota(W)$, $p=2$}
\]
%%%%%%%%%%%%%%%% End Figure 5 %%%%%%%%%%%%%%%%%%%%%%%%%%%%%%%%%%%%%%
\vskip3mm

We denote the empty walk in $U$ by $\emptyset$ and define $\iota_U(\emptyset)=\emptyset$. Then 
\[
W\longmapsto \iota_U(W)
\]
is a bijection from the set ${\mathcal W}_U$ of all walks in $U$ to the set
${\mathcal I}_U$ of all ideals of $U$. In fact, the conditions in Definition~\ref{D3.1} are necessary and sufficient to ensure that for
every $u\in\iota_U(W)$, $(u+D)\cap U\subset\iota_U(W)$. The inverse map $\iota_U^{-1}:{\mathcal I}_U
\rightarrow{\mathcal W}_U$ is denoted by $\omega_U$.
When $U$ is clear from the context, $\iota_U$ and $\omega_U$ are 
simply written as $\iota$ and $\omega$. We call a walk $W$ the {\em boundary} of the ideal $\iota(W)$
and $\iota(W)$ the {\em ideal bounded by} $W$. 
We remind the reader that the {\em boundary} here is unrelated to the {\em border} in \cite{Ber96}

For two walks $W_1,\ W_2\in{\mathcal W}_U$, we say that $W_1\le  W_2$ if
$\iota(W_1)\subset\iota(W_2)$, which simply means that $W_1$ is below and 
to the left of $W_2$. The partially ordered set $({\mathcal I}_U,\,\subset)$
is a lattice where ``$\wedge$'' is ``$\cap$'' and ``$\vee$'' is ``$\cup$''.
Consequently, $({\mathcal W}_U,\,\le)$ is also a lattice with
\[
W_1\wedge W_2=\omega\bigl(\iota(W_1)\cap\iota(W_2)\bigr)
\]
and
\[
W_1\vee W_2=\omega\bigl(\iota(W_1)\cup\iota(W_2)\bigr).
\]

We introduce some operations on walks. Let $U_i=[a_i,b_i]\times[c_i,d_i]$
($i=1,2$), where $a_i\le b_i$ and $c_i\le d_i$ are integers, and assume
$U_1\supset U_2$. Let $W$ be a walk in $U_1$ and let $I=\iota_{U_1}(W)$.
The {\em restriction} of $W$ in $U_2$, denoted by $W|_{U_2}$, is defined to be
$\omega_{U_2}(I\cap U_2)$. If $U_2\subset\iota_{U_1}(W)$, $W|_{U_2}$ is the point $(b_2,d_2)$; otherwise, $W|_{U_2}$ is the walk in $U_2$ consisting of steps and partial steps of $W$. (See Figure 7.)

%%%%%%%%%%%% Figure  %%%%%%%%%%%%%%%%%%%%%%%%%%%%%%%
\vskip8mm
\setlength{\unitlength}{6mm}
\[
\begin{picture}(12,10)
\put(0,0){\line(1,0){12}}
\put(0,10){\line(1,0){12}}
\put(0,0){\line(0,1){10}}
\put(12,0){\line(0,1){10}}
\put(3,2){\line(1,0){7}}
\put(3,8){\line(1,0){7}}
\put(3,2){\line(0,1){6}}
\put(10,2){\line(0,1){6}}

\put(3,0){\line(0,1){0.1}}
\put(10,0){\line(0,1){0.1}}
\put(0,2){\line(1,0){0.1}}
\put(0,8){\line(1,0){0.1}}

\put(0.1,-0.1){\makebox(0,0)[t]{$\scriptstyle a_1$}}
\put(3,-0.1){\makebox(0,0)[t]{$\scriptstyle a_2$}}
\put(10,-0.1){\makebox(0,0)[t]{$\scriptstyle b_2$}}
\put(12,-0.1){\makebox(0,0)[t]{$\scriptstyle b_1$}}
\put(-0.1,0.1){\makebox(0,0)[r]{$\scriptstyle c_1$}}
\put(-0.1,2){\makebox(0,0)[r]{$\scriptstyle c_2$}}
\put(-0.1,8){\makebox(0,0)[r]{$\scriptstyle d_2$}}
\put(-0.1,10){\makebox(0,0)[r]{$\scriptstyle d_1$}}

\put(0,9){\line(1,0){1}}
\put(1,9){\line(0,-1){1}}
\put(1,8){\line(1,0){2}}
\put(4,8){\line(0,-1){2}}
\put(4,6){\line(1,0){2}}
\put(6,6){\line(0,-1){1}}
\put(6,5){\line(1,0){2}}
\put(8,5){\line(0,-1){1}}
\put(8,4){\line(1,0){2}}
\put(10,2){\line(1,0){2}}

\multiput(0,9)(1,0){2}{\makebox(0,0){$\scriptstyle \bullet$}}
\multiput(1,8)(1,0){4}{\makebox(0,0){$\scriptstyle \bullet$}}
\put(4,7){\makebox(0,0){$\scriptstyle \bullet$}}
\multiput(4,6)(1,0){3}{\makebox(0,0){$\scriptstyle \bullet$}}
\multiput(6,5)(1,0){3}{\makebox(0,0){$\scriptstyle \bullet$}}
\multiput(8,4)(1,0){3}{\makebox(0,0){$\scriptstyle \bullet$}}
\multiput(10,2)(1,0){3}{\makebox(0,0){$\scriptstyle \bullet$}}
\put(10,3){\makebox(0,0){$\scriptstyle \bullet$}}

\put(0,9){\line(3,1){6}}
\put(4,8){\line(2,1){2}}
\put(10,4){\line(-1,2){2.5}}
\put(12,2){\line(-1,2){4.5}}

\put(6.75,11){\makebox(0,0){$\scriptstyle W$}}
\put(6.75,9){\makebox(0,0){$\scriptstyle W|_{U_2}$}}
\put(1,1){\makebox(0,0){$\scriptstyle U_1$}}
\put(4,3){\makebox(0,0){$\scriptstyle U_2$}}

\end{picture}
\]
\vskip3mm
\[
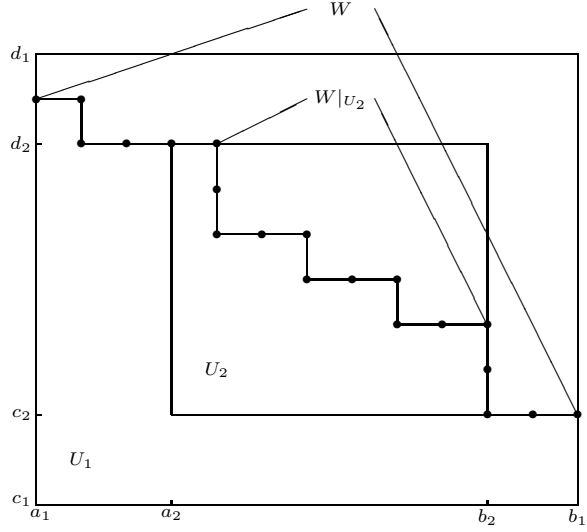
\text{Figure 7. The restriction of a walk}
\]
%%%%%%%%%%%%%%%% End Figure %%%%%%%%%%%%%%%%%%%%%%%%%%%%%%%%%%%%%%
\vskip8mm

For $h,\ v\in{\Bbb Z}$, the shift of $W$ by $h$ horizontal units and $v$ vertical units is a walk in $[a_1+h,b_1+h]\times[c_1+v,d_1+v]$ and is denoted by $W+(h,v)$.

Let $Z$ be a walk in $U_2$ and let $J=\iota_{U_2}(Z)$. A walk $W$ in $U_1$
is called an {\em extension} of $Z$ if $W|_{U_2}=Z$. Let $\overline Z_{U_1}$ and
$\underline Z_{U_1}$ be the {\em highest} and {\em lowest} (the largest and lowest with
respect to $\le$) extensions of $Z$ in $U_1$ respectively. Then 
$\overline Z_{U_1}=\omega_{U_1}(K)$ where $K$ is the largest ideal of $U_1$
such that $K\cap U_2=J$ and $\underline Z_{U_1}=\omega_{U_1}(L)$ where $L$ is the smallest ideal of $U_1$
such that $L\cap U_2=J$. In fact, $\underline Z_{U_1}$ is the boundary of 
$(J+D)\cap U_1$. $\overline Z_{U_1}$ can be obtained from $Z$ easily:
If $Z$ is the point $(b_2,d_2)$ (i.e., $J=U_2$), $\overline Z_{U_1}$ is the point $(b_1,d_1)$.
If $Z$ is not the point $(b_2,d_2)$ and $Z\ne\emptyset$, we extend $Z$ to the lower right with steps alternating between horizontal ones of largest
possible lengths and vertical ones of length 1, and to the upper left with steps alternating between vertical ones of largest
possible lengths and horizontal ones of length 1. If $Z=\emptyset$ and $(a_1,b_1)\ne(a_2,b_2)$, we start from the point $\bigl(
\max\{a_2-1, a_1\},\,\max\{b_2-1, b_1\}\bigr)$ and extend to the lower
right and to the upper left as described above. (See Figure 8.) If $Z=\emptyset$ and $(a_1,b_1)=(a_2,b_2)$, then $\overline Z_{U_1}=\emptyset$.  $\underline Z_{U_1}$ is obtained in a similar way. (See Figure 9.)
\vfill\eject

%%%%%%%%%%%% Figure  %%%%%%%%%%%%%%%%%%%%%%%%%%%%%%%
\vskip5mm
\setlength{\unitlength}{4mm}
\hskip5mm
\begin{picture}(11,10)
\put(0,0){\line(1,0){11}}
\put(0,10){\line(1,0){11}}
\put(0,0){\line(0,1){10}}
\put(11,0){\line(0,1){10}}
\put(3,3){\line(1,0){6}}
\put(3,8){\line(1,0){6}}
\put(3,3){\line(0,1){5}}
\put(9,3){\line(0,1){5}}

\put(2,10){\line(0,-1){5}}
\put(2,5){\line(1,0){2}}
\put(4,5){\line(0,-1){1}}
\put(4,4){\line(1,0){3}}
\put(7,4){\line(0,-1){2}}
\put(7,2){\line(1,0){3}}
\put(10,2){\line(0,-1){1}}
\put(10,1){\line(1,0){1}}

\put(8,7){\makebox(0,0){$\scriptstyle U_2$}}
\put(10,9){\makebox(0,0){$\scriptstyle U_1$}}
\multiput(2,10)(0,-1){5}{\makebox(0,0){$\scriptstyle \bullet$}}
\put(2,5){\makebox(0,0){$\scriptstyle \bullet$}}
\put(3,5){\makebox(0,0){$\scriptstyle \bullet$}}
\put(4,5){\makebox(0,0){$\scriptstyle \bullet$}}
\multiput(4,4)(1,0){3}{\makebox(0,0){$\scriptstyle \bullet$}}
\put(7,4){\makebox(0,0){$\scriptstyle \bullet$}}
\put(7,3){\makebox(0,0){$\scriptstyle \bullet$}}
\multiput(7,2)(1,0){3}{\makebox(0,0){$\scriptstyle \bullet$}}
\put(10,2){\makebox(0,0){$\scriptstyle \bullet$}}
\put(10,1){\makebox(0,0){$\scriptstyle \bullet$}}
\put(11,1){\makebox(0,0){$\scriptstyle \bullet$}}
\end{picture}
\hskip2cm
%%%%%
\begin{picture}(11,10)
\put(0,0){\line(1,0){11}}
\put(0,10){\line(1,0){11}}
\put(0,0){\line(0,1){10}}
\put(11,0){\line(0,1){10}}

\put(4,4){\line(1,0){6}}
\put(4,9){\line(1,0){6}}
\put(4,4){\line(0,1){5}}
\put(10,4){\line(0,1){5}}

\put(3,3){\line(1,0){3}}
\put(6,3){\line(0,-1){1}}
\put(6,2){\line(1,0){3}}
\put(9,2){\line(0,-1){1}}
\put(9,1){\line(1,0){2}}
\put(3,3){\line(0,1){7}}

\put(9,8){\makebox(0,0){$\scriptstyle U_2$}}
\put(10.5,9.5){\makebox(0,0){$\scriptstyle U_1$}}
\multiput(3,10)(0,-1){7}{\makebox(0,0){$\scriptstyle \bullet$}}
\multiput(3,3)(1,0){3}{\makebox(0,0){$\scriptstyle \bullet$}}
\put(6,3){\makebox(0,0){$\scriptstyle \bullet$}}
\multiput(6,2)(1,0){3}{\makebox(0,0){$\scriptstyle \bullet$}}
\put(9,2){\makebox(0,0){$\scriptstyle \bullet$}}
\multiput(9,1)(1,0){2}{\makebox(0,0){$\scriptstyle \bullet$}}
\put(11,1){\makebox(0,0){$\scriptstyle \bullet$}}

\end{picture}

\[
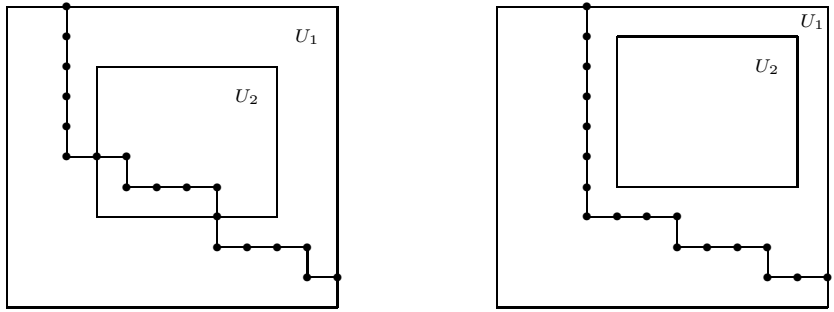
\text{Figure 8. Examples of highest extensions, $p=3$}
\]

%%%%%%%%%%%% Figure  %%%%%%%%%%%%%%%%%%%%%%%%%%%%%%%
\vskip5mm
\setlength{\unitlength}{4mm}
\hskip5mm
\begin{picture}(11,10)
\put(0,0){\line(1,0){11}}
\put(0,10){\line(1,0){11}}
\put(0,0){\line(0,1){10}}
\put(11,0){\line(0,1){10}}
\put(2,2){\line(1,0){7}}
\put(2,8){\line(1,0){7}}
\put(2,2){\line(0,1){6}}
\put(9,2){\line(0,1){6}}

\put(0,8){\line(1,0){2}}
\put(3,8){\line(0,-1){1}}
\put(3,7){\line(1,0){1}}
\put(4,7){\line(0,-1){3}}
\put(4,4){\line(1,0){2}}
\put(6,4){\line(0,-1){1}}
\put(6,3){\line(1,0){1}}
\put(7,3){\line(0,-1){3}}

\put(8,7){\makebox(0,0){$\scriptstyle U_2$}}
\put(10,9){\makebox(0,0){$\scriptstyle U_1$}}
\put(0,9){\makebox(0,0){$\scriptstyle \bullet$}}
\multiput(0,8)(1,0){3}{\makebox(0,0){$\scriptstyle \bullet$}}
\put(3,8){\makebox(0,0){$\scriptstyle \bullet$}}
\put(3,7){\makebox(0,0){$\scriptstyle \bullet$}}
\multiput(4,7)(0,-1){3}{\makebox(0,0){$\scriptstyle \bullet$}}
\multiput(4,4)(1,0){2}{\makebox(0,0){$\scriptstyle \bullet$}}
\put(6,4){\makebox(0,0){$\scriptstyle \bullet$}}
\put(6,3){\makebox(0,0){$\scriptstyle \bullet$}}
\put(7,3){\makebox(0,0){$\scriptstyle \bullet$}}
\multiput(7,2)(0,-1){2}{\makebox(0,0){$\scriptstyle \bullet$}}
\put(7,0){\makebox(0,0){$\scriptstyle \bullet$}}
\end{picture}
\hskip2cm
%%%%%
\begin{picture}(11,10)
\put(0,0){\line(1,0){11}}
\put(0,10){\line(1,0){11}}
\put(0,0){\line(0,1){10}}
\put(11,0){\line(0,1){10}}

\put(3,2){\line(1,0){6}}
\put(3,7){\line(1,0){6}}
\put(3,2){\line(0,1){5}}
\put(9,2){\line(0,1){5}}

\put(0,9){\line(1,0){3}}
\put(3,9){\line(0,-1){1}}
\put(3,8){\line(1,0){3}}
\put(6,8){\line(0,-1){1}}
\put(9,2){\line(0,-1){2}}

\put(8,6){\makebox(0,0){$\scriptstyle U_2$}}
\put(10,9){\makebox(0,0){$\scriptstyle U_1$}}
\multiput(0,9)(1,0){3}{\makebox(0,0){$\scriptstyle \bullet$}}
\put(3,9){\makebox(0,0){$\scriptstyle \bullet$}}
\multiput(3,8)(1,0){3}{\makebox(0,0){$\scriptstyle \bullet$}}
\put(6,8){\makebox(0,0){$\scriptstyle \bullet$}}
\multiput(6,7)(1,0){3}{\makebox(0,0){$\scriptstyle \bullet$}}
\multiput(9,7)(0,-1){7}{\makebox(0,0){$\scriptstyle \bullet$}}
\put(9,0){\makebox(0,0){$\scriptstyle \bullet$}}

\end{picture}

\[
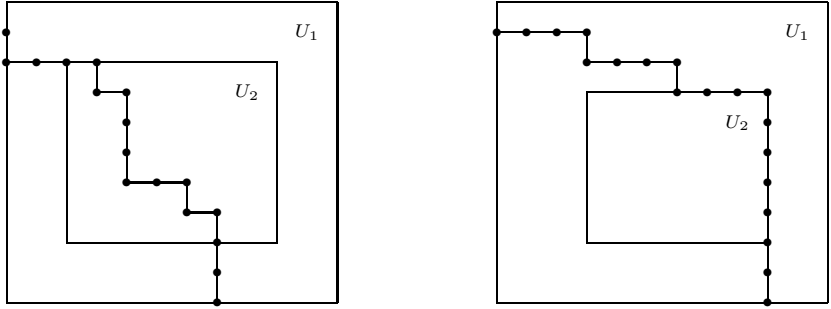
\text{Figure 9. Examples of lowest extensions, $p=3$}
\]

\vskip5mm
%%%%%%%%%%%%%%%%%%%%%%%%%%%%%%%%%%%%%%%%%%%%%

We list some obvious properties of restrictions and extensions. Let $U_i$, $i=1,2,3$, be rectangles in
${\Bbb Z}^2$ such that $U_1\supset U_2\supset U_3$. Let $W$ be a walk in $U_1$ and $Z$ a walk in $U_3$.
We have
\[
(W|_{U_2})|_{U_3}=W|_{U_3},
\]
\[
\overline{(\overline Z_{U_2})}_{U_1}=\overline Z_{U_1},
\]
\[
\underline{(\underline Z_{U_2})}_{U_1}=\underline Z_{U_1},
\]
\[
(\overline Z_{U_2})|_{U_3}=(\underline Z_{U_2})|_{U_3}=Z.
\]
\vskip8mm

%%%%%%%%%%%%%%%%%%%%%%%%%%%%%%%%%%%%%%%%%%%%
%    section 4
%%%%%%%%%%%%%%%%%%%%%%%%%%%%%%%%%%%%%%%%%%%%
\section{Enumerating Ideals of ${\mathcal U}$}

In this section we assume $r=3$. Since $A^3$ is the identity matrix, $A^3$-invariant ideals of ${\mathcal U}$ are simply ideals of ${\mathcal U}$.
Recall that $n=\frac m3(p-1)$. Put
\[
U=[0,n]^2,
\]
and partition ${\mathcal U}$ as
\[
{\mathcal U}=\bigcup_{i=0}^n\bigl(U\times\{i\}\bigr).
\]
A sequence of ideals $J_0,\dots,J_{i-1}$ (or $J_{i+1},\dots,J_n$) of $U$ is called {\em forward} (respectively, {\em backward}) {\em consistent} if 
$\bigcup_{j=0}^{i-1}(J_j\times\{j\})$ is an ideal of $U\times[0,i-1]$ (respectively, $\bigcup_{j=i+1}^n(J_j\times\{j\})$ is an ideal of $U\times[i+1,n]$).  An ideal $J_i$ of $U$ is said to be {\em consistent} with $J_0,\dots,J_{i-1}$
(or $J_{i+1},\dots,J_n$) if $J_0,\dots,J_{i-1},J_i$ (respectively, $J_i,J_{i+1},\dots,J_n$) is forward (backward) consistent. 

\medskip
\noindent{\bf Note.}
In the terminology of Section 2, the statement that $J_i$ is {\em consistent} with $J_0,\dots,J_{i-1}$ means that $J_i\times\{i\}$ is {\em compatible}
with $J_j\times\{j\}$, $0\le j<i$, with respect to the partition $\mathcal U=\bigcup_{j=0}^n(U\times\{j\})$. The meaning of the statement that 
$J_i$ is {\em consistent} with $J_{i+1},\dots,J_n$ is similar.

\medskip

Given a forward consistent sequence of ideals $J_0,\dots,J_{i-1}$ (or a backward consistent sequence $J_{i+1},\dots,J_n$), our goal in this section is to enumerate all ideals $J_i$ of $U$ which are consistent
with $J_0,\dots,J_{i-1}$ (or $J_{i+1},\dots,J_n$). 
When $n<p$, the problem is trivial:
In this case, the partial order $\prec$ in ${\mathcal U}$ is the cartesian product of linear orders, hence
$J_i$ is consistent with $J_0,\dots,J_{i-1}$ (or $J_{i+1},\dots,J_n$) if and only if
$J_i\subset J_{i-1}$ (or $J_i\supset J_{i+1}$.) When $n\ge p$,
the problem is more complex.
The main result of this section is the determination of two
walks $X_i$ and $Y_i$ in $U$, which can be computed  from the
boundaries of $J_0,\dots,J_{i-1}$ (respectively, the boundaries of $J_{i+1},\dots,J_n$), such that $J_i$ is consistent
with $J_0,\dots,J_{i-1}$ (or $J_{i+1},\dots,J_n$) if and only if $X_i\le \omega(J_i)\le Y_i$.

\begin{lem}\label{L4.1} Let $i$ be an integer with $0\le i\le n$ and let $J_i$ be an ideal of 
$U$. 

\begin{enumerate}
\item[(i)] Let $J_0,\dots,J_{i-1}$ be a forward consistent sequence of ideals of $U$. Then
$J_i$ is consistent with $J_0,\dots,J_{i-1}$ if and only if 
\begin{equation}\label{JiJj}
\bigl[J_j+D-(i-j)(0,p)\bigr]\cap U\subset J_i,\quad 0\le j<i
\end{equation}
and 
\begin{equation}\label{i-ap-b}
\bigl[J_i+D+(a,0)\bigr]\cap U\subset J_{i-ap-b},
\end{equation}
\begin{equation}\label{i-ap-b1}
\bigl[J_i+D+(a+1,-p^2+pb)\bigr]\cap U\subset J_{i-ap-b}
\end{equation}
for all $a,b\in{\Bbb Z}$ with $a\ge 0$, $0\le b\le p-1$, $ap+b\le i$.

\item[(ii)] Let $J_{i+1},\dots,J_n$ be a backward consistent sequence of ideals of $U$. Then
$J_i$ is consistent with $J_{i+1},\dots,J_n$ if and only if 
\begin{equation}\label{JiJj1}
\bigl[J_i+D-(j-i)(0,p)\bigr]\cap U\subset J_j,\quad i<j\le n
\end{equation}
and 
\begin{equation}\label{i+ap+b}
\bigl[J_{i+ap+b}+D+(a,0)\bigr]\cap U\subset J_i,
\end{equation}
\begin{equation}\label{i+ap+b1}
\bigl[J_{i+ap+b}+D+(a+1,-p^2+pb)\bigr]\cap U\subset J_i
\end{equation}
for all $a,b\in{\Bbb Z}$ with $a\ge 0$, $0\le b\le p-1$, $i+ap+b\le n$.
\end{enumerate}
\end{lem}

\begin{proof}
(i) By Lemma~\ref{L2.1} (ii), $J_i$ is consistent with $J_0,\dots,J_{i-1}$
if and only if for every $0\le j<i$, 
\begin{equation}\label{PfL4.1-1}
\bigl(J_j\times\{j\}+\Delta\bigr)\cap\bigl(U\times\{i\} \bigr)
\subset J_i\times\{i\}
\end{equation}
and
\begin{equation}\label{PfL4.1-2}
\bigl(J_i\times\{i\}+\Delta\bigr)\cap\bigl(U\times\{j\} \bigr)
\subset J_j\times\{j\}.
\end{equation}
However, by (\ref{x2y2z2}), we see that (\ref{PfL4.1-1}) is equivalent to (\ref{JiJj})
and (\ref{PfL4.1-2}) is equivalent to 
\begin{equation}\label{PfL4.1-3}
\bigl[J_i+D+(i-j)(\frac 1p,0)\bigr]\cap U\subset J_j.
\end{equation}
By Lemma~\ref{L3.1}, (\ref{PfL4.1-3}) is equivalent to (\ref{i-ap-b})
and (\ref{i-ap-b1}).

The proof of (ii) is essentially the same.
\end{proof}

\begin{lem}\label{L4.2}
Let $J$ and $K$ be ideals of $U$ with boundaries $W$ and $Z$ respectively. Let $a\ge 0$ and $b\ge 0$ be integers and let $\bar K$ be the largest ideal of $[0,a+n]\times[-b,n]$ such that
$\bar K\cap U=K$. Then the following conditions are equivalent.
\begin{enumerate}
\item[(i)] 
\begin{equation}\label{J+D}
\bigl[J+D+(a,-b)\bigr]\cap U\subset K.
\end{equation}

\item[(ii)] 
\begin{equation}\label{J+a-b}
J+(a,-b)\subset \bar K\cap\bigl((a,-b)+U\bigr).
\end{equation}

\item[(iii)] 
\begin{equation}\label{Kbar}
\bigl[J+D+(a,-b)\bigr]\cap\bigl([0,a+n]\times[-b,n]\bigr)\subset \bar K.
\end{equation}

\item[(iv)] 
\begin{equation}\label{W_under}
\bigl[\underline{(W+(a,-b))}_{[0,a+n]\times[-b,n]}\bigr]\bigm|_U\le Z.
\end{equation}

\item[(v)] 
\begin{equation}\label{W}
W+(a,-b)\le\bigl[\overline Z_{[0,a+n]\times[-b,n]}\bigr]\bigm|_{[a,a+n]\times[-b,-b+n]}.
\end{equation}

\item[(vi)] 
\begin{equation}\label{W-Z}
\underline{(W+(a,-b))}_{[0,a+n]\times[-b,n]}
\le\overline Z_{[0,a+n]\times[-b,n]}.
\end{equation}

\end{enumerate}
\end{lem}

\begin{proof}
Condition (iv) is a restatement of (i) in terms of boundaries. In fact, 
$\bigl[\underline{(W+(a,-b))}_{[0,a+n]\times[-b,n]}\bigr]\bigm|_U$ is the boundary of $\bigl[J+D+(a,-b)\bigr]\cap U$. In the same way, (ii) $\Leftrightarrow$ (v) and (iii) $\Leftrightarrow$ (vi). Condition (vi) follows from (iv) through the operation $\overline{(\ )}_{[0,a+n]\times[-b,n]}$;  condition (iv) follows from (vi) through the operation $(\ )|_U$. 
Similarly, (v) $\Leftrightarrow$ (vi) through operations $\underline{(\ )}_{[0,a+n]\times[-b,n]}$ and $(\ )|_{[a,a+n]\times[-b,-b+n]}$.
\end{proof}

%%%%%%%%%%%% Figure  %%%%%%%%%%%%%%%%%%%%%%%%%%%%%%%
\setlength{\unitlength}{3.5mm}
\[
\begin{picture}(18,17)
\put(0,5){\vector(1,0){18}}
\put(2,0){\vector(0,1){17}}
\put(2,1){\line(1,0){14}}
\put(2,15){\line(1,0){14}}
\put(6,11){\line(1,0){10}}
\put(16,1){\line(0,1){14}}
\put(12,5){\line(0,1){10}}
\put(6,1){\line(0,1){10}}

\put(2,13){\line(1,0){3}}
\put(5,13){\line(0,-1){3}}
\put(5,10){\line(1,0){3}}
\put(8,10){\line(0,-1){2}}
\put(8,8){\line(1,0){2}}
\put(10,8){\line(0,-1){1}}
\put(10,7){\line(1,0){3}}
\put(13,7){\line(0,-1){1}}
\put(13,6){\line(1,0){3}}

\put(2,10){\line(1,0){2}}
\put(4,10){\line(0,-1){1}}
\put(4,9){\line(1,0){3}}
\put(7,9){\line(0,-1){2}}
\put(7,7){\line(1,0){2}}
\put(9,7){\line(0,-1){1}}
\put(9,6){\line(1,0){1}}
\put(10,6){\line(0,-1){2}}
\put(10,4){\line(1,0){2}}
\put(12,4){\line(0,-1){1}}
\put(12,3){\line(1,0){1}}
\put(13,3){\line(0,-1){1}}
\put(13,2){\line(1,0){1}}
\put(14,2){\line(0,-1){1}}

\put(1.9,4.9){\makebox(0,0)[tr]{$\scriptstyle 0$}}
\put(5.9,4.9){\makebox(0,0)[tr]{$\scriptstyle a$}}
\put(12,4.9){\makebox(0,0)[t]{$\scriptstyle n$}}
\put(16.1,4.9){\makebox(0,0)[tl]{$\scriptstyle a+n$}}
\put(1.9,1){\makebox(0,0)[r]{$\scriptstyle -b$}}
\put(1.9,11){\makebox(0,0)[r]{$\scriptstyle -b+n$}}
\put(1.9,15){\makebox(0,0)[r]{$\scriptstyle n$}}

\put(2,13){\makebox(0,0){$\scriptstyle \circ$}}
\put(5,13){\makebox(0,0){$\scriptstyle \circ$}}
\put(5,10){\makebox(0,0){$\scriptstyle \circ$}}
\put(8,10){\makebox(0,0){$\scriptstyle \circ$}}
\put(8,8){\makebox(0,0){$\scriptstyle \circ$}}
\put(10,8){\makebox(0,0){$\scriptstyle \circ$}}
\put(10,7){\makebox(0,0){$\scriptstyle \circ$}}
\put(12,7){\makebox(0,0){$\scriptstyle \circ$}}
\put(13,7){\makebox(0,0){$\scriptstyle \circ$}}
\put(13,6){\makebox(0,0){$\scriptstyle \circ$}}
\put(16,6){\makebox(0,0){$\scriptstyle \circ$}}

\put(2,10){\makebox(0,0){$\scriptstyle \bullet$}}
\put(4,10){\makebox(0,0){$\scriptstyle \bullet$}}
\put(4,9){\makebox(0,0){$\scriptstyle \bullet$}}
\put(6,9){\makebox(0,0){$\scriptstyle \bullet$}}
\put(7,9){\makebox(0,0){$\scriptstyle \bullet$}}
\put(7,7){\makebox(0,0){$\scriptstyle \bullet$}}
\put(9,7){\makebox(0,0){$\scriptstyle \bullet$}}
\put(9,6){\makebox(0,0){$\scriptstyle \bullet$}}
\put(10,6){\makebox(0,0){$\scriptstyle \bullet$}}
\put(10,4){\makebox(0,0){$\scriptstyle \bullet$}}
\put(12,4){\makebox(0,0){$\scriptstyle \bullet$}}
\put(12,3){\makebox(0,0){$\scriptstyle \bullet$}}
\put(13,3){\makebox(0,0){$\scriptstyle \bullet$}}
\put(13,2){\makebox(0,0){$\scriptstyle \bullet$}}
\put(14,2){\makebox(0,0){$\scriptstyle \bullet$}}
\put(14,1){\makebox(0,0){$\scriptstyle \bullet$}}

\put(1,-2){\makebox(0,0){$\scriptstyle \circ$}}
\put(2,-2){\makebox(0,0){$\scriptstyle \circ$}}
\put(1,-2){\line(1,0){1}}
\put(3,-2){\makebox(0,0)[l]{$\scriptstyle :\ \omega(\bar K)$}}

\put(1,-3){\makebox(0,0){$\scriptstyle \bullet$}}
\put(2,-3){\makebox(0,0){$\scriptstyle \bullet$}}
\put(1,-3){\line(1,0){1}}
\put(3,-3){\makebox(0,0)[l]{$\scriptstyle :\ \omega\bigl([J+D+(a,-b)]\cap([0,a+n]\times[-b,n])\bigr)$}}

\end{picture}
\]
\vskip1cm
\[
\text{Figure 10. Illustration of Lemma~\ref{L4.2}}
\]
%%%%%%%%%%%%%%%% End Figure %%%%%%%%%%%%%%%%%%%%%%%%%%%%%%%%%%%%%%
\vskip3mm

\begin{lem}\label{L4.3}
Let $J, K, L$ be ideals of $U$ and let $b,c$ be positive integers. If
\begin{equation}\label{Eq4.9}
\bigl[J+D+(0,-b)\bigr]\cap U\subset K
\end{equation}
and
\begin{equation}\label{Eq4.10}
\bigl[K+D+(0,-c)\bigr]\cap U\subset L,
\end{equation}
then
\begin{equation}\label{Eq4.11}
\bigl[J+D+(0,-b-c)\bigr]\cap U\subset L
\end{equation}
\end{lem}

\begin{proof}
Let $\bar K$ be the largest ideal of $[0,n]\times[-b-c,-c+n]$ such that 
\begin{equation}\label{Eq4.12}
\bar K\cap\bigl([0,n]\times[-c,-c+n]\bigr)=K+(0,-c).
\end{equation}
Then by (\ref{Eq4.9}) and Lemma~\ref{L4.2}, 
\begin{equation}\label{Eq4.13}
J+(0,-b-c)\subset\bar K\cap\bigl([0,n]\times[-b-c,-b-c+n]\bigr).
\end{equation}
Let $\bar L$ be the largest ideal of $[0,n]\times[-b-c,n]$ such that $\bar L\cap U=L$.
Put $\tilde L=\bar L\cap\bigl([0,n]\times[-c,n]\bigr)$. Clearly, $\tilde L$ is the largest ideal
of $[0,n]\times[-c,n]$ such that $\tilde L\cap U=L$. Thus by (\ref{Eq4.10}) and Lemma~\ref{L4.2}, 
\begin{equation}\label{Eq4.14}
K+(0,-c)\subset\tilde L\cap\bigl([0,n]\times[-c,-c+n]\bigr)=\bar L\cap\bigl([0,n]\times[-c,-c+n]\bigr).
\end{equation}

Let $\hat L$ be the largest ideal of $[0,n]\times[-b-c,-c+n]$ such that 
\begin{equation}\label{Eq4.15}
\hat L\cap\bigl([0,n]\times[-c,-c+n]\bigr)=\bar L\cap\bigl([0,n]\times[-c,-c+n]\bigr).
\end{equation}
We claim that
\begin{equation}\label{Eq4.16}
\hat L=\bar L\cap\bigl([0,n]\times[-b-c,-c+n]\bigr).
\end{equation}
In fact, $\omega(\bar L)$ is the highest extension of $\omega\bigl(\bar L\cap([0,n]\times[-c,n])\bigr)$;
$\omega(\hat L)$ is the highest extension of $\omega\bigl(\bar L\cap([0,n]\times[-c,-c+n])\bigr)$.
Since both extensions follow the same rules (described in the last paragraph of Section 3), the new steps (in 
$[0,n]\times[-b-c,-c]$) 
in both extensions are identical. Therefore (\ref{Eq4.16}) is proved.

Note that $\bar K$ is an ideal of $[0,n]\times[-b-c,-c+n]$ and that by (\ref{Eq4.12}) and (\ref{Eq4.14}),
\[
\bar K\cap\bigl([0,n]\times[-c,-c+n]\bigr)\subset\bar L\cap\bigl([0,n]\times[-c,-c+n]\bigr).
\]
By the maximality of $\hat L$, we have $\bar K\subset \hat L$.
However, (\ref{Eq4.16}) implies that $\hat L\subset\bar L$. Thus we have $\bar K\subset \bar L$.
Hence by (\ref{Eq4.13}), we have
\[
J+(0,-b-c)\subset\bar L\cap\bigl([0,n]\times[-b-c,-b-c+n]\bigr),
\]
which, by Lemma~\ref{L4.2}, implies (\ref{Eq4.11}).
\end{proof}

%%%%%%%%%%%% Figure 11 %%%%%%%%%%%%%%%%%%%%%%%%%%%%%%%
\setlength{\unitlength}{4mm}
\[
\begin{picture}(24,19)
\put(0,7){\vector(1,0){14}}
\put(2,0){\vector(0,1){19}}

\put(2,1){\line(1,0){10}}
\put(2,3){\line(1,0){10}}
\put(2,11){\line(1,0){10}}
\put(2,13){\line(1,0){10}}
\put(2,17){\line(1,0){10}}
\put(12,1){\line(0,1){16}}

\put(2,15){\line(1,0){2}}
\put(4,15){\line(0,-1){1}}
\put(4,14){\line(1,0){2}}
\put(6,14){\line(0,-1){3}}
\put(8,11){\line(0,-1){3}}
\put(8,8){\line(1,0){1}}
\put(9,8){\line(0,-1){2}}
\put(9,6){\line(1,0){3}}

\put(3,13){\line(0,-1){1}}
\put(3,12){\line(1,0){2}}
\put(5,12){\line(0,-1){2}}
\put(5,10){\line(1,0){2}}
\put(7,10){\line(0,-1){4}}
\put(7,6){\line(1,0){1}}
\put(8,6){\line(0,-1){3}}
\put(8,3){\line(1,0){1}}
\put(9,3){\line(0,-1){1}}
\put(9,2){\line(1,0){3}}

\put(4,11){\line(0,-1){2}}
\put(4,9){\line(1,0){1}}
\put(5,9){\line(0,-1){1}}
\put(5,8){\line(1,0){1}}
\put(6,8){\line(0,-1){2}}
\put(6,6){\line(1,0){1}}
\put(7,6){\line(0,-1){4}}
\put(7,2){\line(1,0){1}}
\put(8,2){\line(0,-1){1}}

\put(2,15){\makebox(0,0){$\scriptstyle \times$}}
\put(4,15){\makebox(0,0){$\scriptstyle \times$}}
\put(4,14){\makebox(0,0){$\scriptstyle \times$}}
\put(6,14){\makebox(0,0){$\scriptstyle \times$}}
\put(6,11){\makebox(0,0){$\scriptstyle \times$}}
\put(8,11){\makebox(0,0){$\scriptstyle \times$}}
\put(8,8){\makebox(0,0){$\scriptstyle \times$}}
\put(9,8){\makebox(0,0){$\scriptstyle \times$}}
\put(9,7){\makebox(0,0){$\scriptstyle \times$}}
\put(9,6){\makebox(0,0){$\scriptstyle \times$}}
\put(12,6){\makebox(0,0){$\scriptstyle \times$}}

\put(3,13){\makebox(0,0){$\scriptstyle \circ$}}
\put(3,12){\makebox(0,0){$\scriptstyle \circ$}}
\put(5,12){\makebox(0,0){$\scriptstyle \circ$}}
\put(5,10){\makebox(0,0){$\scriptstyle \circ$}}
\put(7,10){\makebox(0,0){$\scriptstyle \circ$}}
\put(7,6){\makebox(0,0){{\large $ \circ$}}}
\put(8,6){\makebox(0,0){$\scriptstyle \circ$}}
\put(8,3){\makebox(0,0){$\scriptstyle \circ$}}
\put(9,3){\makebox(0,0){$\scriptstyle \circ$}}
\put(9,2){\makebox(0,0){$\scriptstyle \circ$}}
\put(12,2){\makebox(0,0){$\scriptstyle \circ$}}

\put(4,11){\makebox(0,0){$\scriptstyle \bullet$}}
\put(4,9){\makebox(0,0){$\scriptstyle \bullet$}}
\put(5,9){\makebox(0,0){$\scriptstyle \bullet$}}
\put(5,8){\makebox(0,0){$\scriptstyle \bullet$}}
\put(6,8){\makebox(0,0){$\scriptstyle \bullet$}}
\put(6,6){\makebox(0,0){$\scriptstyle \bullet$}}
\put(7,6){\makebox(0,0){$\scriptscriptstyle \bullet$}}
\put(7,2){\makebox(0,0){$\scriptstyle \bullet$}}
\put(8,2){\makebox(0,0){$\scriptstyle \bullet$}}
\put(8,1){\makebox(0,0){$\scriptstyle \bullet$}}
\put(10,1){\makebox(0,0){$\scriptstyle \bullet$}}

\put(14,16){\makebox(0,0){$\scriptstyle \bullet$}}
\put(15,16){\makebox(0,0){$\scriptstyle \bullet$}}
\put(14,16){\line(1,0){1}}
\put(16,16){\makebox(0,0)[l]{$\scriptstyle :\ \omega(J+(0,-b-c))$}}

\put(14,15){\makebox(0,0){$\scriptstyle \circ$}}
\put(15,15){\makebox(0,0){$\scriptstyle \circ$}}
\put(14,15){\line(1,0){1}}
\put(16,15){\makebox(0,0)[l]{$\scriptstyle :\ \omega(\bar K)$}}

\put(14,14){\makebox(0,0){$\scriptstyle \times$}}
\put(15,14){\makebox(0,0){$\scriptstyle \times$}}
\put(14,14){\line(1,0){1}}
\put(16,14){\makebox(0,0)[l]{$\scriptstyle :\ \omega(\bar L)$}}

\put(1.9,1){\makebox(0,0)[r]{$\scriptstyle -b-c$}}
\put(1.9,3){\makebox(0,0)[r]{$\scriptstyle -c$}}
\put(1.9,6.9){\makebox(0,0)[tr]{$\scriptstyle 0$}}
\put(1.9,11){\makebox(0,0)[r]{$\scriptstyle -b-c+n$}}
\put(1.9,13){\makebox(0,0)[r]{$\scriptstyle -c+n$}}
\put(1.9,17){\makebox(0,0)[r]{$\scriptstyle n$}}
\put(12.1,7.1){\makebox(0,0)[bl]{$\scriptstyle n$}}

\end{picture}
\]
\[
\text{Figure 11. Illustration of Lemma~\ref{L4.3}}
\]
%%%%%%%%%%%%%%%% End Figure %%%%%%%%%%%%%%%%%%%%%%%%%%%%%%%%%%%%%%
\vskip3mm

\begin{lem}\label{L4.4}
Let $J, K, L$ be ideals of $U$ and let $a, b,c,d$ be nonnegative integers. Assume that
\begin{equation}\label{Eq4.20}
\bigl[J+D+(a,-b)\bigr]\cap U\subset K
\end{equation}
and
\begin{equation}\label{Eq4.21}
\bigl[K+D+(c,-d)\bigr]\cap U\subset L.
\end{equation}
Furthermore, assume that $K\ne\emptyset$, $K\ne U$ and that $\omega(K)$ is not a single horizontal step.
(Note that when $n\ge p$, $\omega(K)$ is never a single horizontal step.) Then we have
\begin{equation}\label{Eq4.22}
\bigl[J+D+(a+c,-b-d)\bigr]\cap U\subset L.
\end{equation}
\end{lem}

\begin{proof} 
Let $\bar L$ be the largest ideal of $[0,a+c+n]\times[-b-d,n]$ such that $\bar L\cap U= L$ and let $\bar K$ be the largest ideal of $[c,a+c+n]\times[-b-d,-d+n]$ such that
\[
\bar K\cap\bigl([c,c+n]\times[-d,-d+n]\bigr)= K+(c,-d).
\]
By (\ref{Eq4.20}) and Lemma~\ref{L4.2},
\begin{equation}\label{Eq4.23}
J+(a+c,-b-d)\subset\bar K\cap\bigl([a+c,a+c+n]\times[-b-d,-b-d+n]\bigr).
\end{equation}
Put $\tilde L=\bar L\cap\bigl([0,c+n]\times[-d,n]\bigr)$. Clearly, $\tilde L$ is the largest ideal of $[0,c+n]\times[-d,n]$ such that 
\begin{equation}\label{Ltilde}
\tilde L\cap U= L.
\end{equation}
By (\ref{Eq4.21}) and Lemma~\ref{L4.2},
\begin{equation}\label{Eq4.24}
K+(c,-d)\subset\tilde L\cap\bigl([c,c+n]\times[-d,-d+n]\bigr).
\end{equation}

Let $\hat K$ be the smallest ideal of $[0,c+n]\times[-d,n]$ such that
\[
\hat K\cap\bigl([c,c+n]\times[-d,-d+n]\bigr)=K+(c,-d).
\]
By (\ref{Eq4.24}) and the minimality of $\hat K$, we have
\begin{equation}\label{Eq4.25}
\hat K\subset \tilde L.
\end{equation}
The walk $\omega(\hat K)$ is an extension of $\omega(K+(c,-d))$ to the upper left; the walk $\omega(\bar K)$ is an extension of $\omega(K+(c,-d))$ to the lower right. (See Figure 12.) Since $K+(c,-d)\ne\emptyset$,
$K+(c,-d)\ne[c,c+n]\times[-d,-d+n]$, and since $\omega\bigl(K+(c,-d)\bigr)$
is not a single horizontal step, the union (in the obvious sense) of the
walks $\omega(\hat K)$ and $\omega(\bar K)$ is a walk in $[0,a+c+n]\times[-b-d,-b-d+n]$. Denote this walk by $W$. Note that
\[
\begin{array}{rll}
\iota(W)\cap U\kern-2mm&=\hat K\cap U\cr
&\subset\tilde L\cap U&\text{(by (\ref{Eq4.25}))}\cr
&= L&\text{(by (\ref{Ltilde}))}.\cr
\end{array}
\]
Thus by the maximality of $\bar L$, we have $\iota(W)\subset\bar L$. Hence
\[
\begin{split}
&J+(a+c,-b-d)\cr
\subset\,&\bar K\cap\bigl([a+c,\,a+c+n]\times[-b-d,-b-d+n]\bigr)\quad\quad\text{(by (\ref{Eq4.23}))}\cr
=\,&\iota(W)\cap\bigl([a+c,\,a+c+n]\times[-b-d,-b-d+n]\bigr)\cr
\subset\,& \bar L\cap \bigl([a+c,\,a+c+n]\times[-b-d,-b-d+n]\bigr).\cr
\end{split}
\]
By Lemma~\ref{L4.2}, (\ref{Eq4.22}) follows.
\end{proof}

\vskip3mm

\noindent{\bf Remark}. If $K=\emptyset$ or $K=U$, or $\omega(K)$ is a single horizontal step, the conclusion in
Lemma~\ref{L4.4} may not be true. Counterexamples are given in Figures
13 -- 15.

%%%%%%%%%%%% Figure  %%%%%%%%%%%%%%%%%%%%%%%%%%%%%%%
\setlength{\unitlength}{4mm}
\[
\begin{picture}(22,18)
\put(0,6){\vector(1,0){22}}
\put(2,0){\vector(0,1){18}}

\put(1.8,16){\makebox(0,0)[r]{$\scriptstyle n$}}
\put(1.8,13){\makebox(0,0)[r]{$\scriptstyle -d+n$}}
\put(1.8,3){\makebox(0,0)[r]{$\scriptstyle -d$}}
\put(1.8,1){\makebox(0,0)[r]{$\scriptstyle -b-d$}}
\put(1.8,11){\makebox(0,0)[r]{$\scriptstyle -b-d+n$}}
\put(1.8,5.8){\makebox(0,0)[tr]{$\scriptstyle 0$}}
\put(4.2,5.8){\makebox(0,0)[tl]{$\scriptstyle c$}}
\put(12,5.8){\makebox(0,0)[t]{$\scriptstyle n$}}
\put(14.1,5.9){\makebox(0,0)[tl]{$\scriptstyle c+n$}}
\put(20.1,6.1){\makebox(0,0)[bl]{$\scriptstyle a+c+n$}}
\put(2,1){\line(1,0){0.1}}
\put(2,3){\line(1,0){0.1}}
\put(2,13){\line(1,0){0.1}}

\put(2,16){\line(1,0){10}}
\put(4,13){\line(1,0){10}}
\put(4,3){\line(1,0){10}}
\put(10,11){\line(1,0){10}}
\put(10,1){\line(1,0){10}}
\put(12,6){\line(0,1){10}}
\put(4,3){\line(0,1){10}}
\put(14,3){\line(0,1){10}}
\put(10,1){\line(0,1){10}}
\put(20,1){\line(0,1){10}}

\put(2,15){\line(1,0){1}}
\put(3,15){\line(0,-1){1}}
\put(3,14){\line(1,0){2}}
\put(5,14){\line(0,-1){2}}
\put(5,12){\line(1,0){2}}
\put(7,12){\line(0,-1){2}}
\put(10,9){\line(1,0){3}}
\put(13,9){\line(0,-1){1}}
\put(13,8){\line(1,0){3}}
\put(16,8){\line(0,-1){1}}
\put(16,7){\line(1,0){3}}
\put(19,7){\line(0,-1){1}}

\put(2,11){\line(1,0){3}}
\put(5,11){\line(0,-1){1}}
\put(5,10){\line(1,0){5}}
\put(8,10){\line(0,-1){2}}
\put(8,8){\line(1,0){3}}
\put(11,8){\line(0,-1){1}}
\put(11,7){\line(1,0){2}}
\put(13,7){\line(0,-1){2}}
\put(13,5){\line(1,0){3}}
\put(16,5){\line(0,-1){1}}
\put(16,4){\line(1,0){3}}
\put(19,4){\line(0,-1){1}}
\put(19,3){\line(1,0){1}}

\put(10,5){\line(1,0){2}}
\put(12,5){\line(0,-1){1}}
\put(12,4){\line(1,0){3}}
\put(15,4){\line(0,-1){1}}
\put(15,3){\line(1,0){3}}
\put(18,3){\line(0,-1){2}}

\put(2,12){\line(1,-4){3.2}}
\put(14,5){\line(-3,-2){8.8}}

\put(4,11){\line(1,-2){6.4}}
\put(20,3){\line(-2,-1){9.6}}

\put(5,-1){\makebox(0,0)[t]{$\scriptstyle \omega(\hat K)$}}
\put(10.2,-2){\makebox(0,0)[t]{$\scriptstyle \omega(\bar K)$}}

\put(2,15){\makebox(0,0){$\scriptstyle \times$}}
\put(3,15){\makebox(0,0){$\scriptstyle \times$}}
\put(3,14){\makebox(0,0){$\scriptstyle \times$}}
\put(5,14){\makebox(0,0){$\scriptstyle \times$}}
\put(5,12){\makebox(0,0){$\scriptstyle \times$}}
\put(7,12){\makebox(0,0){$\scriptstyle \times$}}
\put(7,10){\makebox(0,0){$\scriptstyle \times$}}
\put(10,10){\makebox(0,0){$\scriptstyle \times$}}
\put(10,9){\makebox(0,0){$\scriptstyle \times$}}
\put(13,9){\makebox(0,0){$\scriptstyle \times$}}
\put(13,8){\makebox(0,0){$\scriptstyle \times$}}
\put(14,8){\makebox(0,0){$\scriptstyle \times$}}
\put(16,8){\makebox(0,0){$\scriptstyle \times$}}
\put(16,7){\makebox(0,0){$\scriptstyle \times$}}
\put(19,7){\makebox(0,0){$\scriptstyle \times$}}
\put(19,6){\makebox(0,0){$\scriptstyle \times$}}
\put(20,6){\makebox(0,0){$\scriptstyle \times$}}

\put(2,12){\makebox(0,0){$\scriptstyle \bullet$}}
\put(2,11){\makebox(0,0){$\scriptstyle \bullet$}}
\put(4,11){\makebox(0,0){$\scriptstyle \bullet$}}
\put(5,11){\makebox(0,0){$\scriptstyle \bullet$}}
\put(5,10){\makebox(0,0){$\scriptstyle \bullet$}}
\put(8,10){\makebox(0,0){$\scriptstyle \bullet$}}
\put(8,8){\makebox(0,0){$\scriptstyle \bullet$}}
\put(11,8){\makebox(0,0){$\scriptstyle \bullet$}}
\put(11,7){\makebox(0,0){$\scriptstyle \bullet$}}
\put(12,7){\makebox(0,0){$\scriptstyle \bullet$}}
\put(13,7){\makebox(0,0){$\scriptstyle \bullet$}}
\put(13,5){\makebox(0,0){$\scriptstyle \bullet$}}
\put(14,5){\makebox(0,0){$\scriptstyle \bullet$}}
\put(16,5){\makebox(0,0){$\scriptstyle \bullet$}}
\put(16,4){\makebox(0,0){$\scriptstyle \bullet$}}
\put(19,4){\makebox(0,0){$\scriptstyle \bullet$}}
\put(19,3){\makebox(0,0){$\scriptstyle \bullet$}}
\put(20,3){\makebox(0,0){$\scriptstyle \bullet$}}

\put(10,5){\makebox(0,0){$\scriptstyle \circ$}}
\put(12,5){\makebox(0,0){$\scriptstyle \circ$}}
\put(12,4){\makebox(0,0){$\scriptstyle \circ$}}
\put(15,4){\makebox(0,0){$\scriptstyle \circ$}}
\put(15,3){\makebox(0,0){$\scriptstyle \circ$}}
\put(18,3){\makebox(0,0){$\scriptstyle \circ$}}
\put(18,1){\makebox(0,0){$\scriptstyle \circ$}}

\put(6,-4){\makebox(0,0){$\scriptstyle \times$}}
\put(7,-4){\makebox(0,0){$\scriptstyle \times$}}
\put(6,-4){\line(1,0){1}}
\put(8,-4){\makebox(0,0)[l]{$\scriptstyle :\ \omega(\bar L)$}}
\put(6,-5){\makebox(0,0){$\scriptstyle \circ$}}
\put(7,-5){\makebox(0,0){$\scriptstyle \circ$}}
\put(6,-5){\line(1,0){1}}
\put(8,-5){\makebox(0,0)[l]{$\scriptstyle :\ \omega(J+(a+b,-b-c))$}}

\end{picture}
\]
\vskip1.7cm
\[
\text{Figure 12. Proof of Lemma~\ref{L4.4}}
\]
%%%%%%%%%%%%%%%% End Figure %%%%%%%%%%%%%%%%%%%%%%%%%%%%%%%%%%%%%%

%%%%%%%%%%%% Figure  %%%%%%%%%%%%%%%%%%%%%%%%%%%%%%%
\vskip5mm
\setlength{\unitlength}{4mm}
\[
\begin{picture}(16,11)
\put(0,4){\line(1,0){7}}
\put(0,11){\line(1,0){7}}
\put(8,3){\line(1,0){7}}
\put(8,10){\line(1,0){7}}
\put(9,7){\line(1,0){7}}
\put(9,0){\line(1,0){7}}
\put(0,4){\line(0,1){7}}
\put(7,4){\line(0,1){7}}
\put(8,3){\line(0,1){7}}
\put(15,3){\line(0,1){7}}
\put(9,0){\line(0,1){7}}
\put(16,0){\line(0,1){7}}

\put(0,6){\line(1,0){1}}
\put(1,6){\line(0,-1){1}}
\put(1,5){\line(1,0){2}}
\put(3,5){\line(0,-1){1}}
\put(5,4){\line(0,-1){1}}
\put(5,3){\line(1,0){2}}
\put(7,3){\line(0,-1){1}}
\put(7,2){\line(1,0){2}}
\put(9,1){\line(1,0){2}}
\put(11,1){\line(0,-1){1}}

\multiput(0,10)(0,-1){4}{\line(1,0){0.1}}
\multiput(2,4)(2,0){3}{\line(0,1){0.1}}
\multiput(8,9)(0,-1){6}{\line(1,0){0.1}}
\multiput(9,6)(0,-1){6}{\line(1,0){0.1}}
\multiput(10,3)(1,0){5}{\line(0,1){0.1}}
\put(10,0){\line(0,1){0.1}}
\multiput(12,0)(1,0){4}{\line(0,1){0.1}}

\put(0,6){\makebox(0,0){$\scriptstyle \circ$}}
\put(1,6){\makebox(0,0){$\scriptstyle \circ$}}
\put(1,5){\makebox(0,0){$\scriptstyle \circ$}}
\put(3,5){\makebox(0,0){$\scriptstyle \circ$}}
\put(3,4){\makebox(0,0){$\scriptstyle \circ$}}
\put(5,4){\makebox(0,0){$\scriptstyle \circ$}}
\put(5,3){\makebox(0,0){$\scriptstyle \circ$}}
\put(7,3){\makebox(0,0){$\scriptstyle \circ$}}
\put(7,2){\makebox(0,0){$\scriptstyle \circ$}}
\put(9,2){\makebox(0,0){$\scriptstyle \circ$}}
\put(9,1){\makebox(0,0){$\scriptstyle \circ$}}
\put(11,1){\makebox(0,0){$\scriptstyle \circ$}}
\put(11,0){\makebox(0,0){$\scriptstyle \circ$}}

\put(0,5){\makebox(0,0){$\scriptstyle \times$}}
\put(0,4){\makebox(0,0){$\scriptstyle \times$}}
\put(1,4){\makebox(0,0){$\scriptstyle \times$}}

\put(0,-1){\makebox(0,0){$\scriptstyle \circ$}}
\put(1,-1){\makebox(0,0){$\scriptstyle \circ$}}
\put(0,-1){\line(1,0){1}}
\put(2,-1){\makebox(0,0)[l]{$\scriptstyle :\ \omega{\textstyle(}[J+D+(a+c,-b-d)]\cap([0,a+c+n]\times[-b-d,n])
{\textstyle)}$}}
\put(0,-2){\makebox(0,0){$\scriptstyle \times$}}
\put(1,-2){\makebox(0,0){$\scriptstyle \times$}}
\put(0,-2){\line(1,0){1}}
\put(2,-2){\makebox(0,0)[l]{$\scriptstyle :\ \omega(L)$}}
\put(0,-3){\makebox(0,0)[l]{$\scriptstyle  p=2$}}

\end{picture}
\]
\vskip1.2cm
\[
\text{Figure 13. A counterexample of Lemma~\ref{L4.4}: $K=\emptyset$}
\]
%%%%%%%%%%%%%%%% End Figure %%%%%%%%%%%%%%%%%%%%%%%%%%%%%%%%%%%%%%

%%%%%%%%%%%% Figure  %%%%%%%%%%%%%%%%%%%%%%%%%%%%%%%
\vskip8mm
\setlength{\unitlength}{4mm}
\[
\begin{picture}(16,10)
\put(0,3){\line(1,0){7}}
\put(0,10){\line(1,0){7}}
\put(1,1){\line(1,0){7}}
\put(1,8){\line(1,0){7}}
\put(9,0){\line(1,0){7}}
\put(9,7){\line(1,0){7}}
\put(0,3){\line(0,1){7}}
\put(7,3){\line(0,1){7}}
\put(1,1){\line(0,1){7}}
\put(8,1){\line(0,1){7}}
\put(9,0){\line(0,1){7}}
\put(16,0){\line(0,1){7}}

\multiput(0,9)(0,-1){6}{\line(1,0){0.1}}
\multiput(1,7)(0,-1){6}{\line(1,0){0.1}}
\multiput(9,6)(0,-1){6}{\line(1,0){0.1}}
\multiput(1,3)(1,0){6}{\line(0,1){0.1}}
\multiput(2,1)(1,0){6}{\line(0,1){0.1}}
\multiput(10,0)(1,0){6}{\line(0,1){0.1}}

\put(4,10){\line(0,-1){1}}
\put(4,9){\line(1,0){3}}
\put(6,10){\line(0,-1){2}}
\put(8,10){\line(0,-1){1}}
\put(8,9){\line(1,0){2}}
\put(10,9){\line(0,-1){1}}
\put(10,8){\line(1,0){2}}
\put(12,8){\line(0,-1){1}}
\put(14,7){\line(0,-1){1}}
\put(14,6){\line(1,0){2}}

\put(8,10){\makebox(0,0){$\scriptstyle \circ$}}
\put(8,9){\makebox(0,0){$\scriptstyle \circ$}}
\put(10,9){\makebox(0,0){$\scriptstyle \circ$}}
\put(10,8){\makebox(0,0){$\scriptstyle \circ$}}
\put(12,8){\makebox(0,0){$\scriptstyle \circ$}}
\put(12,7){\makebox(0,0){$\scriptstyle \circ$}}
\put(14,7){\makebox(0,0){$\scriptstyle \circ$}}
\put(14,6){\makebox(0,0){$\scriptstyle \circ$}}
\put(16,6){\makebox(0,0){$\scriptstyle \circ$}}

\put(6,10){\makebox(0,0){$\scriptstyle \times$}}
\put(6,9){\makebox(0,0){$\scriptstyle \times$}}
\put(7,9){\makebox(0,0){$\scriptstyle \times$}}

\put(4,10){\makebox(0,0){$\scriptstyle \bullet$}}
\put(4,9){\makebox(0,0){$\scriptstyle \bullet$}}
\put(6,9){\makebox(0,0){$\scriptstyle \bullet$}}
\put(6,8){\makebox(0,0){$\scriptstyle \bullet$}}
\put(8,8){\makebox(0,0){$\scriptstyle \bullet$}}

\put(0,-1){\makebox(0,0){$\scriptstyle \circ$}}
\put(1,-1){\makebox(0,0){$\scriptstyle \circ$}}
\put(0,-1){\line(1,0){1}}
\put(2,-1){\makebox(0,0)[l]{$\scriptstyle :\ \omega{\textstyle(}[J+D+(a+c,-b-d)]\cap([0,a+c+n]\times[-b-d,n])
{\textstyle)}$}}
\put(0,-2){\makebox(0,0){$\scriptstyle \times$}}
\put(1,-2){\makebox(0,0){$\scriptstyle \times$}}
\put(0,-2){\line(1,0){1}}
\put(2,-2){\makebox(0,0)[l]{$\scriptstyle :\ \omega(L)$}}
\put(0,-3){\makebox(0,0){$\scriptstyle \bullet$}}
\put(1,-3){\makebox(0,0){$\scriptstyle \bullet$}}
\put(0,-3){\line(1,0){1}}
\put(2,-3){\makebox(0,0)[l]{$\scriptstyle :\ \omega{\textstyle(}[K+D+(c,-d)]\cap([0,c+n]\times[-d,n])
{\textstyle)}$}}
\put(0,-4){\makebox(0,0)[l]{$\scriptstyle  p=2$}}

\end{picture}
\]
\vskip1.5cm
\[
\text{Figure 14. A counterexample of Lemma~\ref{L4.4}: $K=U$}
\]
%%%%%%%%%%%%%%%% End Figure %%%%%%%%%%%%%%%%%%%%%%%%%%%%%%%%%%%%%%
\vfill\eject

%%%%%%%%%%%% Figure  %%%%%%%%%%%%%%%%%%%%%%%%%%%%%%%
\vskip5mm
\setlength{\unitlength}{6mm}
\[
\begin{picture}(9,6)
\put(0,2){\line(1,0){4}}
\put(0,6){\line(1,0){4}}
\put(3,1){\line(1,0){4}}
\put(3,5){\line(1,0){4}}
\put(5,0){\line(1,0){4}}
\put(5,4){\line(1,0){4}}
\put(0,2){\line(0,1){4}}
\put(4,2){\line(0,1){4}}
\put(3,1){\line(0,1){4}}
\put(7,1){\line(0,1){4}}
\put(5,0){\line(0,1){4}}
\put(9,0){\line(0,1){4}}

\put(0,5){\line(1,0){0.1}}
\put(3,4){\line(1,0){0.1}}
\put(5,2){\line(1,0){0.1}}
\multiput(1,2)(1,0){2}{\line(0,1){0.1}}
\multiput(6,0)(1,0){2}{\line(0,1){0.1}}
\multiput(4,1)(2,0){2}{\line(0,1){0.1}}

\put(0,4){\line(1,0){1}}
\put(1,4){\line(0,-1){1}}
\put(0,3){\line(1,0){8}}
\put(8,3){\line(0,-1){3}}

\put(0,4){\makebox(0,0){{\large $\circ$}}}
\put(1,4){\makebox(0,0){$\scriptstyle \circ$}}
\put(1,3){\makebox(0,0){$\scriptstyle \circ$}}
\put(5,3){\makebox(0,0){$\scriptstyle \circ$}}
\put(8,3){\makebox(0,0){$\scriptstyle \circ$}}
\put(8,0){\makebox(0,0){$\scriptstyle \circ$}}
\put(0,4){\makebox(0,0){$\scriptstyle \times$}}
\put(0,3){\makebox(0,0){$\scriptstyle \times$}}
\put(4,3){\makebox(0,0){$\scriptstyle \times$}}
\put(0,4){\makebox(0,0){$\scriptscriptstyle \bullet$}}
\put(0,3){\makebox(0,0){$\scriptstyle \bullet$}}
\put(5,3){\makebox(0,0){$\scriptstyle \bullet$}}
\put(7,3){\makebox(0,0){$\scriptstyle \bullet$}}

\put(0,-1){\makebox(0,0){$\scriptstyle \circ$}}
\put(1,-1){\makebox(0,0){$\scriptstyle \circ$}}
\put(0,-1){\line(1,0){1}}
\put(2,-1){\makebox(0,0)[l]{$\scriptstyle :\ \omega{\textstyle(}[J+D+(a+c,-b-d)]\cap([0,a+c+n]\times[-b-d,n])
{\textstyle)}$}}
\put(0,-2){\makebox(0,0){$\scriptstyle \times$}}
\put(1,-2){\makebox(0,0){$\scriptstyle \times$}}
\put(0,-2){\line(1,0){1}}
\put(2,-2){\makebox(0,0)[l]{$\scriptstyle :\ \omega(L)$}}
\put(0,-3){\makebox(0,0){$\scriptstyle \bullet$}}
\put(1,-3){\makebox(0,0){$\scriptstyle \bullet$}}
\put(0,-3){\line(1,0){1}}
\put(2,-3){\makebox(0,0)[l]{$\scriptstyle :\ \omega{\textstyle(}[K+D+(c,-d)]\cap([0,c+n]\times[-d,n])
{\textstyle)}$}}
\put(0,-4){\makebox(0,0)[l]{$\scriptstyle  p=7$}}

\end{picture}
\]
\vskip2cm
\[
\text{Figure 15. A counterexample of Lemma~\ref{L4.4}: $\omega(K)$ is a single horizontal step}
\]
%%%%%%%%%%%%%%%% End Figure %%%%%%%%%%%%%%%%%%%%%%%%%%%%%%%%%%%%%%
\vskip3mm

\begin{thm}\label{T4.5} Let $i$ be an integer with $0\le i\le n$.
Let $J_{i+1},\dots,J_n$ be a backward consistent sequence of ideals of $U$ and let $J_i$ be an ideal of $U$.
Then $J_i$ is consistent with $J_{i+1},\dots,J_n$ if and only if the following conditions are satisfied.

\begin{enumerate}
\item[(i)] $J_i\supset J_{i+1}$.

\item[(ii)] $\bigl[J_i+D-(0,p)\bigr]\cap U\subset J_{i+1}$.

\item[(iii)] $\bigl[J_{i+p}+D+(1,0)\bigr]\cap U\subset J_i$. (If $i+p>n$, this condition is null.)

\item[(iv)] Let $\alpha_i$ be the largest integer such that $1\le \alpha_i\le p-1$, $i+\alpha_i\le n$ and
$J_{i+\alpha_i}\ne\emptyset$. Then $\bigl[J_{i+\alpha_i}+D+(1,\, -p^2+p\alpha_i)\bigr]\cap U\subset J_i$.
(If such an $\alpha_i$ does not exist, this condition is null.)
\end{enumerate}
\end{thm}

\begin{proof}
First note that the theorem holds when $n<p$. In fact, in this case, 
since the partial order $\prec$ in ${\mathcal U}$ is the cartesian product of linear orders,
$J_i$ is consistent with $J_{i+1},\dots,J_n$ if and only if (i) is satisfied. Meanwhile, as one can easily see, (ii)
is automatically satisfied; (iii) is null;  (iv) is either automatically satisfied or
is null. Therefore we assume $n\ge p$.

We show that (\ref{JiJj1}) -- (\ref{i+ap+b1}) in Lemma~\ref{L4.1} together are equivalent to conditions (i) -- (iv) in Theorem~\ref{T4.5}

($\Rightarrow$) Condition (i) follows from (\ref{i+ap+b}) with $a=0$ and $b=1$ since $[J_{i+1}+D]\cap U=J_{i+1}$. Condition (ii) is a special
case of (\ref{JiJj1}). Conditions (iii) and (iv) are special cases of
(\ref{i+ap+b}) and (\ref{i+ap+b1}).

($\Leftarrow$) To prove (\ref{JiJj1}), let $i<j\le n$.
By (ii) and the fact that $J_{i+1},\dots,J_n$ is backward consistent, we have
\[
\begin{cases}
\bigl[J_i+D-(0,p)\bigr]\cap U\subset J_{i+1},\cr
\bigl[J_{i+1}+D-(j-i-1)(0,p)\bigr]\cap U\subset J_j,\cr
\end{cases}
\]
Thus by Lemma~\ref{L4.3},
\[
\bigl[J_i+D-(j-i)(0,p)\bigr]\cap U\subset J_j.
\]

To prove (\ref{i+ap+b}), let $a,b\in{\Bbb Z}$ with $a\ge 0$, $0\le b\le p-1$, $i+ap+b\le n$.
We may assume $a\ge 1$ since (\ref{i+ap+b}) becomes obvious when $a=0$.
By (iii) and the fact that $J_{i+1},\dots,J_n$ is backward consistent, we have
\begin{equation}\label{Ji+ap+D}
\begin{cases}
\bigl[J_{i+p}+D+(1,0)\bigr]\cap U\subset J_i,\cr
\bigl[J_{i+ap}+D+(a-1,0)\bigr]\cap U\subset J_{i+p}.\cr
\end{cases}
\end{equation}
We claim that
\begin{equation}\label{T4.4-1}
\bigl[J_{i+ap}+D+(a,0)\bigr]\cap U\subset J_i.
\end{equation}
In fact, if $J_{i+p}\ne\emptyset$ and  $J_{i+p}\ne U$, then (\ref{T4.4-1}) follows from (\ref{Ji+ap+D}) and Lemma~\ref{L4.4}. If $J_{i+p}=\emptyset$, then by (i), $J_{i+ap}=\emptyset$ since $J_{i+1},\dots,
J_n$ is backward consistent. Thus (\ref{T4.4-1}) holds. If $J_{i+p}=U$, by (i), we have 
$J_i=U$ and (\ref{T4.4-1}) also holds. Since $J_{i+ap+b}\subset J_{i+ap}$,
we have
\[
\bigl[J_{i+ap+b}+D+(a,0)\bigr]\cap U\subset\bigl[J_{i+ap}+D+(a,0)\bigr]\cap U\subset J_i.
\]

Finally, we prove (\ref{i+ap+b1}). We may assume $b\ge 1$, since if $b=0$,
we have
\begin{equation}\label{Ji+ap}
\begin{array}{rll}
&\bigl[J_{i+ap}+D+(a+1,-p^2)\bigr]\cap U\cr
\subset\kern-3mm& \bigl[J_{i+ap}+D+(a,0)\bigr]\cap U\ & \text{(since $(1,-p^2)\in D$)}\cr
\subset\kern-3mm&J_i &\text{(by (\ref{T4.4-1}))}.\cr
\end{array}
\end{equation} 
In (iv), if $\alpha_i$ does not exist or $\alpha_i<b$, then $J_{i+b}=\emptyset$. Hence $J_{i+ap+b}=\emptyset$ and we are done. So assume that $\alpha_i\ge b$. By (iv) and the fact that
$J_{i+1},\dots, J_n$ is backward consistent, we have
\begin{equation}\label{I1}
\begin{cases}
\bigl[J_{i+\alpha_i}+D+(1,-p^2+p\alpha_i)\bigr]\cap U\subset J_i,\cr
\bigl[J_{i+b}+D+(0,-p(\alpha_i-b))\bigr]\cap U\subset J_{i+\alpha_i},\cr
\bigl[J_{i+ap+b}+D+(a,0)\bigr]\cap U\subset J_{i+b}.\cr
\end{cases}
\end{equation}
If neither of $J_{i+\alpha_i}$ and $J_{i+b}$ is $\emptyset$ or $U$, by (\ref{I1}) and Lemma~\ref{L4.4}, we have
\[
\bigl[J_{i+ap+b}+D+(a+1,-p^2+pb)\bigr]\cap U\subset J_i,
\]
which is (\ref{i+ap+b1}).
If one of $J_{i+\alpha_i}$ and $J_{i+b}$ is $\emptyset$ or $U$, then 
$J_{i+b}=\emptyset$ or $J_{i+b}=U$ or $J_{i+\alpha_i}=U$ since $J_{i+\alpha_i}\ne\emptyset$. Thus
$J_{i+ap+b}=\emptyset$ or $J_i=U$ and (\ref{i+ap+b1}) also holds.       
\end{proof}

%%%%%%%%%%%%%%%% Theorem 4.6 %%%%%%%%%%%%%%%%%%%%%%%%%%%%%%%%

\begin{thm}\label{T4.6} Let $i$ be an integer with $0\le i\le n$.
Let $J_0,\dots,J_{i-1}$ be a forward consistent sequence of ideals of $U$ and let $J_i$ be an ideal of $U$.
Then $J_i$ is consistent with $J_0,\dots,J_{i-1}$ if and only if the following conditions are satisfied.

\begin{enumerate}
\item[(i)] $J_i\subset J_{i-1}$.

\item[(ii)] $J_i\supset\bigl[J_{i-1}+D-(0,p)\bigr]\cap U$.

\item[(iii)] $\bigl[J_i+D+(1,0)\bigr]\cap U\subset J_{i-p}$. (If $i-p<0$, this condition is null.)

\item[(iv)] Let $\beta_i$ be the largest integer such that $1\le \beta_i\le p-1$, $i-\beta_i\ge 0$ and
$J_{i-\beta_i}\ne U$. Then $\bigl[J_i+D+(1,\, -p^2+p\beta_i)\bigr]\cap U\subset J_{i-\beta_i}$.
(If such a $\beta_i$ does not exist, this condition is null.)
\end{enumerate}
\end{thm}

\begin{proof}
By the same reason in the proof of Theorem~\ref{T4.5}, we may assume $n\ge p$.

We show that (\ref{JiJj}) -- (\ref{i-ap-b1}) in Lemma~\ref{L4.1} together are equivalent to
conditions (i) -- (iv) in Theorem~\ref{T4.6}.  Since the proof is essentially the same as the proof of Theorem~\ref{T4.5}, we only show that (i) -- (iv) of Theorem~\ref{T4.6} imply (\ref{i-ap-b1}). 

Let $a, b$ be integers such that $a\ge0$, $0\le b\le p-1$ and $i-ap-b\ge 0$. By 
an argument similar to
(\ref{Ji+ap}), we may assume $b\ge 1$. In (iv), if $\beta_i$ does not exist or if $\beta_i<b$, then $J_{i-b}=U$. Hence $J_{i-ap-b}=U$ and  (\ref{i-ap-b1})
is obvious. So we may assume that $\beta_i\ge b$. By (iv)
and the fact that $J_0,\dots,J_{i-1}$ is forward consistent,  we have
\begin{equation}\label{I2}
\begin{cases}
\bigl[J_i+D+(1,-p^2+p\beta_i)\bigr]\cap U\subset J_{i-\beta_i},\cr
\bigl[J_{i-\beta_i}+D+(0,-p(\beta_i-b))\bigr]\cap U\subset J_{i-b},\cr
\bigl[J_{i-b}+D+(a,0)\bigr]\cap U\subset J_{i-ap-b}.\cr
\end{cases}
\end{equation}
If neither of $J_{i-b}$ and $J_{i-\beta_i}$ is $\emptyset$ or $U$, (\ref{i-ap-b1}) follows from 
(\ref{I2}) and Lemma~\ref{L4.4}.
If one of $J_{i-b}$ and $J_{i-\beta_i}$ is $\emptyset$ or $U$, then 
$J_{i-b}=\emptyset$ or $J_{i-b}=U$ or $J_{i-\beta_i}=\emptyset$ since
$J_{i-\beta_i}\ne U$. Thus
$J_{i-ap-b}=U$ or $J_i=\emptyset$;
in either case, (\ref{i-ap-b1}) holds.
\end{proof}

\begin{cor}\label{C4.7} {\em (Backward slicing)} Let $i$ be an integer with $0\le i\le n$.
Let $J_{i+1},\dots,J_n$ be a backward consistent sequence of ideals of $U$ and let $J_i$ be an ideal of $U$. Put $W_j=\omega(J_j)$, $i\le j\le n$.
Let
\begin{equation}\label{X}
\begin{split}
X_i=\,&W_{i+1}\vee\bigl[\underline{\bigl(W_{i+p}+(1,0)\bigr)}_{[0,n+1]\times[0,n]}\bigm|_U \bigr]\cr
&\vee\bigl[\underline{\bigl(W_{i+\alpha_i}+(1,-p^2+p\alpha_i)\bigr)}_{[0,n+1]\times[-p^2+p\alpha_i,n]}\bigm|_U\bigr],\cr
\end{split}
\end{equation}
where $\alpha_i$ is defined in Theorem~\ref{T4.5} {\rm (iv)}, and
\begin{equation}\label{Y}
Y_i=\overline{(W_{i+1})}_{[0,n]\times[-p,n]}\bigm|_{[0,n]\times[-p,-p+n]}+(0,p).
\end{equation}
Then $J_i$ is consistent with $J_{i+1},\dots,J_n$ if and only if 
\begin{equation}\label{XWY}
X_i\le W_i\le Y_i.
\end{equation}
\end{cor}

\noindent{\bf Note}. In (\ref{X}), if $i+p>n$, the walk after the first $\vee$ is not defined; if $\alpha_i$ does not exist, the walk after the second $\vee$ is not defined. Our convention, here and later, is that any undefined walk in a $\vee$ or $\wedge$ operation is ignored. 

\begin{proof} The corollary is a restatement of Theorem~\ref{T4.5}
in terms of boundaries. In fact, conditions (i), (iii) and (iv) of Theorem~\ref{T4.5}
are equivalent to
\[
\begin{cases}
W_i\ge W_{i+1},\cr
W_i\ge\underline{\bigl(W_{i+p}+(1,0)\bigr)}_{[0,n+1]\times[0,n]}|_U,\cr
W_i\ge\underline{\bigl(W_{i+\alpha_i}+(1,-p^2+p\alpha_i)\bigr)}_{[0,n+1]\times[-p^2+p\alpha_i,n]}|_U.\cr
\end{cases}
\]
By Lemma~\ref{L4.2}, condition (ii) of Theorem~\ref{T4.5} is equivalent to
\[
W_i\le \overline{(W_{i+1})}_{[0,n]\times[-p,n]}\bigm|_{[0,n]\times[-p,-p+n]}+(0,p).
\]
\end{proof}

\begin{cor}\label{C4.8} {\em (Forward slicing)} Let $i$ be an integer with $0\le i\le n$.
Let $J_0,\dots,J_{i-1}$ be a forward consistent sequence of ideals of $U$ and let $J_i$ be an ideal of $U$. Put $W_j=\omega(J_j)$, $0\le j\le i$.
Let
\begin{equation}\label{X'}
X_i'=\underline{\bigl(W_{i-1}-(0,p)\bigr)}_{[0,n]\times[-p,n]}\bigm|_U
\end{equation}
and
\begin{equation}\label{Y'}
\begin{split}
Y_i'=\,&W_{i-1}\wedge
\bigl[\overline{(W_{i-p})}_{[0,n+1]\times[0,n]}\bigm|_{[1,n+1]\times[0,n]}-(1,0)
\bigr]\wedge\cr
&\bigl[\overline{(W_{i-\beta_i})}_{[0,n+1]\times[-p^2+p\beta_i,n]}\bigm|_
{[1,n+1]\times[-p^2+p\beta_i,-p^2+p\beta_i+n]}-(1,-p^2+p\beta_i)\bigr],\cr
\end{split}
\end{equation}
where $\beta_i$ is defined in Theorem~\ref{T4.6} {\rm (iv)}.
Then $J_i$ is consistent with $J_0,\dots,J_{i-1}$ if and only if 
\begin{equation}\label{XWY'}
X_i'\le W_i\le Y_i'.
\end{equation}
\end{cor}

\begin{proof} The corollary is a restatement of Theorem~\ref{T4.6}
in terms of boundaries. By Lemma~\ref{L4.2}, conditions (i), (iii) and (iv) of Theorem~\ref{T4.6}
are equivalent to
\[
\begin{cases}
W_i\le W_{i-1},\cr
W_i\le\overline{(W_{i-p})}_{[0,n+1]\times[0,n]}\bigm|_{[1,n+1]\times[0,n]}-(1,0)
,\cr
W_i\le \overline{(W_{i-\beta_i})}_{[0,n+1]\times[-p^2+p\beta_i,n]}\bigm|_
{[1,n+1]\times[-p^2+p\beta_i,-p^2+p\beta_i+n]}-(1,-p^2+p\beta_i).\cr
\end{cases}
\]
Condition (ii) of Theorem~\ref{T4.6} is equivalent to
\[
W_i\ge \underline{\bigl(W_{i-1}-(0,p)\bigr)}_{[0,n]\times[-p,n]}\bigm|_U.
\]
\end{proof}

\begin{exmp}
\rm
(Backward slicing) 
Let $p=3$ and $m=12$ ($n=\frac m3(p-1)=8$). A backward consistent sequence 
of ideals $J_8,J_7,\dots,J_0$ is illustrated in Figure 16 through their
boundary walks $W_8, W_7,\dots,W_0$. When choosing walk $W_i$, we first determine the lower bound $X_i$ and the upper bound $Y_i$ defined in
Corollary~\ref{C4.7}. Figure 17 shows how $Y_1$ is determined and Figure 18
shows the procedure to find $X_1$. The ideal $I=\bigcup_{j=0}^8(J_j\times
\{j\})$ of ${\mathcal U}$ is depicted in Figure 19.

\vfill\eject

%%%%%%%%%%%% Figure  %%%%%%%%%%%%%%%%%%%%%%%%%%%%%%%
\vskip5mm
\setlength{\unitlength}{4mm}
\[
\begin{picture}(8,8)
\put(0,0){\line(1,0){8}}
\put(0,8){\line(1,0){8}}
\put(0,0){\line(0,1){8}}
\put(8,0){\line(0,1){8}}
\multiput(1,0)(1,0){7}{\line(0,1){0.1}}
\multiput(0,1)(0,1){7}{\line(1,0){0.1}}
\put(0,1){\makebox(0,0){$\scriptstyle \bullet$}}
\put(1,1){\makebox(0,0){$\scriptstyle \bullet$}}
\put(1,0){\makebox(0,0){$\scriptstyle \bullet$}}
\put(0,1){\line(1,0){1}}
\put(1,1){\line(0,-1){1}}
\put(4,-1){\makebox(0,0){$\scriptstyle W_8$}}

\end{picture}
\kern1cm
\begin{picture}(8,8)
\put(0,0){\line(1,0){8}}
\put(0,8){\line(1,0){8}}
\put(0,0){\line(0,1){8}}
\put(8,0){\line(0,1){8}}
\multiput(1,0)(1,0){7}{\line(0,1){0.1}}
\multiput(0,1)(0,1){7}{\line(1,0){0.1}}
\put(0,1){\makebox(0,0){$\scriptstyle \times$}}
\put(1,1){\makebox(0,0){$\scriptstyle \times$}}
\put(1,0){\makebox(0,0){$\scriptstyle \times$}}
\put(0,3){\makebox(0,0){$\scriptstyle \bullet$}}
\put(0,2){\makebox(0,0){$\scriptstyle \bullet$}}
\put(1,2){\makebox(0,0){$\scriptscriptstyle \bullet$}}
\put(1,1){\makebox(0,0){$\scriptstyle \bullet$}}
\put(2,1){\makebox(0,0){$\scriptstyle \bullet$}}
\put(2,0){\makebox(0,0){$\scriptstyle \bullet$}}
\put(0,4){\makebox(0,0){$\scriptstyle \circ$}}
\put(1,4){\makebox(0,0){$\scriptstyle \circ$}}
\put(1,2){\makebox(0,0){{\large $\circ$}}}
\put(4,2){\makebox(0,0){$\scriptstyle \circ$}}
\put(4,1){\makebox(0,0){$\scriptstyle \circ$}}
\put(7,1){\makebox(0,0){$\scriptstyle \circ$}}
\put(7,0){\makebox(0,0){$\scriptstyle \circ$}}
\put(8,0){\makebox(0,0){$\scriptstyle \circ$}}
\put(0,4){\line(1,0){1}}
\put(1,4){\line(0,-1){4}}
\put(0,2){\line(1,0){4}}
\put(4,2){\line(0,-1){1}}
\put(4,1){\line(1,0){3}}
\put(7,1){\line(0,-1){1}}
\put(0,1){\line(1,0){2}}
\put(2,1){\line(0,-1){1}}
\put(4,-1){\makebox(0,0){$\scriptstyle W_7,\ \alpha_7=1$}}

\end{picture}
\kern1cm
\begin{picture}(8,8)
\put(0,0){\line(1,0){8}}
\put(0,8){\line(1,0){8}}
\put(0,0){\line(0,1){8}}
\put(8,0){\line(0,1){8}}
\multiput(1,0)(1,0){7}{\line(0,1){0.1}}
\multiput(0,1)(0,1){7}{\line(1,0){0.1}}
\put(0,3){\makebox(0,0){$\scriptstyle \times$}}
\put(0,2){\makebox(0,0){$\scriptstyle \times$}}
\put(1,2){\makebox(0,0){$\scriptstyle \times$}}
\put(1,1){\makebox(0,0){$\scriptstyle \times$}}
\put(2,1){\makebox(0,0){$\scriptstyle \times$}}
\put(2,0){\makebox(0,0){$\scriptstyle \times$}}
\put(0,4){\makebox(0,0){$\scriptstyle \bullet$}}
\put(1,4){\makebox(0,0){$\scriptscriptstyle \bullet$}}
\put(1,2){\makebox(0,0){$\scriptstyle \bullet$}}
\put(3,2){\makebox(0,0){$\scriptscriptstyle \bullet$}}
\put(3,1){\makebox(0,0){$\scriptstyle \bullet$}}
\put(4,1){\makebox(0,0){$\scriptstyle \bullet$}}
\put(4,0){\makebox(0,0){$\scriptstyle \bullet$}}
\put(6,0){\makebox(0,0){$\scriptstyle \bullet$}}
\put(0,6){\makebox(0,0){$\scriptstyle \circ$}}
\put(0,5){\makebox(0,0){$\scriptstyle \circ$}}
\put(1,5){\makebox(0,0){$\scriptstyle \circ$}}
\put(1,4){\makebox(0,0){{\large $\circ$}}}
\put(2,4){\makebox(0,0){$\scriptstyle \circ$}}
\put(2,2){\makebox(0,0){$\scriptstyle \circ$}}
\put(5,2){\makebox(0,0){$\scriptstyle \circ$}}
\put(5,1){\makebox(0,0){$\scriptstyle \circ$}}
\put(8,1){\makebox(0,0){$\scriptstyle \circ$}}

\put(0,5){\line(1,0){1}}
\put(1,5){\line(0,-1){4}}
\put(1,1){\line(1,0){1}}
\put(2,1){\line(0,-1){1}}
\put(0,4){\line(1,0){2}}
\put(2,4){\line(0,-1){2}}
\put(0,2){\line(1,0){5}}
\put(5,2){\line(0,-1){1}}
\put(5,1){\line(1,0){3}}
\put(3,2){\line(0,-1){1}}
\put(3,1){\line(1,0){1}}
\put(4,1){\line(0,-1){1}}
\put(4,-1){\makebox(0,0){$\scriptstyle W_6,\ \alpha_6=2$}}

\end{picture}
\]
\vskip5mm

\[
\begin{picture}(8,8)
\put(0,0){\line(1,0){8}}
\put(0,8){\line(1,0){8}}
\put(0,0){\line(0,1){8}}
\put(8,0){\line(0,1){8}}
\multiput(1,0)(1,0){7}{\line(0,1){0.1}}
\multiput(0,1)(0,1){7}{\line(1,0){0.1}}

\put(0,4){\makebox(0,0){$\scriptstyle \times$}}
\put(1,4){\makebox(0,0){$\scriptstyle \times$}}
\put(1,2){\makebox(0,0){$\scriptstyle \times$}}
\put(3,2){\makebox(0,0){$\scriptstyle \times$}}
\put(3,1){\makebox(0,0){$\scriptstyle \times$}}
\put(4,1){\makebox(0,0){$\scriptstyle \times$}}
\put(4,0){\makebox(0,0){$\scriptstyle \times$}}
\put(6,0){\makebox(0,0){$\scriptstyle \times$}}
\put(0,4){\makebox(0,0){$\scriptstyle \bullet$}}
\put(2,4){\makebox(0,0){$\scriptstyle \bullet$}}
\put(2,3){\makebox(0,0){$\scriptstyle \bullet$}}
\put(4,3){\makebox(0,0){$\scriptscriptstyle \bullet$}}
\put(4,2){\makebox(0,0){$\scriptstyle \bullet$}}
\put(5,2){\makebox(0,0){$\scriptstyle \bullet$}}
\put(5,1){\makebox(0,0){$\scriptstyle \bullet$}}
\put(8,1){\makebox(0,0){$\scriptstyle \bullet$}}
\put(0,7){\makebox(0,0){$\scriptstyle \circ$}}
\put(1,7){\makebox(0,0){$\scriptstyle \circ$}}
\put(1,5){\makebox(0,0){$\scriptstyle \circ$}}
\put(3,5){\makebox(0,0){$\scriptstyle \circ$}}
\put(3,4){\makebox(0,0){$\scriptstyle \circ$}}
\put(4,4){\makebox(0,0){$\scriptstyle \circ$}}
\put(4,3){\makebox(0,0){$\circ$}}
\put(6,3){\makebox(0,0){$\scriptstyle \circ$}}
\put(6,2){\makebox(0,0){$\scriptstyle \circ$}}
\put(8,2){\makebox(0,0){$\scriptstyle \circ$}}

\put(0,7){\line(1,0){1}}
\put(1,7){\line(0,-1){2}}
\put(1,5){\line(1,0){2}}
\put(3,5){\line(0,-1){1}}
\put(3,4){\line(1,0){1}}
\put(4,4){\line(0,-1){2}}
\put(4,2){\line(1,0){1}}
\put(5,2){\line(0,-1){1}}
\put(5,1){\line(1,0){3}}
\put(1,4){\line(0,-1){2}}
\put(1,2){\line(1,0){2}}
\put(3,2){\line(0,-1){1}}
\put(3,1){\line(1,0){1}}
\put(4,1){\line(0,-1){1}}
\put(0,4){\line(1,0){2}}
\put(2,4){\line(0,-1){1}}
\put(2,3){\line(1,0){4}}
\put(6,3){\line(0,-1){1}}
\put(6,2){\line(1,0){2}}
\put(4,-1){\makebox(0,0){$\scriptstyle W_5,\ \alpha_5=2$}}

\end{picture}
\kern1cm
\begin{picture}(8,8)
\put(0,0){\line(1,0){8}}
\put(0,8){\line(1,0){8}}
\put(0,0){\line(0,1){8}}
\put(8,0){\line(0,1){8}}
\multiput(1,0)(1,0){7}{\line(0,1){0.1}}
\multiput(0,1)(0,1){7}{\line(1,0){0.1}}

\put(0,4){\makebox(0,0){$\scriptstyle \times$}}
\put(2,4){\makebox(0,0){$\scriptstyle \times$}}
\put(2,3){\makebox(0,0){$\scriptstyle \times$}}
\put(4,3){\makebox(0,0){$\scriptstyle \times$}}
\put(4,2){\makebox(0,0){$\scriptstyle \times$}}
\put(5,2){\makebox(0,0){$\scriptstyle \times$}}
\put(5,1){\makebox(0,0){$\scriptstyle \times$}}
\put(8,1){\makebox(0,0){$\scriptstyle \times$}}
\put(0,6){\makebox(0,0){$\scriptstyle \bullet$}}
\put(0,5){\makebox(0,0){$\scriptstyle \bullet$}}
\put(3,5){\makebox(0,0){$\scriptstyle \bullet$}}
\put(3,4){\makebox(0,0){$\scriptstyle \bullet$}}
\put(5,4){\makebox(0,0){$\scriptscriptstyle \bullet$}}
\put(5,3){\makebox(0,0){$\scriptstyle \bullet$}}
\put(8,3){\makebox(0,0){$\scriptstyle \bullet$}}
\put(0,7){\makebox(0,0){$\scriptstyle \circ$}}
\put(2,7){\makebox(0,0){$\scriptstyle \circ$}}
\put(2,6){\makebox(0,0){$\scriptstyle \circ$}}
\put(4,6){\makebox(0,0){$\scriptstyle \circ$}}
\put(4,5){\makebox(0,0){$\scriptstyle \circ$}}
\put(5,5){\makebox(0,0){$\scriptstyle \circ$}}
\put(5,4){\makebox(0,0){{\large $\circ$}}}
\put(8,4){\makebox(0,0){$\scriptstyle \circ$}}

\put(0,7){\line(1,0){2}}
\put(2,7){\line(0,-1){1}}
\put(2,6){\line(1,0){2}}
\put(4,6){\line(0,-1){1}}
\put(4,5){\line(1,0){1}}
\put(5,5){\line(0,-1){2}}
\put(5,3){\line(1,0){3}}
\put(0,5){\line(1,0){3}}
\put(3,5){\line(0,-1){1}}
\put(3,4){\line(1,0){5}}
\put(0,4){\line(1,0){2}}
\put(2,4){\line(0,-1){1}}
\put(2,3){\line(1,0){2}}
\put(4,3){\line(0,-1){1}}
\put(4,2){\line(1,0){1}}
\put(5,2){\line(0,-1){1}}
\put(5,1){\line(1,0){3}}
\put(4,-1){\makebox(0,0){$\scriptstyle W_4,\ \alpha_4=2$}}

\end{picture}
\kern1cm
\begin{picture}(8,8)
\put(0,0){\line(1,0){8}}
\put(0,8){\line(1,0){8}}
\put(0,0){\line(0,1){8}}
\put(8,0){\line(0,1){8}}
\multiput(1,0)(1,0){7}{\line(0,1){0.1}}
\multiput(0,1)(0,1){7}{\line(1,0){0.1}}

\put(0,6){\makebox(0,0){$\scriptstyle \times$}}
\put(0,5){\makebox(0,0){$\scriptstyle \times$}}
\put(3,5){\makebox(0,0){$\scriptstyle \times$}}
\put(3,4){\makebox(0,0){$\scriptstyle \times$}}
\put(5,4){\makebox(0,0){$\scriptstyle \times$}}
\put(5,3){\makebox(0,0){$\scriptstyle \times$}}
\put(8,3){\makebox(0,0){$\scriptstyle \times$}}
\put(0,6){\makebox(0,0){$\scriptstyle \bullet$}}
\put(0,5){\makebox(0,0){$\scriptstyle \bullet$}}
\put(3,5){\makebox(0,0){$\scriptstyle \bullet$}}
\put(3,4){\makebox(0,0){$\scriptstyle \bullet$}}
\put(5,4){\makebox(0,0){$\scriptstyle \bullet$}}
\put(5,3){\makebox(0,0){$\scriptstyle \bullet$}}
\put(8,3){\makebox(0,0){$\scriptstyle \bullet$}}
\put(3,8){\makebox(0,0){$\scriptstyle \circ$}}
\put(3,7){\makebox(0,0){$\scriptstyle \circ$}}
\put(5,7){\makebox(0,0){$\scriptstyle \circ$}}
\put(5,6){\makebox(0,0){$\scriptstyle \circ$}}
\put(8,6){\makebox(0,0){$\scriptstyle \circ$}}

\put(3,8){\line(0,-1){1}}
\put(3,7){\line(1,0){2}}
\put(5,7){\line(0,-1){1}}
\put(5,6){\line(1,0){3}}
\put(0,5){\line(1,0){3}}
\put(3,5){\line(0,-1){1}}
\put(3,4){\line(1,0){2}}
\put(5,4){\line(0,-1){1}}
\put(5,3){\line(1,0){3}}
\put(4,-1){\makebox(0,0){$\scriptstyle W_3,\ \alpha_3=2$}}

\end{picture}
\]
\vskip5mm

\[
\begin{picture}(8,8)
\put(0,0){\line(1,0){8}}
\put(0,8){\line(1,0){8}}
\put(0,0){\line(0,1){8}}
\put(8,0){\line(0,1){8}}
\multiput(1,0)(1,0){7}{\line(0,1){0.1}}
\multiput(0,1)(0,1){7}{\line(1,0){0.1}}

\put(0,6){\makebox(0,0){$\scriptstyle \times$}}
\put(0,5){\makebox(0,0){$\scriptstyle \times$}}
\put(3,5){\makebox(0,0){$\scriptstyle \times$}}
\put(3,4){\makebox(0,0){$\scriptstyle \times$}}
\put(5,4){\makebox(0,0){$\scriptstyle \times$}}
\put(5,3){\makebox(0,0){$\scriptstyle \times$}}
\put(8,3){\makebox(0,0){$\scriptstyle \times$}}
\put(1,8){\makebox(0,0){$\scriptstyle \bullet$}}
\put(1,6){\makebox(0,0){$\scriptstyle \bullet$}}
\put(3,6){\makebox(0,0){$\scriptstyle \bullet$}}
\put(3,4){\makebox(0,0){$\scriptstyle \bullet$}}
\put(5,4){\makebox(0,0){$\scriptstyle \bullet$}}
\put(5,3){\makebox(0,0){$\scriptstyle \bullet$}}
\put(8,3){\makebox(0,0){$\scriptstyle \bullet$}}
\put(3,8){\makebox(0,0){$\scriptstyle \circ$}}
\put(3,7){\makebox(0,0){$\scriptstyle \circ$}}
\put(5,7){\makebox(0,0){$\scriptstyle \circ$}}
\put(5,6){\makebox(0,0){$\scriptstyle \circ$}}
\put(8,6){\makebox(0,0){$\scriptstyle \circ$}}

\put(3,8){\line(0,-1){1}}
\put(3,7){\line(1,0){2}}
\put(5,7){\line(0,-1){1}}
\put(5,6){\line(1,0){3}}
\put(1,8){\line(0,-1){2}}
\put(1,6){\line(1,0){2}}
\put(3,6){\line(0,-1){2}}
\put(3,4){\line(1,0){2}}
\put(5,4){\line(0,-1){1}}
\put(5,3){\line(1,0){3}}
\put(0,5){\line(1,0){3}}
\put(4,-1){\makebox(0,0){$\scriptstyle W_2,\ \alpha_2=2$}}

\end{picture}
\kern1cm
\begin{picture}(8,8)
\put(0,0){\line(1,0){8}}
\put(0,8){\line(1,0){8}}
\put(0,0){\line(0,1){8}}
\put(8,0){\line(0,1){8}}
\multiput(1,0)(1,0){7}{\line(0,1){0.1}}
\multiput(0,1)(0,1){7}{\line(1,0){0.1}}

\put(1,8){\makebox(0,0){$\scriptstyle \times$}}
\put(1,6){\makebox(0,0){$\scriptstyle \times$}}
\put(3,6){\makebox(0,0){$\scriptstyle \times$}}
\put(3,5){\makebox(0,0){$\scriptstyle \times$}}
\put(4,5){\makebox(0,0){$\scriptstyle \times$}}
\put(4,4){\makebox(0,0){$\scriptstyle \times$}}
\put(6,4){\makebox(0,0){$\scriptstyle \times$}}
\put(6,3){\makebox(0,0){$\scriptstyle \times$}}
\put(8,3){\makebox(0,0){$\scriptstyle \times$}}
\put(2,8){\makebox(0,0){$\scriptstyle \bullet$}}
\put(2,7){\makebox(0,0){$\scriptstyle \bullet$}}
\put(4,7){\makebox(0,0){$\scriptstyle \bullet$}}
\put(4,6){\makebox(0,0){$\scriptstyle \bullet$}}
\put(7,6){\makebox(0,0){$\scriptstyle \bullet$}}
\put(7,3){\makebox(0,0){$\scriptstyle \bullet$}}
\put(8,3){\makebox(0,0){$\scriptstyle \bullet$}}
\put(3,8){\makebox(0,0){$\scriptstyle \circ$}}
\put(3,7){\makebox(0,0){$\scriptstyle \circ$}}
\put(5,7){\makebox(0,0){$\scriptstyle \circ$}}
\put(5,6){\makebox(0,0){$\scriptstyle \circ$}}
\put(8,6){\makebox(0,0){$\scriptstyle \circ$}}

\put(1,8){\line(0,-1){2}}
\put(1,6){\line(1,0){2}}
\put(3,6){\line(0,-1){1}}
\put(3,5){\line(1,0){1}}
\put(4,5){\line(0,-1){1}}
\put(4,4){\line(1,0){2}}
\put(6,4){\line(0,-1){1}}
\put(6,3){\line(1,0){2}}
\put(2,8){\line(0,-1){1}}
\put(2,7){\line(1,0){3}}
\put(5,7){\line(0,-1){1}}
\put(4,6){\line(1,0){4}}
\put(7,6){\line(0,-1){3}}
\put(3,8){\line(0,-1){1}}
\put(4,7){\line(0,-1){1}}
\put(4,-1){\makebox(0,0){$\scriptstyle W_1,\ \alpha_1=2$}}

\end{picture}
\kern1cm
\begin{picture}(8,8)
\put(0,0){\line(1,0){8}}
\put(0,8){\line(1,0){8}}
\put(0,0){\line(0,1){8}}
\put(8,0){\line(0,1){8}}
\multiput(1,0)(1,0){7}{\line(0,1){0.1}}
\multiput(0,1)(0,1){7}{\line(1,0){0.1}}

\put(2,8){\makebox(0,0){$\scriptstyle \times$}}
\put(2,7){\makebox(0,0){$\scriptstyle \times$}}
\put(4,7){\makebox(0,0){$\scriptstyle \times$}}
\put(4,6){\makebox(0,0){$\scriptstyle \times$}}
\put(7,6){\makebox(0,0){$\scriptstyle \times$}}
\put(7,3){\makebox(0,0){$\scriptstyle \times$}}
\put(8,3){\makebox(0,0){$\scriptstyle \times$}}
\put(6,8){\makebox(0,0){$\scriptstyle \bullet$}}
\put(6,6){\makebox(0,0){$\scriptstyle \bullet$}}
\put(7,6){\makebox(0,0){$\scriptscriptstyle \bullet$}}
\put(7,5){\makebox(0,0){$\scriptstyle \bullet$}}
\put(8,5){\makebox(0,0){$\scriptstyle \bullet$}}
\put(7,8){\makebox(0,0){$\scriptstyle \circ$}}
\put(7,6){\makebox(0,0){{\large $\circ$}}}
\put(8,6){\makebox(0,0){$\scriptstyle \circ$}}

\put(2,8){\line(0,-1){1}}
\put(2,7){\line(1,0){2}}
\put(4,7){\line(0,-1){1}}
\put(4,6){\line(1,0){4}}
\put(6,8){\line(0,-1){2}}
\put(7,5){\line(1,0){1}}
\put(7,8){\line(0,-1){5}}
\put(7,3){\line(1,0){1}}
\put(4,-1){\makebox(0,0){$\scriptstyle W_0,\ \alpha_0=2$}}

\end{picture}
\]
\vskip5mm

\[
\begin{picture}(3,1)
\put(0,1){\makebox(0,0){$\scriptstyle W_i:$}}
\put(1,1){\makebox(0,0){$\scriptstyle \bullet$}}
\put(2,1){\makebox(0,0){$\scriptstyle \bullet$}}
\put(1,1){\line(1,0){1}}
\end{picture}
\kern 5mm
\begin{picture}(3,1)
\put(0,1){\makebox(0,0){$\scriptstyle X_i:$}}
\put(1,1){\makebox(0,0){$\scriptstyle \times$}}
\put(2,1){\makebox(0,0){$\scriptstyle \times$}}
\put(1,1){\line(1,0){1}}
\end{picture}
\kern 5mm
\begin{picture}(3,1)
\put(0,1){\makebox(0,0){$\scriptstyle Y_i:$}}
\put(1,1){\makebox(0,0){$\scriptstyle \circ$}}
\put(2,1){\makebox(0,0){$\scriptstyle \circ$}}
\put(1,1){\line(1,0){1}}
\end{picture}
\]

\[
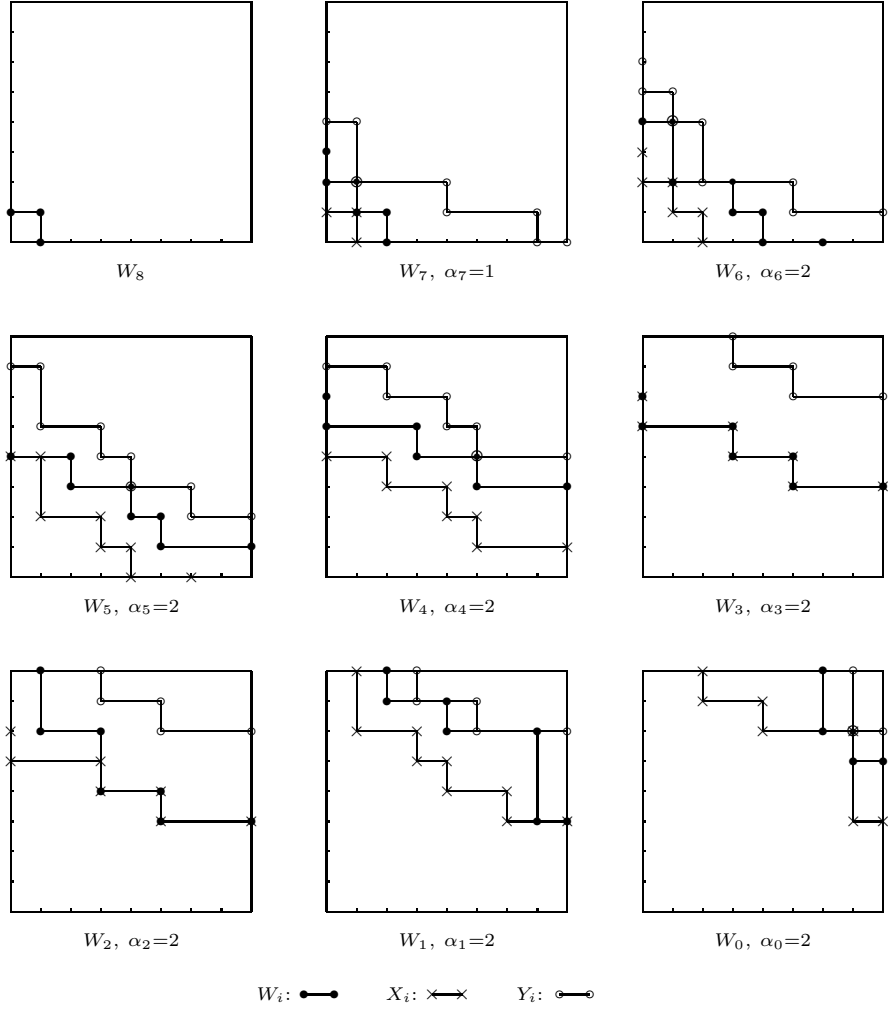
\text{Figure 16. An example of backward slicing}
\]
%%%%%%%%%%%%%%%% End Figure %%%%%%%%%%%%%%%%%%%%%%%%%%%%%%%%%%%%%%
\vfill\eject

%%%%%%%%%%%% Figure 17 %%%%%%%%%%%%%%%%%%%%%%%%%%%%%%%
\vskip5mm
\setlength{\unitlength}{4mm}
\[
\begin{picture}(12,14)
\put(0,4){\vector(1,0){12}}
\put(2,0){\vector(0,1){14}}
\multiput(1,4)(1,0){9}{\line(0,1){0.1}}
\multiput(2,2)(0,1){10}{\line(1,0){0.1}}

\put(3,12){\makebox(0,0){$\scriptstyle \circ$}}
\put(3,10){\makebox(0,0){$\scriptstyle \circ$}}
\put(5,10){\makebox(0,0){$\scriptstyle \circ$}}
\put(5,9){\makebox(0,0){$\scriptstyle \circ$}}
\put(5,8){\makebox(0,0){$\scriptstyle \circ$}}
\put(7,8){\makebox(0,0){$\scriptstyle \circ$}}
\put(7,7){\makebox(0,0){$\scriptstyle \circ$}}
\put(10,7){\makebox(0,0){$\scriptstyle \circ$}}

\put(2,12){\line(1,0){8}}
\put(10,12){\line(0,-1){11}}
\put(2,9){\line(1,0){8}}
\put(2,1){\line(1,0){8}}
\put(3,12){\line(0,-1){2}}
\put(3,10){\line(1,0){2}}
\put(5,10){\line(0,-1){2}}
\put(5,8){\line(1,0){2}}
\put(7,8){\line(0,-1){1}}
\put(7,7){\line(1,0){3}}

\put(3,12){\line(3,1){3}}
\put(10,7){\line(-2,3){4}}
\put(5,9){\line(3,1){6}}
\put(10,7){\line(1,4){1}}

\put(6,13){\makebox(0,0)[b]{$\scriptstyle W_2$}}
\put(11,11){\makebox(0,0)[l]{$\scriptstyle Y_1-(0,3)$}}

\end{picture}
\]

\[
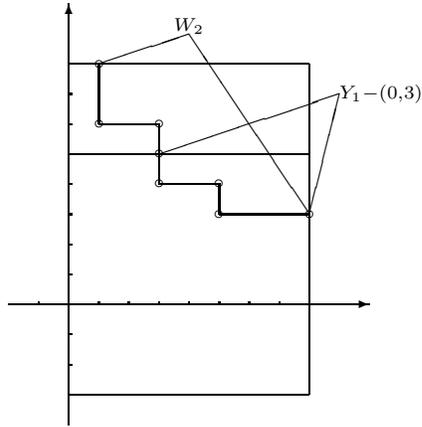
\text{Figure 17. Determination of $Y_1$}
\]
%%%%%%%%%%%%%%%% End Figure %%%%%%%%%%%%%%%%%%%%%%%%%%%%%%%%%%%%%%

%%%%%%%%%%%% Figure  18 %%%%%%%%%%%%%%%%%%%%%%%%%%%%%%%
\vskip5mm
\setlength{\unitlength}{4mm}
\[
\begin{picture}(9,9)
\put(0,0){\vector(1,0){9}}
\put(0,0){\vector(0,1){9}}
\multiput(1,0)(1,0){7}{\line(0,1){0.1}}
\multiput(0,1)(0,1){7}{\line(1,0){0.1}}
\put(0,8){\line(1,0){8}}
\put(8,0){\line(0,1){8}}

\put(1,8){\makebox(0,0){$\scriptstyle \times$}}
\put(1,6){\makebox(0,0){$\scriptstyle \times$}}
\put(3,6){\makebox(0,0){$\scriptstyle \times$}}
\put(3,4){\makebox(0,0){$\scriptstyle \times$}}
\put(5,4){\makebox(0,0){$\scriptstyle \times$}}
\put(5,3){\makebox(0,0){$\scriptstyle \times$}}
\put(8,3){\makebox(0,0){$\scriptstyle \times$}}

\put(1,8){\line(0,-1){2}}
\put(1,6){\line(1,0){2}}
\put(3,6){\line(0,-1){2}}
\put(3,4){\line(1,0){2}}
\put(5,4){\line(0,-1){1}}
\put(5,3){\line(1,0){3}}

\put(4.5,-1){\makebox(0,0){$\scriptstyle W_2$}}

\end{picture}
\kern2cm
\begin{picture}(10,9)
\put(0,0){\vector(1,0){10}}
\put(0,0){\vector(0,1){9}}
\multiput(1,0)(1,0){7}{\line(0,1){0.1}}
\multiput(0,1)(0,1){7}{\line(1,0){0.1}}
\put(0,8){\line(1,0){9}}
\put(1,0){\line(0,1){8}}
\put(8,0){\line(0,1){8}}
\put(9,0){\line(0,1){8}}
\put(-0.1,-0.1){\makebox(0,0)[tr]{$\scriptstyle 0$}}

\put(0,6){\makebox(0,0){$\scriptstyle \times$}}
\put(1,6){\makebox(0,0){$\scriptstyle \times$}}
\put(1,5){\makebox(0,0){$\scriptstyle \times$}}
\put(4,5){\makebox(0,0){$\scriptstyle \times$}}
\put(4,4){\makebox(0,0){$\scriptstyle \times$}}
\put(6,4){\makebox(0,0){$\scriptstyle \times$}}
\put(6,3){\makebox(0,0){$\scriptstyle \times$}}
\put(8,3){\makebox(0,0){$\scriptstyle \times$}}
\put(9,3){\makebox(0,0){$\scriptstyle \times$}}

\put(0,6){\line(1,0){1}}
\put(1,5){\line(1,0){3}}
\put(4,5){\line(0,-1){1}}
\put(4,4){\line(1,0){2}}
\put(6,4){\line(0,-1){1}}
\put(6,3){\line(1,0){3}}

\put(1,6){\line(1,1){3.6}}
\put(9,3){\line(-2,3){4.4}}
\put(0,6){\line(1,-2){3.7}}
\put(8,3){\line(-1,-1){4.3}}

\put(5,10){\makebox(0,0)[b]{$\scriptstyle W_4+(1,0)$}}
\put(3,-2){\makebox(0,0)[l]{$\scriptstyle B=\underline{(W_4+(1,0))}_{[0,9]\times[0,8]}|_U$}}

\end{picture}
\]
\vskip2mm
\[
\begin{picture}(10,12)
\put(0,3){\vector(1,0){10}}
\put(0,0){\vector(0,1){12}}
\multiput(1,3)(1,0){7}{\line(0,1){0.1}}
\multiput(0,1)(0,1){11}{\line(1,0){0.1}}
\put(-0.1,3){\makebox(0,0)[r]{$\scriptstyle 0$}}

\put(0,11){\line(1,0){9}}
\put(9,11){\line(0,-1){11}}
\put(1,8){\line(1,0){8}}
\put(1,8){\line(0,-1){8}}
\put(0,0){\line(1,0){9}}
\put(8,11){\line(0,-1){8}}

\put(0,6){\line(1,0){1}}
\put(1,5){\line(1,0){3}}
\put(4,5){\line(0,-1){1}}
\put(4,4){\line(1,0){2}}
\put(6,4){\line(0,-1){1}}

\put(1,6){\line(1,2){3.2}}
\put(9,3){\line(-1,2){4.7}}
\put(0,6){\line(1,-2){3.6}}
\put(8,3){\line(-1,-1){4.3}}

\put(0,6){\makebox(0,0){$\scriptstyle \times$}}
\put(1,6){\makebox(0,0){$\scriptstyle \times$}}
\put(1,5){\makebox(0,0){$\scriptstyle \times$}}
\put(4,5){\makebox(0,0){$\scriptstyle \times$}}
\put(4,4){\makebox(0,0){$\scriptstyle \times$}}
\put(6,4){\makebox(0,0){$\scriptstyle \times$}}
\put(6,3){\makebox(0,0){$\scriptstyle \times$}}
\put(8,3){\makebox(0,0){$\scriptstyle \times$}}
\put(9,3){\makebox(0,0){$\scriptstyle \times$}}

\put(4.5,13){\makebox(0,0){$\scriptstyle W_3+(1,-3)$}}
\put(3,-2){\makebox(0,0)[l]{$\scriptstyle C=\underline{(W_3+(1,-3))}_{[0,9]\times[-3,8]}|_U$}}

\end{picture}
\kern2cm
\begin{picture}(9,9)
\put(0,0){\vector(1,0){9}}
\put(0,0){\vector(0,1){9}}
\multiput(1,0)(1,0){7}{\line(0,1){0.1}}
\multiput(0,1)(0,1){7}{\line(1,0){0.1}}
\put(0,8){\line(1,0){8}}
\put(8,0){\line(0,1){8}}

\put(1,8){\makebox(0,0){$\scriptstyle \times$}}
\put(1,6){\makebox(0,0){$\scriptstyle \times$}}
\put(3,6){\makebox(0,0){$\scriptstyle \times$}}
\put(3,5){\makebox(0,0){$\scriptstyle \times$}}
\put(4,5){\makebox(0,0){$\scriptstyle \times$}}
\put(4,4){\makebox(0,0){$\scriptstyle \times$}}
\put(6,4){\makebox(0,0){$\scriptstyle \times$}}
\put(6,3){\makebox(0,0){$\scriptstyle \times$}}
\put(8,3){\makebox(0,0){$\scriptstyle \times$}}

\put(1,8){\line(0,-1){2}}
\put(1,6){\line(1,0){2}}
\put(3,6){\line(0,-1){1}}
\put(3,5){\line(1,0){1}}
\put(4,5){\line(0,-1){1}}
\put(4,4){\line(1,0){2}}
\put(6,4){\line(0,-1){1}}
\put(6,3){\line(1,0){2}}

\put(4.5,-1){\makebox(0,0){$\scriptstyle X_1=W_2\vee B\vee C$}}

\end{picture}
\]
\vskip5mm

\[
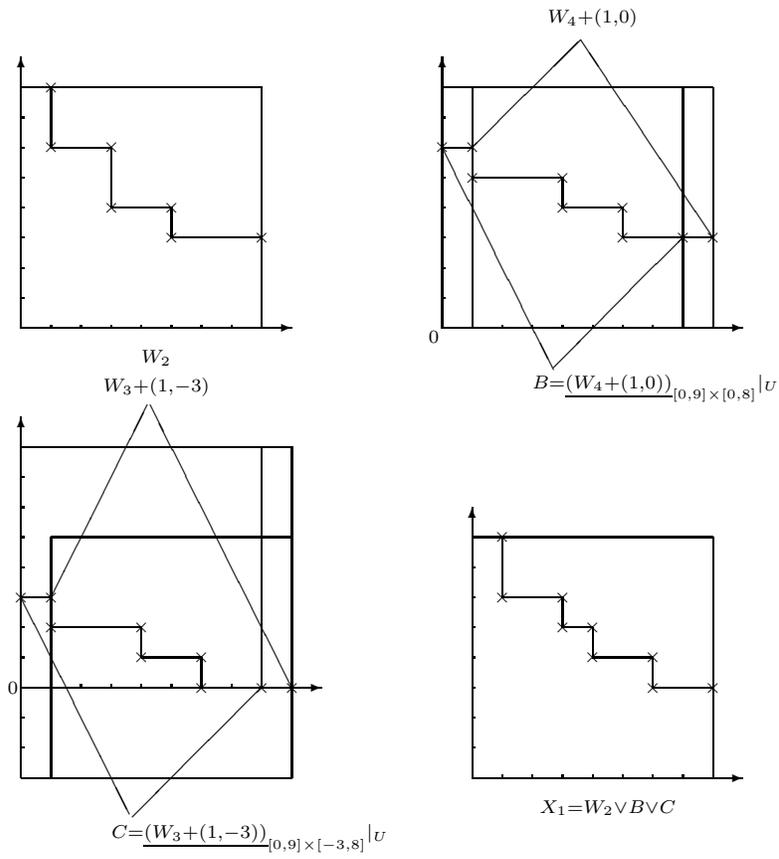
\text{Figure 18. Detremination of $X_1$}
\]
%%%%%%%%%%%%%%%% End Figure %%%%%%%%%%%%%%%%%%%%%%%%%%%%%%%%%%%%%%

%%%%%%%%%%%% Figure 19 %%%%%%%%%%%%%%%%%%%%%%%%%%%%%%%
\vskip5mm
\setlength{\unitlength}{8mm}
\[
\begin{picture}(13,13)
\put(12,4){\vector(1,0){1}}
\put(4,12){\vector(0,1){1}}
\put(0,0){\vector(-1,-1){0.5}}
\put(13,3.8){\makebox(0,0)[t]{$\scriptstyle x$}}
\put(4.1,13){\makebox(0,0)[l]{$\scriptstyle y$}}
\put(-0.5,-0.3){\makebox(0,0)[b]{$\scriptstyle z$}}
\put(12,3.8){\makebox(0,0)[t]{$\scriptstyle 8$}}
\put(4.1,12.2){\makebox(0,0)[bl]{$\scriptstyle 8$}}
\put(-0.1,0){\makebox(0,0)[br]{$\scriptstyle 8$}}
\put(4,12){\line(1,0){6}}
\multiput(6,11)(-0.5,-0.5){2}{\line(1,0){2}}
\multiput(8,10)(-0.5,-0.5){2}{\line(1,0){3}}
\multiput(11,9)(0,-2){2}{\line(1,0){1}}
\multiput(4.5,9.5)(-0.5,-0.5){2}{\line(1,0){2}}
\multiput(4.5,11.5)(-1.5,-0.5){2}{\line(1,0){1}}
\multiput(6.5,7.5)(-1.5,-1.5){2}{\line(1,0){2}}
\put(8.5,6.5){\line(1,0){2}}
\multiput(3,8)(-1,-1){2}{\line(1,0){3}}
\put(7,5){\line(1,0){3}}
\put(2,6){\line(1,0){2}}
\put(2.5,5.5){\line(1,0){1}}
\multiput(4,5)(-0.5,-0.5){2}{\line(1,0){2}}
\multiput(6,4)(-0.5,-0.5){2}{\line(1,0){1}}
\multiput(7,3)(-0.5,-0.5){2}{\line(1,0){3}}
\multiput(1,5)(0,-2){2}{\line(1,0){1}}
\multiput(0.5,2.5)(0,-1){2}{\line(1,0){1}}
\multiput(0,1)(0,-1){2}{\line(1,0){1}}
\multiput(1.5,1.5)(0,-1){2}{\line(1,0){1}}
\put(2,2){\line(1,0){1}}
\put(3,1){\line(1,0){4}}
\put(5.5,1.5){\line(1,0){4}}
\multiput(4.5,2.5)(-0.5,-0.5){2}{\line(1,0){1}}
\multiput(2.5,3.5)(-0.5,-0.5){2}{\line(1,0){2}}

%%%%%%%%%%%%%%%%%%%%%%%%%%%%%%%%%%%%%%%%%%%%%
\put(4,12){\line(-1,-1){1}}
\multiput(6,12)(0,-1){2}{\line(-1,-1){0.5}}
\multiput(8,11)(0,-1){2}{\line(-1,-1){0.5}}
\multiput(11,10)(0,-3){2}{\line(-1,-1){0.5}}
\multiput(4.5,11.5)(0,-2){2}{\line(-1,-1){0.5}}
\put(6,9){\line(1,1){0.5}}
\multiput(3,9)(0,-1){2}{\line(-1,-1){1}}
\put(6,8){\line(-1,-1){1}}
\multiput(6.5,7.5)(2,0){2}{\line(-1,-1){1.5}}
\put(8.5,6.5){\line(-1,-1){1.5}}
\put(12,7){\line(-1,-1){2}}
\multiput(4,6)(0,-1){2}{\line(-1,-1){0.5}}
\multiput(6,5)(0,-1){2}{\line(-1,-1){0.5}}
\multiput(7,4)(0,-1){2}{\line(-1,-1){0.5}}
\put(10,3){\line(-1,-1){0.5}}
\put(12,4){\line(-1,-1){2.5}}
\put(2,6){\line(-1,-1){1}}
\multiput(1,4)(0,-1){2}{\line(-1,-1){0.5}}
\multiput(3,2)(0,-1){2}{\line(-1,-1){0.5}}
\multiput(2.5,3.5)(-0.5,-1.5){2}{\line(-1,-1){1}}
\multiput(0.5,1.5)(1,-1){2}{\line(-1,-1){0.5}}
\put(2.5,5.5){\line(-1,-1){0.5}}
\multiput(4.5,3.5)(0,-1){2}{\line(-1,-1){0.5}}
\multiput(5.5,2.5)(0,-1){2}{\line(-1,-1){0.5}}
\put(7.5,1.5){\line(-1,-1){0.5}}

%%%%%%%%%%%%%%%%%%%%%%%%%%%%%%%%%%%%%%%%%%%%%
\multiput(6,12)(-0.5,-0.5){2}{\line(0,-1){1}}
\multiput(8,11)(-0.5,-0.5){2}{\line(0,-1){1}}
\multiput(11,10)(-0.5,-0.5){2}{\line(0,-1){3}}
\multiput(4.5,11.5)(-0.5,-0.5){2}{\line(0,-1){2}}
\put(10,12){\line(0,-1){2}}
\put(12,9){\line(0,-1){5}}
\put(3,11){\line(0,-1){3}}
\put(6,9){\line(0,-1){1}}
\put(6.5,9.5){\line(0,-1){2}}
\put(2,8){\line(0,-1){2}}
\put(5,7){\line(0,-1){1}}
\multiput(8.5,7.5)(-1.5,-1.5){2}{\line(0,-1){1}}
\multiput(4,6)(-0.5,-0.5){2}{\line(0,-1){1}}
\multiput(6,5)(-0.5,-0.5){2}{\line(0,-1){1}}
\multiput(7,4)(-0.5,-0.5){2}{\line(0,-1){1}}
\put(10,5){\line(0,-1){2}}
\put(9.5,2.5){\line(0,-1){1}}
\put(1,5){\line(0,-1){2}}
\put(2,5){\line(0,-1){3}}
\put(0.5,3.5){\line(0,-1){2}}
\put(1.5,2.5){\line(0,-1){2}}
\multiput(0,1)(1,0){2}{\line(0,-1){1}}
\multiput(3,2)(-0.5,-0.5){2}{\line(0,-1){1}}
\multiput(4.5,3.5)(-0.5,-0.5){2}{\line(0,-1){1}}
\multiput(5.5,2.5)(-0.5,-0.5){2}{\line(0,-1){1}}
\put(2.5,5.5){\line(0,-1){2}}

%%%%%%%%%%%%%%%%%%%%%%%%%%%%%%%%%%%%%%%%%%%%
\end{picture}
\]
\vskip5mm

\[
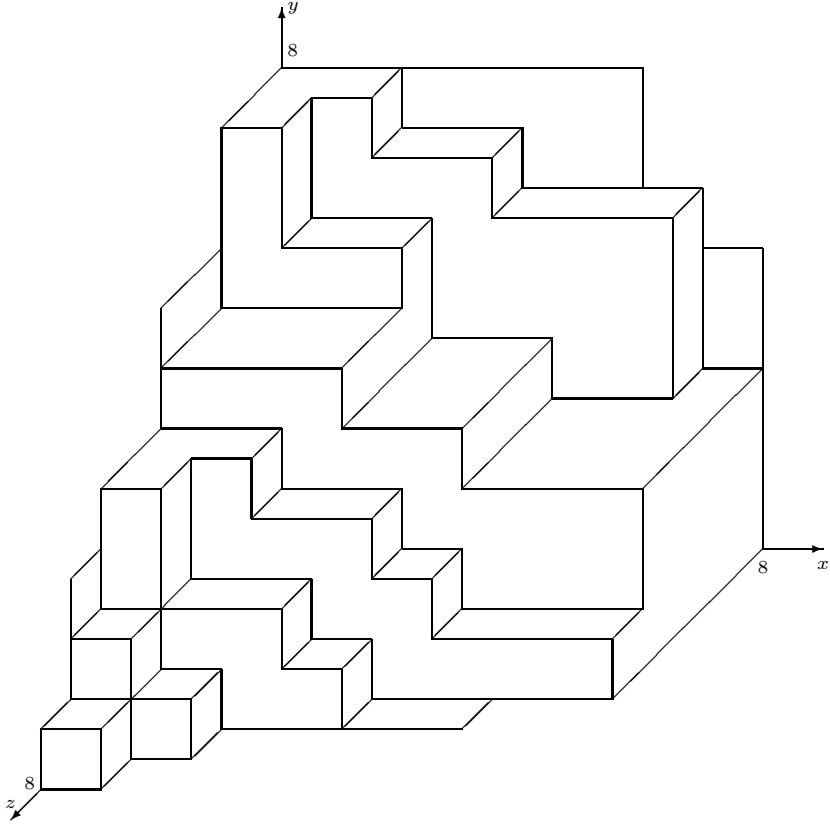
\text{Figure 19. The ideal $\textstyle I=\bigcup_{j=0}^8\bigl(J_j\times\{j\}\bigr)$}
\]
%%%%%%%%%%%%%%%% End Figure %%%%%%%%%%%%%%%%%%%%%%%%%%%%%%%%%%%%%%
\vskip5mm

\end{exmp}

\begin{exmp}
\rm (Forward slicing)
Let $p=3$ and $m=9$ ($n=\frac m3(p-1)=6$). A sequence 
of  walks $W_0, W_1,\dots,W_6$ satisfying (\ref{XWY'}) is given in Figure 20.  The resulting ideal $I=\bigcup_{j=0}^6(\iota(W_j)\times
\{j\})$ of ${\mathcal U}$ is illustrated in Figure 21.
\end{exmp}

\vfill\eject

%%%%%%%%%%%% Figure 20 %%%%%%%%%%%%%%%%%%%%%%%%%%%%%%%
\vskip5mm
\setlength{\unitlength}{4mm}
\[
\begin{picture}(6,6)
\put(0,0){\line(1,0){6}}
\put(0,6){\line(1,0){6}}
\put(0,0){\line(0,1){6}}
\put(6,0){\line(0,1){6}}
\multiput(1,0)(1,0){5}{\line(0,1){0.1}}
\multiput(0,1)(0,1){5}{\line(1,0){0.1}}
\put(5,6){\makebox(0,0){$\scriptstyle \bullet$}}
\put(5,1){\makebox(0,0){$\scriptstyle \bullet$}}
\put(6,1){\makebox(0,0){$\scriptstyle \bullet$}}
\put(5,6){\line(0,-1){5}}
\put(5,1){\line(1,0){1}}
\put(3,-1){\makebox(0,0){$\scriptstyle W_0$}}

\end{picture}
\kern1cm
\begin{picture}(6,6)
\put(0,0){\line(1,0){6}}
\put(0,6){\line(1,0){6}}
\put(0,0){\line(0,1){6}}
\put(6,0){\line(0,1){6}}
\multiput(1,0)(1,0){5}{\line(0,1){0.1}}
\multiput(0,1)(0,1){5}{\line(1,0){0.1}}

\put(0,4){\makebox(0,0){$\scriptstyle \times$}}
\put(2,4){\makebox(0,0){$\scriptstyle \times$}}
\put(2,3){\makebox(0,0){$\scriptstyle \times$}}
\put(5,3){\makebox(0,0){$\scriptstyle \times$}}
\put(5,0){\makebox(0,0){$\scriptstyle \times$}}
\put(5,6){\makebox(0,0){$\scriptscriptstyle \bullet$}}
\put(5,0){\makebox(0,0){$\scriptstyle \bullet$}}
\put(5,6){\makebox(0,0){{\large $\circ$}}}
\put(5,1){\makebox(0,0){$\scriptstyle \circ$}}
\put(6,1){\makebox(0,0){$\scriptstyle \circ$}}

\put(5,6){\line(0,-1){6}}
\put(5,1){\line(1,0){1}}
\put(0,4){\line(1,0){2}}
\put(2,4){\line(0,-1){1}}
\put(2,3){\line(1,0){3}}
\put(3,-1){\makebox(0,0){$\scriptstyle W_1,\ \beta_1=1$}}

\end{picture}
\kern1cm
\begin{picture}(6,6)
\put(0,0){\line(1,0){6}}
\put(0,6){\line(1,0){6}}
\put(0,0){\line(0,1){6}}
\put(6,0){\line(0,1){6}}
\multiput(1,0)(1,0){5}{\line(0,1){0.1}}
\multiput(0,1)(0,1){5}{\line(1,0){0.1}}

\put(0,4){\makebox(0,0){$\scriptstyle \times$}}
\put(2,4){\makebox(0,0){$\scriptstyle \times$}}
\put(2,3){\makebox(0,0){$\scriptstyle \times$}}
\put(5,3){\makebox(0,0){$\scriptstyle \times$}}
\put(5,0){\makebox(0,0){$\scriptstyle \times$}}
\put(3,6){\makebox(0,0){$\scriptstyle \bullet$}}
\put(3,3){\makebox(0,0){$\scriptstyle \bullet$}}
\put(5,3){\makebox(0,0){$\scriptscriptstyle \bullet$}}
\put(5,0){\makebox(0,0){$\scriptscriptstyle \bullet$}}
\put(4,6){\makebox(0,0){$\scriptstyle \circ$}}
\put(4,4){\makebox(0,0){$\scriptstyle \circ$}}
\put(5,4){\makebox(0,0){$\scriptstyle \circ$}}
\put(5,0){\makebox(0,0){{\large $\circ$}}}

\put(0,4){\line(1,0){2}}
\put(2,4){\line(0,-1){1}}
\put(2,3){\line(1,0){3}}
\put(5,4){\line(0,-1){4}}
\put(4,4){\line(1,0){1}}
\put(4,6){\line(0,-1){2}}
\put(3,6){\line(0,-1){3}}
\put(3,-1){\makebox(0,0){$\scriptstyle W_2,\ \beta_2=2$}}

\end{picture}
\kern1cm
\begin{picture}(6,6)
\put(0,0){\line(1,0){6}}
\put(0,6){\line(1,0){6}}
\put(0,0){\line(0,1){6}}
\put(6,0){\line(0,1){6}}
\multiput(1,0)(1,0){5}{\line(0,1){0.1}}
\multiput(0,1)(0,1){5}{\line(1,0){0.1}}

\put(0,4){\makebox(0,0){$\scriptstyle \times$}}
\put(0,3){\makebox(0,0){$\scriptstyle \times$}}
\put(3,3){\makebox(0,0){$\scriptstyle \times$}}
\put(3,0){\makebox(0,0){$\scriptstyle \times$}}
\put(5,0){\makebox(0,0){$\scriptstyle \times$}}
\put(0,5){\makebox(0,0){$\scriptstyle \bullet$}}
\put(1,5){\makebox(0,0){$\scriptstyle \bullet$}}
\put(1,4){\makebox(0,0){$\scriptstyle \bullet$}}
\put(3,4){\makebox(0,0){$\scriptstyle \bullet$}}
\put(3,1){\makebox(0,0){$\scriptstyle \bullet$}}
\put(4,1){\makebox(0,0){$\scriptscriptstyle \bullet$}}
\put(4,0){\makebox(0,0){$\scriptstyle \bullet$}}
\put(3,6){\makebox(0,0){$\scriptstyle \circ$}}
\put(3,3){\makebox(0,0){{\large $\circ$}}}
\put(4,3){\makebox(0,0){$\scriptstyle \circ$}}
\put(4,1){\makebox(0,0){{\large $\circ$}}}
\put(5,1){\makebox(0,0){$\scriptstyle \circ$}}
\put(5,0){\makebox(0,0){{\large $\circ$}}}

\put(0,5){\line(1,0){1}}
\put(1,5){\line(0,-1){1}}
\put(1,4){\line(1,0){2}}
\put(3,6){\line(0,-1){6}}
\put(0,3){\line(1,0){4}}
\put(4,3){\line(0,-1){3}}
\put(3,1){\line(1,0){2}}
\put(5,1){\line(0,-1){1}}
\put(3,-1){\makebox(0,0){$\scriptstyle W_3,\ \beta_3=2$}}

\end{picture}
\]

\[
\begin{picture}(6,6)
\put(0,0){\line(1,0){6}}
\put(0,6){\line(1,0){6}}
\put(0,0){\line(0,1){6}}
\put(6,0){\line(0,1){6}}
\multiput(1,0)(1,0){5}{\line(0,1){0.1}}
\multiput(0,1)(0,1){5}{\line(1,0){0.1}}

\put(0,2){\makebox(0,0){$\scriptstyle \times$}}
\put(1,2){\makebox(0,0){$\scriptstyle \times$}}
\put(1,1){\makebox(0,0){$\scriptstyle \times$}}
\put(3,1){\makebox(0,0){$\scriptstyle \times$}}
\put(3,0){\makebox(0,0){$\scriptstyle \times$}}
\put(0,3){\makebox(0,0){$\scriptstyle \bullet$}}
\put(2,3){\makebox(0,0){$\scriptstyle \bullet$}}
\put(2,1){\makebox(0,0){$\scriptstyle \bullet$}}
\put(3,1){\makebox(0,0){$\scriptscriptstyle \bullet$}}
\put(3,0){\makebox(0,0){$\scriptstyle \bullet$}}
\put(0,5){\makebox(0,0){$\scriptstyle \circ$}}
\put(1,5){\makebox(0,0){$\scriptstyle \circ$}}
\put(1,4){\makebox(0,0){$\scriptstyle \circ$}}
\put(3,4){\makebox(0,0){$\scriptstyle \circ$}}
\put(3,1){\makebox(0,0){{\large $\circ$}}}
\put(4,1){\makebox(0,0){$\scriptstyle \circ$}}
\put(4,0){\makebox(0,0){$\scriptstyle \circ$}}

\put(0,5){\line(1,0){1}}
\put(1,5){\line(0,-1){1}}
\put(1,4){\line(1,0){2}}
\put(3,4){\line(0,-1){4}}
\put(0,3){\line(1,0){2}}
\put(2,3){\line(0,-1){2}}
\put(0,2){\line(1,0){1}}
\put(1,2){\line(0,-1){1}}
\put(1,1){\line(1,0){3}}
\put(4,1){\line(0,-1){1}}
\put(3,-1){\makebox(0,0){$\scriptstyle W_4,\ \beta_4=2$}}

\end{picture}
\kern1cm
\begin{picture}(6,6)
\put(0,0){\line(1,0){6}}
\put(0,6){\line(1,0){6}}
\put(0,0){\line(0,1){6}}
\put(6,0){\line(0,1){6}}
\multiput(1,0)(1,0){5}{\line(0,1){0.1}}
\multiput(0,1)(0,1){5}{\line(1,0){0.1}}

\put(0,2){\makebox(0,0){$\scriptstyle \bullet$}}
\put(1,2){\makebox(0,0){$\scriptstyle \bullet$}}
\put(1,0){\makebox(0,0){$\scriptstyle \bullet$}}
\put(0,3){\makebox(0,0){$\scriptstyle \circ$}}
\put(2,3){\makebox(0,0){$\scriptstyle \circ$}}
\put(2,1){\makebox(0,0){$\scriptstyle \circ$}}
\put(3,1){\makebox(0,0){$\scriptstyle \circ$}}
\put(3,0){\makebox(0,0){$\scriptstyle \circ$}}

\put(0,3){\line(1,0){2}}
\put(2,3){\line(0,-1){2}}
\put(2,1){\line(1,0){1}}
\put(3,1){\line(0,-1){1}}
\put(0,2){\line(1,0){1}}
\put(1,2){\line(0,-1){2}}
\put(3,-1){\makebox(0,0){$\scriptstyle W_5,\ \beta_5=2$}}

\end{picture}
\kern1cm
\begin{picture}(6,6)
\put(0,0){\line(1,0){6}}
\put(0,6){\line(1,0){6}}
\put(0,0){\line(0,1){6}}
\put(6,0){\line(0,1){6}}
\multiput(1,0)(1,0){5}{\line(0,1){0.1}}
\multiput(0,1)(0,1){5}{\line(1,0){0.1}}

\put(0,1){\makebox(0,0){$\scriptstyle \bullet$}}
\put(1,1){\makebox(0,0){$\scriptstyle \bullet$}}
\put(1,0){\makebox(0,0){$\scriptscriptstyle \bullet$}}
\put(0,2){\makebox(0,0){$\scriptstyle \circ$}}
\put(1,2){\makebox(0,0){$\scriptstyle \circ$}}
\put(1,0){\makebox(0,0){{\large $\circ$}}}

\put(0,2){\line(1,0){1}}
\put(1,2){\line(0,-1){2}}
\put(0,1){\line(1,0){1}}
\put(3,-1){\makebox(0,0){$\scriptstyle W_6,\ \beta_6=2$}}

\end{picture}
\kern1cm
\begin{picture}(6,6)

\end{picture}
\]

\[
\begin{picture}(3,1)
\put(0,1){\makebox(0,0){$\scriptstyle W_i:$}}
\put(1,1){\makebox(0,0){$\scriptstyle \bullet$}}
\put(2,1){\makebox(0,0){$\scriptstyle \bullet$}}
\put(1,1){\line(1,0){1}}
\end{picture}
\kern 5mm
\begin{picture}(3,1)
\put(0,1){\makebox(0,0){$\scriptstyle X_i':$}}
\put(1,1){\makebox(0,0){$\scriptstyle \times$}}
\put(2,1){\makebox(0,0){$\scriptstyle \times$}}
\put(1,1){\line(1,0){1}}
\end{picture}
\kern 5mm
\begin{picture}(3,1)
\put(0,1){\makebox(0,0){$\scriptstyle Y_i':$}}
\put(1,1){\makebox(0,0){$\scriptstyle \circ$}}
\put(2,1){\makebox(0,0){$\scriptstyle \circ$}}
\put(1,1){\line(1,0){1}}
\end{picture}
\]

\[
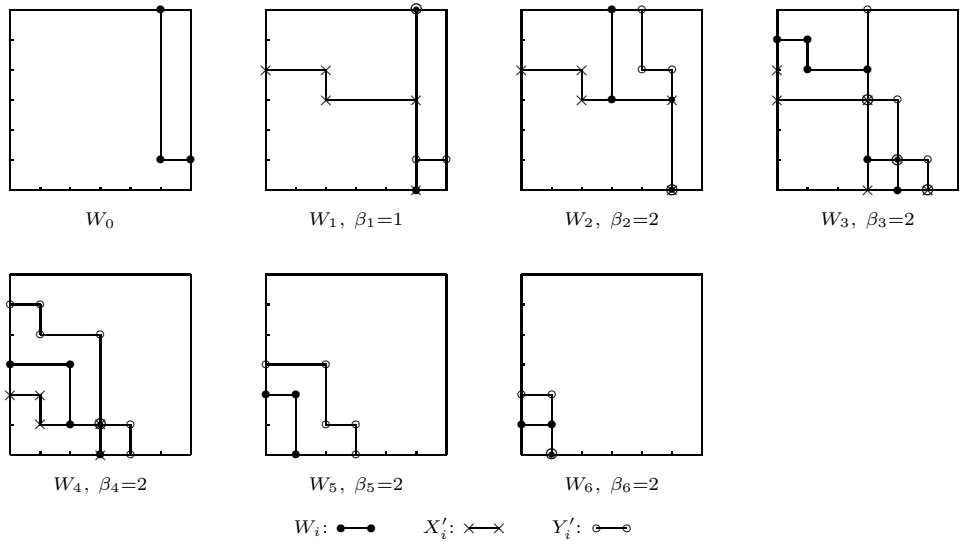
\text{Figure 20. An example of forward slicing}
\]
%%%%%%%%%%%%%%%% End Figure %%%%%%%%%%%%%%%%%%%%%%%%%%%%%%%%%%%%%%

%%%%%%%%%%%% Figure 21 %%%%%%%%%%%%%%%%%%%%%%%%%%%%%%%
\vskip5mm
\setlength{\unitlength}{8mm}
\[
\begin{picture}(11,11)
\put(9,4){\vector(1,0){2}}
\put(4,10){\vector(0,1){1}}
\put(1,1){\vector(-1,-1){1}}
\put(11,3.8){\makebox(0,0)[t]{$\scriptstyle x$}}
\put(4.1,11){\makebox(0,0)[l]{$\scriptstyle y$}}
\put(0,0.2){\makebox(0,0)[b]{$\scriptstyle z$}}
\put(10,3.8){\makebox(0,0)[t]{$\scriptstyle 6$}}
\put(4.1,10.2){\makebox(0,0)[bl]{$\scriptstyle 6$}}
\put(0.9,1){\makebox(0,0)[br]{$\scriptstyle 6$}}

\put(4,10){\line(1,0){5}}
\put(3,9){\line(1,0){3}}
\put(3,8){\line(1,0){1}}
\put(2.5,7.5){\line(1,0){1}}
\put(2.5,5.5){\line(1,0){2}}
\put(2,5){\line(1,0){2}}
\put(2,4){\line(1,0){1}}
\put(1.5,3.5){\line(1,0){1}}
\put(1.5,2.5){\line(1,0){1}}
\put(6.5,9.5){\line(1,0){2}}
\put(6.5,6.5){\line(1,0){2}}
\put(6,6){\line(1,0){2}}
\put(4,7){\line(1,0){2}}
\put(3.5,6.5){\line(1,0){2}}
\put(4.5,3.5){\line(1,0){2}}
\put(3,2){\line(1,0){2}}
\put(9,5){\line(1,0){1}}
\put(1,1){\line(1,0){1}}
\put(1,2){\line(1,0){1}}
\put(4,3){\line(1,0){1}}
\put(5.5,2.5){\line(1,0){1}}
\put(6,4){\line(1,0){1}}
\put(7,3){\line(1,0){1}}

\put(3,9){\line(0,-1){1}}
\put(2.5,7.5){\line(0,-1){2}}
\put(3.5,7.5){\line(0,-1){1}}
\put(4,8){\line(0,-1){1}}
\put(6,6){\line(0,-1){2}}
\put(5.5,6.5){\line(0,-1){4}}
\put(6.5,9.5){\line(0,-1){3}}
\put(8.5,9.5){\line(0,-1){3}}
\put(9,10){\line(0,-1){6}}
\put(10,5){\line(0,-1){1}}
\put(8,6){\line(0,-1){3}}
\put(2,5){\line(0,-1){1}}
\put(1.5,3.5){\line(0,-1){1}}
\put(1,2){\line(0,-1){1}}
\put(2,2){\line(0,-1){1}}
\put(2.5,3.5){\line(0,-1){1}}
\put(3,4){\line(0,-1){2}}
\put(4,5){\line(0,-1){2}}
\put(4.5,5.5){\line(0,-1){2}}
\put(6,9){\line(0,-1){2}}
\put(5,3){\line(0,-1){1}}
\put(6.5,3.5){\line(0,-1){1}}
\put(7,4){\line(0,-1){1}}

\put(4,10){\line(-1,-1){1}}
\put(9,10){\line(-1,-1){0.5}}
\put(6.5,9.5){\line(-1,-1){0.5}}
\put(6,7){\line(-1,-1){0.5}}
\put(6,4){\line(-1,-1){1}}
\put(9,4){\line(-1,-1){1}}
\put(3,8){\line(-1,-1){0.5}}
\put(4,8){\line(-1,-1){0.5}}
\put(4,7){\line(-1,-1){0.5}}
\put(2.5,5.5){\line(-1,-1){0.5}}
\put(4.5,5.5){\line(-1,-1){0.5}}
\put(4.5,3.5){\line(-1,-1){0.5}}
\put(2,4){\line(-1,-1){0.5}}
\put(3,4){\line(-1,-1){0.5}}
\put(1.5,2.5){\line(-1,-1){0.5}}
\put(2.5,2.5){\line(-1,-1){0.5}}
\put(3,2){\line(-1,-1){1}}
\put(6.5,6.5){\line(-1,-1){0.5}}
\put(8.5,6.5){\line(-1,-1){0.5}}
\put(5.5,2.5){\line(-1,-1){0.5}}
\put(7,3){\line(-1,-1){0.5}}
\put(7,4){\line(-1,-1){0.5}}

\end{picture}
\]
\[
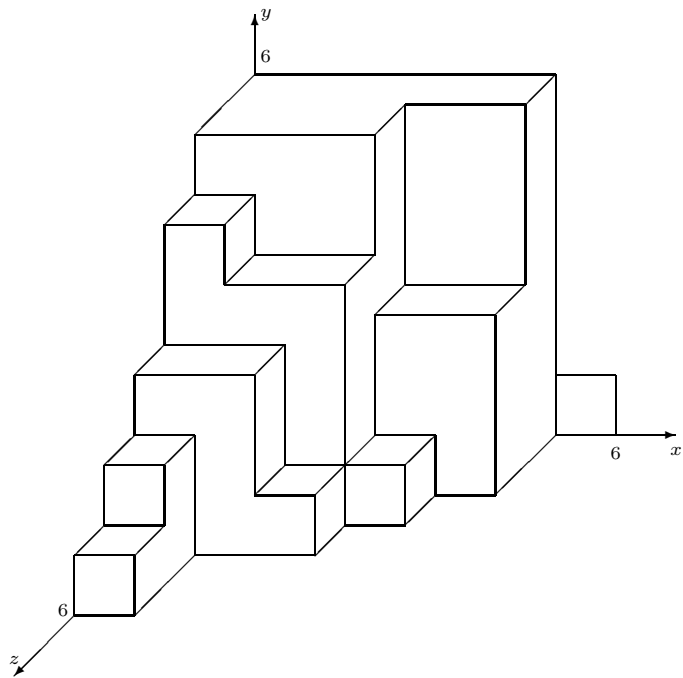
\text{Figure 21. The ideal $\textstyle I=\bigcup_{j=0}^6\bigl(\iota(W_j)\times\{j\}\bigr)$}
\]
%%%%%%%%%%%%%%%% End Figure %%%%%%%%%%%%%%%%%%%%%%%%%%%%%%%%%%%%%%

%%%%%%%%%%%%%%%%%%%%%%%%%%%%%%%%%%%%%%%%%%%%
%    section 5
%%%%%%%%%%%%%%%%%%%%%%%%%%%%%%%%%%%%%%%%%%%%
\section{Enumerating $A$-Invariant Ideals of ${\mathcal U}$}

In this section, we consider the case $r=1$. 
Therefore, we are interested in ideals of $\mathcal U$ which are invariant (symmetric) under the action of $A$. The problem here is more difficult than the one in Section 4.

In order to enumerate the  $A$-invariant ideals of ${\mathcal U}$, we partition ${\mathcal U}$ as
\[
{\mathcal U}=\bigcup_{i=0}^n{\mathcal V}_i
\]
where
\[
{\mathcal V}_i=\bigl\{
(x,y,z)\in{\mathcal U}:x\le i,\,y\le i,\,z\le i\ \text{and at least
one of $x,y,z$ is $i$}\bigr\}.
\]
For any subset $X\subset {\Bbb R}^3$, we denote its image under $A$,
i.e., $\{xA:x\in X\}$, by $X^A$. Put
\[
V_i=[0,i]^2\times\{i\}.
\]  
Then 
\[
{\mathcal V}_i=V_i\cup V_i^A\cup V_i^{A^2}.
\]

Let $I$ be an $A$-invariant ideal of ${\mathcal V}_i$. Write
\[
I\cap V_i=J\times\{i\}.
\]
Then $J$ is an ideal of $[0,i]^2$ such that 
\[
I= (J\times\{i\})\cup(J\times\{i\})^A\cup (J\times\{i\})^{A^2}
\]
and
\begin{equation}\label{x:y:}
\{x:(x,i)\in J\}=\{y:(i,y)\in J\}.
\end{equation}
On the other hand, if $J$ is any subset of $[0,i]^2$ satisfying (\ref{x:y:}), then
the $A$-invariant subset $I= (J\times\{i\})\cup(J\times\{i\})^A\cup (J\times\{i\})^{A^2}\subset {\mathcal V}_i$ has the property that $I\cap V_i=J\times\{i\}$.

Let $J_j$ ($0\le j\le i$) be an ideal of $[0,j]^2$ such that
\begin{equation}\label{x:}
\{x:(x,j)\in J_j\}=\{y:(j,y)\in J_j\}.
\end{equation}
We call the sequence $J_0,\dots,J_{i-1}$ {\em consistent} if
\[
\bigcup_{j=0}^{i-1}\Bigl[(J_j\times\{j\})\cup(J_j\times\{j\})^A\cup (J_j\times\{j\})^{A^2}\Bigr]
\]
is an $A$-invariant ideal of $[0,i-1]^3$. The ideal $J_i$ of
$[0,i]^2$ is said to be {\em consistent} with $J_0,\dots,J_{i-1}$ if
the sequence $J_0,\dots,J_{i-1},J_i$ is consistent. Note that the meaning
of consistency here is different from that of Section 4.

\medskip
\noindent{\bf Note.}
In the terminology of Section 2, the statement that $J_i$ is {\em consistent} with $J_0,\dots,j_{i-1}$ means that $\bigcup_{s=0}^2(J_i\times\{i\})^{A^s}$ is {\em compatible} with $\bigcup_{s=0}^2(J_j\times\{j\})^{A^s}$, $0\le j<i$, with respect to the partition $\mathcal U=\bigcup_{j=0}^n\mathcal V_j$.

\medskip

Given a consistent sequence of ideals $J_0,\dots,J_{i-1}$ and an ideal $J_i$
of $[0,i]^2$. Our goal in this section, roughly speaking, is to determined two walks 
$\Phi_i$ and $\Psi_i$ in $[0,i]^2$ such that $J_i$ is consistent with
$J_0,\dots,J_{i-1}$ if and only if $\Phi_i\le\omega(J_i)\le \Psi_i$.

\begin{lem}\label {L5.1} 
Let $0\le i\le n$. 
Let $J_j$ ($0\le j\le i$) be an ideal of $[0,j]^2$ such that $J_0,\dots,J_{i-1}$ is a consistent sequence. 
Write
\begin{equation}\label{bigcup}
\bigcup_{j=0}^{i-1}\Bigl[(J_j\times\{j\})\cup(J_j\times\{j\})^A\cup (J_j\times\{j\})^{A^2}\Bigr]
=\bigcup_{j=0}^{i-1}\bigl(J_{i,j}\times\{j\}\bigr),
\end{equation}
where $J_{i,j}$ $(0\le j\le i-1$) is a ideal of $[0,i-1]^2$, and write
\[
\omega(J_i)=\bigl((x_0,y_0),\dots,(x_k,y_k)\bigr).
\]
Then $J_i$ is consistent with $J_0,\dots,J_{i-1}$ if and only if the following conditions are satisfied:

\begin{equation}\label{5-1}
(x_0,y_0)=(y_k,x_k)\quad\text{if}\ y_0=i.
\end{equation} 

\begin{equation}\label{5-2}
\bigl(J_i\times\{i\}+\Delta\bigr)\cap\bigl([0,i-1]^2\times\{j\}\bigr)
\subset J_{i,j}\times\{j\}\quad\text{for all}\ 0\le j<i.
\end{equation}

\begin{equation}\label{5-3}
\bigl(J_{i,j}\times\{j\}+\Delta\bigr)\cap V_i
\subset J_i\times\{i\}\quad\text{for all}\ 0\le j<i.
\end{equation}

\begin{equation}\label{5-4}
\bigl(J_i\times\{i\}+\Delta\bigr)\cap V_i^A
\subset\bigl(J_i\times\{i\}\bigr)^A.
\end{equation}

\begin{equation}\label{5-5}
\bigl(J_i\times\{i\}+\Delta\bigr)\cap V_i^{A^2}
\subset\bigl(J_i\times\{i\}\bigr)^{A^2}.
\end{equation}
\end{lem}

\begin{proof}
Let $I=(J_i\times\{i\})\cup(J_i\times\{i\})^A\cup(J_i\times\{i\})^{A^2}$
and denote by $I'$ the ideal of $[0,i-1]^3$ in (\ref{bigcup}).

($\Rightarrow$) Equation (\ref{5-1}) follows from (\ref{x:}).
Since $J_0,\dots,J_{i-1},J_i$ is a consistent sequence of ideals, $I\cup I'$ is an ideal of $[0,i]^3$.
By Lemma~\ref{L2.1} (ii), we have
\[
\begin{cases}
(I\cap V_i+\Delta)\cap[0,i-1]^3\subset I',\cr
(I'+\Delta)\cap V_i\subset I\cap V_i,\cr
(I\cap V_i+\Delta)\cap V_i^A\subset I\cap V_i^A,\cr
(I\cap V_i+\Delta)\cap V_i^{A^2}\subset I\cap V_i^{A^2}.\cr
\end{cases}
\]
These inclusions are equivalent to (\ref{5-2}) -- (\ref{5-5}) respectively.

($\Leftarrow$) First, from (\ref{5-1}), we have
\[
\{x:(x,i)\in J_i\}=\{y:(i,y)\in J_i\}.
\]
Thus (cf. the statement after (\ref{x:y:})),
\begin{equation}\label{I3}
I\cap V_i=J_i\times\{i\}.
\end{equation}

From (\ref{I3}), (\ref{5-4}), and (\ref{5-5}), we have
\[
(I\cap V_i+\Delta)\cap V_i^{A^k}\subset I\cap V_i^{A^k},\quad k=0,1,2.
\]
Hence 
\[
(I\cap V_i+\Delta)\cap{\mathcal V}_i\subset I.
\]
Since $I$ is $A$-invariant, we have
\[
(I+\Delta)\cap{\mathcal V}_i\subset I,
\]
which means that $I$ is an ideal of ${\mathcal V}_i$.

From (\ref{I3}), (\ref{5-2}), and (\ref{5-3}), we have
\[
\begin{cases}
(I\cap V_i+\Delta)\cap[0,i-1]^3\subset I',\cr
(I'+\Delta)\cap V_i\subset I\cap V_i.\cr
\end{cases}
\]
Since both $I$ and $I'$ are $A$-invariant, we obtain
\[
\begin{cases}
(I\cap V_i^{A^k}+\Delta)\cap[0,i-1]^3\subset I',\cr
(I'+\Delta)\cap V_i^{A^k}\subset I\cap V_i^{A^k},\cr
\end{cases}
\quad\quad k=0,1,2.
\]
Therefore,
\begin{equation}\label{5-6}
\begin{cases} 
(I+\Delta)\cap[0,i-1]^3\subset I',\cr
(I'+\Delta)\cap{\mathcal V}_i\subset I.\cr
\end{cases}
\end{equation}
By (\ref{5-6}) and Lemma~\ref{L2.1} (ii), $I\cup I'$ is an ideal of $[0,i]^3$, i.e., $J_0,\dots,J_{i-1},J_i$ is consistent.
\end{proof}

\begin{lem}\label{L5.2}
In Lemma~\ref{L5.1}, {\rm (\ref{5-1}) -- (\ref{5-3})} imply 
{\rm (\ref{5-4})}.
\end{lem}

\begin{proof}
First assume $i<p$. In this case, the partial order $\prec$ in $[0,i]^3$
is the cartesian product of linear orders and (\ref{5-4}) follows from
(\ref{5-1}) trivially. (See Figure 22.)

%%%%%%%%%%%% Figure  %%%%%%%%%%%%%%%%%%%%%%%%%%%%%%%
\vskip5mm
\setlength{\unitlength}{7.2mm}
\[
\begin{picture}(13,13)
\put(11,5){\vector(1,0){2}}
\put(5,11){\vector(0,1){2}}
\put(2,2){\vector(-1,-1){2}}

\put(2,2){\line(1,0){6}}
\put(2,2){\line(0,1){6}}
\put(2,8){\line(1,0){6}}
\put(8,2){\line(0,1){6}}
\put(5,11){\line(1,0){6}}
\put(11,5){\line(0,1){6}}
\put(2,8){\line(1,1){3}}
\put(8,8){\line(1,1){3}}
\put(4,8){\line(1,1){3}}
\put(8,2){\line(1,1){3}}
\put(4,8){\line(0,-1){2}}
\put(4,6){\line(1,0){1}}
\put(5,6){\line(0,-1){1}}
\put(5,5){\line(1,0){2}}
\put(7,5){\line(0,-1){1}}
\put(7,4){\line(1,0){1}}
\put(4,8){\line(1,1){0.5}}
\put(4.5,8.5){\line(1,0){1}}
\put(5.5,8.5){\line(1,1){1}}
\put(6.5,9.5){\line(1,0){1}}
\put(7.5,9.5){\line(1,1){0.5}}
\put(8,10){\line(1,0){2}}

\put(13,4.8){\makebox(0,0)[t]{$\scriptstyle x$}}
\put(5.2,13){\makebox(0,0)[l]{$\scriptstyle y$}}
\put(0,0.5){\makebox(0,0){$\scriptstyle z$}}

\put(7,11){\makebox(0,0){$\scriptstyle \bullet$}}
\put(4,8){\makebox(0,0){$\scriptscriptstyle \bullet$}}
\put(4,8){\makebox(0,0){{\large $\circ$}}}
\put(4.5,8.5){\makebox(0,0){$\scriptstyle \circ$}}
\put(5.5,8.5){\makebox(0,0){$\scriptstyle \circ$}}
\put(6.5,9.5){\makebox(0,0){$\scriptstyle \circ$}}
\put(7.5,9.5){\makebox(0,0){$\scriptstyle \circ$}}
\put(8,10){\makebox(0,0){$\scriptstyle \circ$}}
\put(10,10){\makebox(0,0){$\scriptstyle \circ$}}
\put(4,8){\makebox(0,0){$\scriptstyle \times$}}
\put(4,6){\makebox(0,0){$\scriptstyle \times$}}
\put(5,6){\makebox(0,0){$\scriptstyle \times$}}
\put(5,5){\makebox(0,0){$\scriptstyle \times$}}
\put(7,5){\makebox(0,0){$\scriptstyle \times$}}
\put(7,4){\makebox(0,0){$\scriptstyle \times$}}
\put(8,4){\makebox(0,0){$\scriptstyle \times$}}

\put(1,-1){\makebox(0,0){$\scriptstyle \times$}}
\put(2,-1){\makebox(0,0){$\scriptstyle \times$}}
\put(1,-1){\line(1,0){1}}
\put(2.5,-1){\makebox(0,0)[l]{$\scriptstyle :\ \text{boundary of}\ J_i\times\{i\}$}}
\put(1,-1.5){\makebox(0,0){$\scriptstyle \circ$}}
\put(2,-1.5){\makebox(0,0){$\scriptstyle \circ$}}
\put(1,-1.5){\line(1,0){1}}
\put(2.5,-1.5){\makebox(0,0)[l]{$\scriptstyle :\ \text{boundary of}\ (J_i\times\{i\})^A$}}
\put(1,-2){\makebox(0,0){$\scriptstyle \bullet$}}
\put(2,-2){\makebox(0,0){$\scriptstyle \bullet$}}
\put(1,-2){\line(1,0){1}}
\put(2.5,-2){\makebox(0,0)[l]{$\scriptstyle :\ \text{boundary of}\ (J_i\times\{i\}+\Delta)\cap V_i^A$}}

\end{picture}
\]
\vskip1.5cm
\[
\text{Figure 22. $(J_i\times\{i\}+\Delta)\cap V_i^A\subset
(J_i\times\{i\})^A$ when $i<p$}
\]
%%%%%%%%%%%%%%%% End Figure %%%%%%%%%%%%%%%%%%%%%%%%%%%%%%%%%%%%%%
\vskip5mm

So we assume that $i\ge p$. Let $I$ and $I'$ be as in the proof of
Lemma~\ref{L5.1}. Note that $I\cap V_i=J_i\times\{i\}$ by (\ref{5-1}).

For each $u=(x',i,z')\in\bigl(J_i\times\{i\}+\Delta\bigr)\cap V_i^A$,
we want to show that $u\in(J_i\times\{i\})^A$. Note that there exists
$(x,y)\in J_i$ such that $u\prec(x,y,i)$.

If $y=i$, then $(x,y,i)\in(J_i\times\{i\})\cap V_i^A=I\cap V_i\cap V_i^A
\subset(I\cap V_i)^A=(J_i\times\{i\})^A$. Since $(J_i\times\{i\})^A$ is an
ideal of $V_i^A$, we have $u\in(J_i\times\{i\})^A$. Thus we assume $y<i$.

If $z'=i$, then $(x',i,i)\prec(x,y,i)$ implies $(x',i)\prec(x,y)$, hence
$(x',i)\in J_i$. By (\ref{5-1}), $(i,x')\in J_i$, hence $u=(x',i,i)=
(i,x',i)A\in(J_i\times\{i\})^A$. Thus we assume $z'<i$.

By (\ref{Iff4}), $(x',i,z')\prec(x,y,i)$ if and only if
\[
(x',i,z')\prec(x,i,i)-(i-y)(p,0,0)=(x-(i-y)p, i,i).
\]
Thus we have
\[
\begin{split}
(x',i,z')\,&\prec (x',i,z')+(p,-1,0)\cr
&\prec(x-(i-y)p, i,i)+(p,-1,0)\cr
&=(x-(i-1-y)p, i-1, i)\cr
&=(x,y,i)+(i-1-y)(-p,1,0)\cr
&\prec(x,y,i),
\end{split}
\]
i.e.,
\begin{equation}\label{5-a}
(x',i,z')\prec(x'+p,i-1,z')\prec(x,y,i).
\end{equation}

If $(z',x')\prec(x,y)$, then $(z',x')\in J_i$. Thus $(x',i,z')=(z',x',i)A
\in(J_i\times\{i\})^A$. Therefore, we assume $(z',x')\not\prec (x,y)$.

We claim that
\begin{equation}\label{x'}
x'<i-p.
\end{equation}
In fact, since $(x',i,z')\prec(x-(i-y)p,i,i)$, we have $(z',x')\prec
(i,x-(i-y)p)$, i.e.,
\begin{equation}\label{5-A}
x'\le x-(i-y)p+\frac 1p(i-z').
\end{equation}
If $z'>x$, (\ref{5-A}) gives
\[
\begin{split}
x'\,&<x-(i-y)p+\frac 1p(i-x)\cr
&=\frac{p-1}px+py-\frac{p^2-1}p i\cr
&\le\frac{p-1}p i+p(i-1)-\frac{p^2-1}p i\cr
&=i-p.\cr
\end{split}
\]
If $z'\le x$, since $(z',x')\not\prec(x,y)$, we must have
\begin{equation}\label{5-B}
x'>y+\frac 1p(x-z').
\end{equation}
Combining (\ref{5-A}) and (\ref{5-B}), we have
\[
x-(i-y)p+\frac 1p (i-z')>y+\frac 1p(x-z')
\]
which gives
\[
(p-1)y>\frac1p x-x+pi-\frac1p i=\frac{p^2-1}pi-\frac{p-1}px,
\]
i.e.,
\[
y>\frac{p+1}pi-\frac1p x\ge i,
\]
which is a contradiction. Thus (\ref{x'}) is proved.

Now we have $(x'+p,i-1,z')\prec(x,y,i)$ and $(x'+p,i-1,z')\in[0,i-1]^3$.
By (\ref{5-2}), $(x'+p,i-1,z')\in(J_i\times\{i\}+\Delta)\cap[0,i-1]^3\subset I'$. Thus we have
\[
\begin{array}{rll}
(x',i,z')\kern-3mm&\in(I'+\Delta)\cap V_i^A\quad&\text{(by (\ref{5-a}))}\cr
&\subset (J_i\times\{i\})^A&\text{(by (\ref{5-3}) and the $A$-symmetry of $I'$).}
\end{array}
\]
\end{proof}

\begin{lem}\label{L5.3}
Assume that in Lemma~\ref{L5.1}, {\rm (\ref{5-1}) -- (\ref{5-4})} are satisfied. Then {\rm (\ref{5-5})} is equivalent to
\begin{equation}\label{5-13}
\max\bigl\{y:(i-1,y)\in J_i\bigr\}\le\max\bigl\{x:(x,i-p)\in J_i\bigr\}
\quad\text{if}\ i\ge p.
\end{equation}
\end{lem}

\begin{proof}
First assume $i<p$. Then (\ref{5-13}) is satisfied without instance. Since in case, the partial order $\prec$ in $[0,i]^3$ is the cartesian product of linear orders, (\ref{5-5}) holds trivially. (See Figure 23.) So we assume
that $i\ge p$. Again, let $I$ and $I'$ be as in the proof of Lemma~\ref{L5.1}.

%%%%%%%%%%%% Figure 23 %%%%%%%%%%%%%%%%%%%%%%%%%%%%%%%
\setlength{\unitlength}{7.2mm}
\[
\begin{picture}(13,13)
\put(11,5){\vector(1,0){2}}
\put(5,11){\vector(0,1){2}}
\put(2,2){\vector(-1,-1){2}}

\put(2,2){\line(1,0){6}}
\put(2,2){\line(0,1){6}}
\put(2,8){\line(1,0){6}}
\put(8,2){\line(0,1){6}}
\put(5,11){\line(1,0){6}}
\put(11,5){\line(0,1){6}}
\put(2,8){\line(1,1){3}}
\put(8,8){\line(1,1){3}}
\put(8,4){\line(1,1){3}}
\put(8,2){\line(1,1){3}}
\put(4,8){\line(0,-1){2}}
\put(4,6){\line(1,0){1}}
\put(5,6){\line(0,-1){1}}
\put(5,5){\line(1,0){2}}
\put(7,5){\line(0,-1){1}}
\put(7,4){\line(1,0){1}}
\put(8,4){\line(1,1){1}}
\put(9,5){\line(0,1){1}}
\put(9,6){\line(1,1){0.5}}
\put(9.5,6.5){\line(0,1){2}}
\put(9.5,8.5){\line(1,1){0.5}}
\put(10,9){\line(0,1){1}}

\put(13,4.8){\makebox(0,0)[t]{$\scriptstyle x$}}
\put(5.2,13){\makebox(0,0)[l]{$\scriptstyle y$}}
\put(0,0.5){\makebox(0,0){$\scriptstyle z$}}

\put(11,7){\makebox(0,0){$\scriptstyle \bullet$}}
\put(8,4){\makebox(0,0){$\scriptscriptstyle \bullet$}}
\put(8,4){\makebox(0,0){{\large $\circ$}}}
\put(9,5){\makebox(0,0){$\scriptstyle \circ$}}
\put(9,6){\makebox(0,0){$\scriptstyle \circ$}}
\put(9.5,6.5){\makebox(0,0){$\scriptstyle \circ$}}
\put(9.5,8.5){\makebox(0,0){$\scriptstyle \circ$}}
\put(10,9){\makebox(0,0){$\scriptstyle \circ$}}
\put(10,10){\makebox(0,0){$\scriptstyle \circ$}}
\put(4,8){\makebox(0,0){$\scriptstyle \times$}}
\put(4,6){\makebox(0,0){$\scriptstyle \times$}}
\put(5,6){\makebox(0,0){$\scriptstyle \times$}}
\put(5,5){\makebox(0,0){$\scriptstyle \times$}}
\put(7,5){\makebox(0,0){$\scriptstyle \times$}}
\put(7,4){\makebox(0,0){$\scriptstyle \times$}}
\put(8,4){\makebox(0,0){$\scriptstyle \times$}}

\put(1,0){\makebox(0,0){$\scriptstyle \times$}}
\put(2,0){\makebox(0,0){$\scriptstyle \times$}}
\put(1,0){\line(1,0){1}}
\put(2.5,0){\makebox(0,0)[l]{$\scriptstyle :\ \text{boundary of}\ J_i\times\{i\}$}}
\put(1,-0.5){\makebox(0,0){$\scriptstyle \circ$}}
\put(2,-0.5){\makebox(0,0){$\scriptstyle \circ$}}
\put(1,-0.5){\line(1,0){1}}
\put(2.5,-0.5){\makebox(0,0)[l]{$\scriptstyle :\ \text{boundary of}\ (J_i\times\{i\})^{A^2}$}}
\put(1,-1){\makebox(0,0){$\scriptstyle \bullet$}}
\put(2,-1){\makebox(0,0){$\scriptstyle \bullet$}}
\put(1,-1){\line(1,0){1}}
\put(2.5,-1){\makebox(0,0)[l]{$\scriptstyle :\ \text{boundary of}\ (J_i\times\{i\}+\Delta)\cap V_i^{A^2}$}}

\end{picture}
\]
\vskip1cm
\[ 
\text{Figure 23. $(J_i\times\{i\}+\Delta)\cap V_i^{A^2}\subset
(J_i\times\{i\})^{A^2}$ when $i<p$}
\]
%%%%%%%%%%%%%%%% End Figure %%%%%%%%%%%%%%%%%%%%%%%%%%%%%%%%%%%%%%
\vskip5mm

{\it Proof of} ``(\ref{5-13}) $\Rightarrow$ (\ref{5-5})''. Let $u=(i,y',z')\in(J_i\times\{i\}+\Delta)\cap V_i^{A^2}$. We want to show that
$u\in (J_i\times\{i\})^{A^2}$.

Note that there exists $(x,y)\in J_i$ such that $u\prec(x,y,i)$. Also note that (\ref{5-1}) implies that $I\cap V_i=J_i\times\{i\}$.

If $x=i$, then $(x,y,i)\in(J_i\times\{i\})\cap V_i^{A^2}=I\cap V_i\cap
V_i^{A^2}\subset(I\cap V_i)^{A^2}=(J_i\times\{i\})^{A^2}$. Since 
$(J_i\times\{i\})^{A^2}$ is an ideal of $V_i^{A^2}$, we have $u\in (J_i\times\{i\})^{A^2}$.

Next, assume $x<i-1$. By (\ref{Iff3}), $(i,y',z')\prec(x,y,i)$ if and only if
\[
(i,y',z')\prec(i,y,i)-(i-x)(0,0,p)=(i,y,i-p(i-x)).
\]
Thus we have
\[
\begin{split}
(i,y',z')\,&\prec(i,y,i-p(i-x))\cr
&\prec(i,y,i-p(i-x))+(i-x-1)(-1,0,p)\cr
&=(x+1,y,i-p)\cr
&\prec(x,y,i),\cr
\end{split}
\]
i.e.,
\begin{equation}\label{5-C}
u=(i,y',z')\prec(x+1,y,i-p)\prec(x,y,i).
\end{equation}
If $y=i$, then
\[
\begin{array}{rll}
(x+1,y,i-p)\kern-3mm&\in(J_i\times\{i\}+\Delta)\cap V_i^A\cr
&\subset (J_i\times\{i\})^A&\text{(by (\ref{5-4}))}.
\end{array}
\]
Thus
\[
\begin{array}{rll}
u\kern-3mm&\in\bigl[(J_i\times\{i\})^A+\Delta\bigr]\cap V_i^{A^2}\cr
&\in\bigl[(J_i\times\{i\}+\Delta)\cap V_i^A\bigr]^A\cr
&\subset(J_i\times\{i\})^{A^2}&\text{(by (\ref{5-4}) again)}.
\end{array}
\]
If $y<i$, then $(x+1,y,i-p)\in(J_i\times\{i\}+\Delta)\cap[0,i-1]^3\subset I'$. Hence we have
\[
\begin{array}{rll}
u\kern-3mm&\in(I'+\Delta)\cap V_i^{A^2}\quad&\text{(by (\ref{5-C}))}\cr
&\subset (J_i\times\{i\})^{A^2}&\text{(by (\ref{5-3}) and the $A$-symmetry of $I'$)}.\cr
\end{array}
\]
 
Finally, assume $x=i-1$. By (\ref{Iff3}), we have
\[
\bigl((x,y,i)+\Delta\bigr)\cap V_i^{A^2}=\{i\}\times\bigl[\bigl((y,i-p)+D\bigr)\cap[0,i]^2\bigr].
\]
(See Figure 24.) However, by (\ref{5-13}), $y\le\max\{x:(x,i-p)\in J_i\}$.
Thus $(y,i-p)\in J_i$. Therefore,
\[
\begin{array}{rll}
u\kern-3mm&\in\bigl((x,y,i)+\Delta\bigr)\cap V_i^{A^2}\cr
&=\{i\}\times \bigl[\bigl((y,i-p)+D\bigr)\cap[0,i]^2\bigr]\cr
&\subset\{i\}\times J_i\cr
&=(J_i\times\{i\})^{A^2}.
\end{array}
\]
\vfill\eject

%%%%%%%%%%%% Figure 24 %%%%%%%%%%%%%%%%%%%%%%%%%%%%%%%
\vskip5mm
\setlength{\unitlength}{6mm}
\[
\begin{picture}(9,8)
\put(0,0){\vector(1,0){9}}
\put(0,0){\vector(0,1){8}}
\put(7,0){\line(0,1){7}}
\put(0,7){\line(1,0){7}}
\put(8,2){\line(-2,1){8.5}}
\put(3.5,6){\line(1,-4){2}}
\put(0,5){\line(1,1){0.6}}
\put(0,4){\line(1,1){1.4}}
\put(0,3){\line(1,1){2}}
\put(0,2){\line(1,1){2.6}}
\put(0,1){\line(1,1){3.4}}
\put(0,0){\line(1,1){4}}
\put(1,0){\line(1,1){3.1}}
\put(2,0){\line(1,1){2.3}}
\put(3,0){\line(1,1){1.6}}
\put(4,0){\line(1,1){0.8}}

\put(-0.1,-0.1){\makebox(0,0)[tr]{$\scriptstyle 0$}}
\put(7,-0.1){\makebox(0,0)[t]{$\scriptstyle i$}}
\put(9,-0.1){\makebox(0,0)[t]{$\scriptstyle y$}}
\put(-0.1,7){\makebox(0,0)[r]{$\scriptstyle i$}}
\put(0.1,8){\makebox(0,0)[l]{$\scriptstyle z$}}
\put(4.1,4.1){\makebox(0,0)[bl]{$\scriptstyle (y,i-p)$}}
\put(5.6,-1){\makebox(0,0)[l]{$\scriptstyle \text{slope}=-p^2$}}
\put(8,2.1){\makebox(0,0)[bl]{$\scriptstyle \text{slope}=-\frac 1p$}}

\end{picture}
\]
\vskip1cm
\[
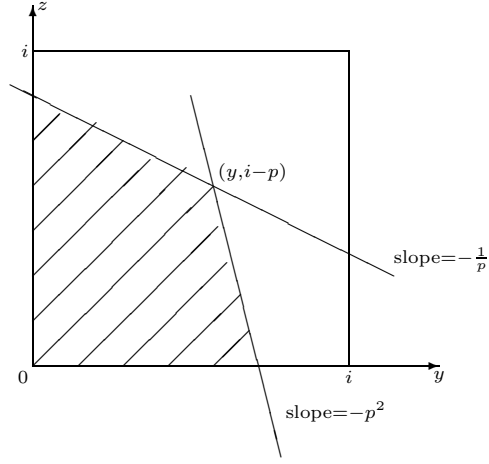
\text{Figure 24. The cross section of $(i-1,y,i)+\Delta$ in $V_i^{A^2}$}
\]
%%%%%%%%%%%%%%%% End Figure %%%%%%%%%%%%%%%%%%%%%%%%%%%%%%%%%%%%%%
\vskip5mm

{\it Proof of} ``(\ref{5-5}) $\Rightarrow$ (\ref{5-13})''. We may assume that
$\{y:(i-1,y)\in J_i\}\ne \emptyset$. Let $\bar y=\max\{y:(i-1,y)\in J_i\}$.
Then $(i-1,\bar y)\in J_i$. Hence
\[
\begin{array}{rll}
&\{i\}\times\bigl[\bigl((\bar y,i-p)+D\bigr)\cap[0,i]^2\bigr]\cr
=\kern-2mm &\bigl((i-1,\bar y,i)+\Delta\bigr)\cap V_i^{A^2}&\text{(by (\ref{Iff3}))}\cr
\subset\kern-2mm &(J_i\times\{i\}+\Delta)\cap V_i^{A^2}\cr
\subset\kern-2mm &(J_i\times\{i\})^{A^2}&\text{(by (\ref{5-5}))}\cr
=\kern-2mm &\{i\}\times J_i.\cr
\end{array}
\]
In particular, $(\bar y, i-p)\in J_i$. Therefore
\[
\bar y\le \max\{x:(x,i-p)\in J_i\},
\]
which is (\ref{5-13}).
\end{proof}

\begin{lem}\label{L5.4}
Let $J$ be an ideal of $[0,i-1]^2$ and $K$ an ideal of $[0,i]^2$. Let $b\ge 0$ be an integer. Then
\begin{equation}\label{L5.5A}
[J+D+(0,-b)]\cap[0,i]^2\subset K
\end{equation}
if and only if
\begin{equation}\label{L5.5B}
[J+D+(0,-b)]\cap[0,i-1]^2\subset K\cap [0,i-1]^2.
\end{equation}
\end{lem}

\begin{proof}
We only have to prove that (\ref{L5.5B}) $\Rightarrow$ (\ref{L5.5A}). Let $(x,y)\in[J+D+(0,-b)]\cap[0,i]^2$, we want to show that $(x,y)\in K$.

If $(x,y)\in[0,i-1]^2$, we are done by (\ref{L5.5B}). So assume $(x,y)\notin[0,i-1]^2$, i.e., $x=i$ or $y=i$.

There exists $(x',y')\in J+(0,-b)$ such that $(x,y)\prec (x',y')$. 
If $y'\ge 0$, then $(x',y')\in [0,i-1]^2$, hence $(x',y')\in[J+D+(0,-b)]\cap [0,i-1]^2\subset
K\cap [0,i-1]^2\subset K$. Therefore $(x,y)\in K$.

If $y'<0$, since $(x,y)\prec (x',y')$, we must have $x<x'$. By the assumption, $y=i$. From Figure 25, we have
\[
(x,i)\prec\bigl(x'-(i-1-y')p,\,i-1\bigr)\prec (x',y')
\]
and
\[
x'-(i-1-y')p\in [x,x']\subset [0,i-1].
\]
Hence $\bigl(x'-(i-1-y')p,\,i-1\bigr)\in [J+D+(0,-b)]\cap [0,i-1]^2\subset
K\cap [0,i-1]^2\subset K$. Therefore, we also have $(x,y)\in K$.
\end{proof}

%%%%%%%%%%%% Figure 25 %%%%%%%%%%%%%%%%%%%%%%%%%%%%%%%
\vskip5mm
\setlength{\unitlength}{5mm}
\[
\begin{picture}(12,5)
\put(0,4){\line(1,0){10}}
\put(2,5){\line(2,-1){10}}
\put(2,4){\makebox(0,0){$\scriptstyle \bullet$}}
\put(6,3){\makebox(0,0){$\scriptstyle \bullet$}}
\put(10,1){\makebox(0,0){$\scriptstyle \bullet$}}
\put(2,3.8){\makebox(0,0)[t]{$\scriptstyle (x,i)$}}
\put(6,3.2){\makebox(0,0)[bl]{$\scriptstyle (x'-(i-1-y')p,\, i-1)$}}
\put(10,1.2){\makebox(0,0)[bl]{$\scriptstyle (x',y')$}}
\put(3,4.5){\makebox(0,0)[bl]{$\scriptstyle \text{slope}=-\frac 1p$}}

\end{picture}
\]
\[
\text{Figure 25. Proof of Lemma~\ref{L5.4}}
\]
%%%%%%%%%%%%%%%% End Figure %%%%%%%%%%%%%%%%%%%%%%%%%%%%%%%%%%%%%%
\vskip3mm

\begin{lem}\label{L5.7}

Let $J$ and $K$ be ideals of $[0,i-1]^2$ where $i\ge p$ and $J\ne[0,i-1]^2$, $J\ne\emptyset$.
Let $b,c\ge 0$ be integers. Let $\hat J$ be the largest ideal of $[0,i]^2$ such that $\hat J\cap[0,i-1]^2=J$ and $\hat K$ the largest ideal of $[0,i]^2$ such that $\hat K\cap[0,i-1]^2=K$. If
\begin{equation}\label{5.7A}
\bigl[J+(b,-c)+D\bigr]\cap[0,i-1]^2\subset K,
\end{equation}
then
\begin{equation}\label{5.7D}
\bigl[\hat J+(b,-c)+D\bigr]\cap[0,i]^2\subset \hat K.
\end{equation}
\end{lem}

\begin{proof}
Let $\omega(J)=W$ and $\omega(K)=Z$. Then $\omega(\hat K)=\overline Z_{[0,i]^2}$,
\[
\omega\bigl(\hat J+(b,-c)\bigr)=\overline{(W+(b,-c))}_{[b,b+i]\times[-c,-c+i]},
\]
and 
\begin{equation}\label{I4}
\begin{split}
\omega\bigl(\bigl[\hat J+(b,-c)+D\bigr]\cap[0,i]^2\Bigr)\,&
=\underline{\omega(\hat J+(b,-c))}_{[0,b+i]\times[-c,i]}\Bigm|_{[0,i]^2}\cr
&=Y|_{[0,i]^2},\cr
\end{split}
\end{equation}
where
\[
Y=
\underline{\bigl[ \overline{(W+(b,-c))}_{[b,b+i]\times[-c,-c+i]}\bigr]}_{[0,b+i]\times[-c,i]}.
\]
Since $J\ne[0,i-1]^2$ and $J\ne\emptyset$, we have $\hat J\ne[0,i]^2$ and $\hat J\ne\emptyset$. Thus $\omega(\hat J+(b,-c))$
is neither $\emptyset$ nor the single point $(b+i,-c+i)$. Since $i\ge p$, $\omega(\hat J+(b,-c))$ is not a single horizontal step. Therefore, the extension from $\omega(\hat J+(b,-c))$ to 
$Y$ requires the same additional steps as the extension from $W+(b,-c)$ to $\underline{(W+(b,-c))}_{[0,b+i-1]\times[-c,i]}$. (See Figure 26.) Thus
$Y$ is the union (in the obvious sense) of
\begin{equation}\label{5.7B}
\underline{(W+(b,-c))}_{[0,b+i-1]\times[-c,i]}\quad\text{and}\quad
\omega(\hat J+(b,-c)).
\end{equation}

By (\ref{5.7A}), we have 
\begin{equation}\label{5.7C}
\underline{(W+(b,-c))}_{[0,b+i-1]\times[-c,i-1]}\bigm|_{[0,i-1]^2}\le Z.
\end{equation}
By (\ref{5.7B}), 
$Y$ is an extension of $\underline{(W+(b,-c))}_{[0,b+i-1]\times[-c,i]}$, hence an extension of
$\underline{(W+(b,-c))}_{[0,b+i-1]\times[-c,i-1]}\bigm|_{[0,i-1]^2}$. Thus (\ref{5.7C}) gives
$Y\le \overline Z_{[0,b+i,]\times[-c,i]}$. Taking restriction on $[0,i]^2$, we have
\[
Y|_{[0,i]^2} \le \overline Z_{[0,i]^2}.
\]
Using (\ref{I4}), we have
\[
\omega\bigl(\bigl[\hat J+(b,-c)+D\bigr]\cap[0,i]^2\Bigr)\le \overline Z_{[0,i]^2}=\omega(\hat K),
\]
which proves (\ref{5.7D}).
\end{proof}

%%%%%%%%%%%% Figure 26 %%%%%%%%%%%%%%%%%%%%%%%%%%%%%%%
\vskip5mm
\setlength{\unitlength}{4mm}
\[
\begin{picture}(15,14)
\put(2,1){\line(1,0){0.1}}
\put(2,8){\line(1,0){0.1}}
\put(0,4){\vector(1,0){15}}
\put(2,0){\vector(0,1){14}}
\put(2,11){\line(1,0){7}}
\put(6,1){\line(1,0){7}}
\put(6,8){\line(1,0){7}}
\put(9,4){\line(0,1){7}}
\put(6,1){\line(0,1){7}}
\put(13,1){\line(0,1){7}}
\put(2,9){\line(1,0){1}}
\put(3,9){\line(0,-1){2}}
\put(3,7){\line(1,0){4}}
\put(7,7){\line(0,-1){1}}
\put(7,6){\line(1,0){3}}
\put(2,6){\line(1,0){2}}
\put(4,6){\line(0,-1){1}}
\put(4,5){\line(1,0){7}}
\put(11,5){\line(0,-1){2}}
\put(11,3){\line(1,0){3}}

\multiput(2,12)(0.5,0){16}{\line(1,0){0.3}}
\multiput(6,9)(0.5,0){16}{\line(1,0){0.3}}
\multiput(13.2,1)(0.5,0){2}{\line(1,0){0.3}}
\multiput(10,12)(0,-0.5){16}{\line(0,-1){0.3}}
\multiput(14,9)(0,-0.5){16}{\line(0,-1){0.3}}
\multiput(6,9)(0,-0.5){2}{\line(0,-1){0.3}}

\put(2,6){\makebox(0,0){$\scriptstyle \bullet$}}
\put(4,6){\makebox(0,0){$\scriptstyle \bullet$}}
\put(4,5){\makebox(0,0){$\scriptstyle \bullet$}}
\put(6,5){\makebox(0,0){$\scriptstyle \bullet$}}
\put(11,5){\makebox(0,0){$\scriptstyle \bullet$}}
\put(11,3){\makebox(0,0){$\scriptstyle \bullet$}}
\put(13,3){\makebox(0,0){$\scriptstyle \bullet$}}
\put(14,3){\makebox(0,0){$\scriptstyle \bullet$}}
\put(2,9){\makebox(0,0){$\scriptstyle \times$}}
\put(3,9){\makebox(0,0){$\scriptstyle \times$}}
\put(3,7){\makebox(0,0){$\scriptstyle \times$}}
\put(7,7){\makebox(0,0){$\scriptstyle \times$}}
\put(7,6){\makebox(0,0){$\scriptstyle \times$}}
\put(9,6){\makebox(0,0){$\scriptstyle \times$}}
\put(10,6){\makebox(0,0){$\scriptstyle \times$}}

\put(1.8,1){\makebox(0,0)[r]{$\scriptstyle -c$}}
\put(1.8,8){\makebox(0,0)[r]{$\scriptstyle i-1-c$}}
\put(1.8,11){\makebox(0,0)[r]{$\scriptstyle i-1$}}
\put(5.8,3.8){\makebox(0,0)[tr]{$\scriptstyle b$}}
\put(9,3.8){\makebox(0,0)[t]{$\scriptstyle i-1$}}

\put(1,-1){\makebox(0,0){$\scriptstyle \times$}}
\put(2,-1){\makebox(0,0){$\scriptstyle \times$}}
\put(1,-1){\line(1,0){1}}
\put(3,-1){\makebox(0,0)[l]{$\scriptstyle :\ \overline Z_{[0,i]^2}$}}
\put(1,-2.5){\makebox(0,0){$\scriptstyle \bullet$}}
\put(2,-2.5){\makebox(0,0){$\scriptstyle \bullet$}}
\put(1,-2.5){\line(1,0){1}}
\put(3,-2.5){\makebox(0,0)[l]{$\scriptstyle :\ Y$}}

\end{picture}
\]
\vskip1cm
\[
\text{Figure 26. Proof of Lemma~\ref{L5.7}}
\]
%%%%%%%%%%%%%%%% End Figure %%%%%%%%%%%%%%%%%%%%%%%%%%%%%%%%%%%%%%
\vskip3mm
\begin{lem}\label{L5.9}
In Lemma~\ref{L5.1}, let $W_{i,j}=\omega(J_{i,j})$ ($0\le j<i$) and 
$W_i=\omega(J_i)$.
Set
\begin{equation}\label{P}
S_i=\underline{(W_{i,i-1}+(0,-p))}_{[0,i]\times[-p,i]}\bigm|_{[0,i]^2}
\end{equation}
and
\begin{equation}\label{Q}
\begin{split}
T_i=\,&\overline{(W_{i,i-1})}_{[0,i]^2}\wedge
\bigl[\overline{(W_{i,i-p})}_{[0,1+i]\times[0,i]}\bigm|_{[1,1+i]\times[0,i]}-(1,0)\bigr]
\wedge\cr
&\bigl[\overline{(W_{i,i-\beta_i})}_{[0,1+i]\times[-p^2+p\beta_i,i]}\bigm|_{[1,1+i]\times[-p^2+p\beta_i,-p^2+p\beta_i+i]}-(1,-p^2+p\beta_i)\bigr],\cr
\end{split}
\end{equation}
where $\beta_i$ is the largest integer such that $1\le\beta\le p-1$, 
$i-\beta_i\ge 0$ and $J_{i,i-\beta_i}\ne[0,i-1]^2$. (If $\beta_i$ does not exist, 
the last walk at the right hand side of (\ref{Q}) is ignored.)
Then {\rm (\ref{5-2})} and {\rm (\ref{5-3})} hold if and only if
\[
S_i\le W_i\le T_i.
\]
\end{lem}

\begin{proof}
We will show that (\ref{5-3}) is equivalent to $S_i\le W_i$ and that (\ref{5-2}) is equivalent to $W_i\le T_i$.

First we claim that (\ref{5-3}) is equivalent to
\begin{equation}\label{5-7}
\bigl[J_{i,j}+D-(i-j)(0,p)\bigr]\cap[0,i]^2\subset J_i
\quad\text{for all}\ 0\le j<i
\end{equation}
and that (\ref{5-2}) is equivalent to
\begin{equation}\label{5-8}
\bigl[J_i+D+(a,0)\bigr]\cap[0,i-1]^2\subset J_{i,i-ap-b}
\end{equation}
and
\begin{equation}\label{5-9}
\bigl[J_i+D+(a+1,-p^2+bp)\bigr]\cap[0,i-1]^2\subset J_{i,i-ap-b},
\end{equation}
where $a,b\in{\Bbb Z},\ a\ge 0,\ 0\le b\le p-1$ and $ap+b\le i$.
The proof of these claims is the same as the proof of Lemma~\ref{L4.1} (i).

Therefore, it suffices to establish the following relations:
\vskip2mm

\begin{enumerate}
\item[] (\ref{5-7}) $\Leftrightarrow$ $S_i\le W_i$;

\item[] (\ref{5-8}) and (\ref{5-9}) $\Leftrightarrow$ $W_i\le T_i$.
\end{enumerate}
\vskip2mm

{\it Proof of} ``(\ref{5-7}) $\Leftrightarrow$ $S_i\le W_i$''. By Lemma~\ref{L5.4}, (\ref{5-7}) is equivalent to
\begin{equation}\label{5.6A}
\bigl[J_{i,j}+D-(i-j)(0,p)\bigr]\cap[0,i-1]^2\subset J_i\cap[0,i-1]^2\quad \text{for all}\
0\le j <i.
\end{equation}
Since $\bigcup_{j=0}^{i-1}(J_{i,j}\times\{j\})$ is an ideal of $[0,i-1]^3$, we have (cf. (\ref{JiJj}))
\begin{equation}\label{I5}
\bigl[ J_{i,j}+D-(i-1-j)(0,p)\bigr]\cap [0,i-1]^2\subset J_{i,i-1}\quad\text{for all}\ 0\le j<i.
\end{equation}
Note that Lemma~\ref{L4.3} remains true with $[0,i-1]^2$ in place of $U$. Thus by Lemma~\ref{L4.3} and 
(\ref{I5}),
we see that
(\ref{5.6A}) holds for all $0\le j<i$ if and only if it holds for $j=i-1$,
i.e., if and only if
\begin{equation}\label{5.6B}
\bigl[J_{i,i-1}+D-(0,p)\bigr]\cap[0,i-1]^2\subset J_i\cap[0,i-1]^2.
\end{equation}
By Lemma~\ref{L5.4} again, (\ref{5.6B}) is equivalent to
\begin{equation}\label{Ji,i-1}
\bigl[J_{i,i-1}+D-(0,p)\bigr]\cap[0,i]^2\subset J_i.
\end{equation}
In terms of boundaries, (\ref{Ji,i-1}) is equivalent to $S_i\le W_i$.
\vskip2mm

{\it Proof of} ``(\ref{5-8}) and (\ref{5-9}) $\Leftrightarrow$ $W_i\le T_i$''. Let $\hat J_{i,j}$ ($0\le j<i$) be the largest ideal of $[0,i]^2$ such that
$\hat J_{i,j}\cap[0,i-1]^2=J_{i,j}$. We claim that $W_i\le T_i$
is equivalent to the following three conditions:
\begin{equation}\label{5.8A}
J_i\subset\hat J_{i,i-1}.
\end{equation}
\begin{equation}\label{5.8B}
\bigl[J_i+D+(1,0)\bigr]\cap[0,i]^2\subset\hat J_{i,i-p}\quad\text{if}\ i\ge p.
\end{equation}
\begin{equation}\label{5.8C}
\bigl[J_i+D+(1,-p^2+p\beta_i)\bigr]\cap[0,i]^2\subset\hat J_{i,i-\beta_i}.
\end{equation}
(If $\beta_i$ does not exist, condition (\ref{5.8C}) is null.)

In fact, (\ref{5.8A}) is equivalent to
\[
W_i\le \overline{(W_{i,i-1})}_{[0,i]^2}.
\]
By Lemma~\ref{L4.2}, (\ref{5.8B}) is equivalent to
\[
\begin{split}
W_i+(1,0)\,&\le\overline{\bigl(\overline{(W_{i,i-p})}_{[0,i]^2}\bigr)}_{[0,1+i]\times[0,i]}\bigm|_{[1,1+i]\times[0,i]}\cr
&=\overline{(W_{i,i-p})}_{[0,1+i]\times[0,i]}\bigm|_{[1,1+i]\times[0,i]},\cr
\end{split}
\]
and (\ref{5.8C}) is equivalent to
\[
\begin{split}
&W_i+(1,-p^2+p\beta_i)\cr
\le\,& \overline{\bigl(\overline{(W_{i,i-\beta_i})}_{[0,i]^2}\bigr)}_{[0,1+i]\times[-p^2+p\beta_i,i]}\bigm|_{[1,1+i]\times[-p^2+p\beta_i,-p^2+p\beta_i+i]}\cr
&\,=\overline{(W_{i,i-\beta_i})}_{[0,1+i]\times[-p^2+p\beta_i,i]}\bigm|_{[1,1+i]\times[-p^2+p\beta_i,-p^2+p\beta_i+i]}.\cr
\end{split}
\]
Thus (\ref{5.8A}) -- (\ref{5.8C}) together are equivalent to $W_i\le T_i$.

Therefore, it remains to show that (\ref{5-8}) and (\ref{5-9}) $\Leftrightarrow$  (\ref{5.8A}) -- (\ref{5.8C}).
\vskip2mm

{\it Proof of} ``(\ref{5-8}) and (\ref{5-9}) $\Rightarrow$  (\ref{5.8A}) -- (\ref{5.8C})''.
In (\ref{5-8}), letting $a=0$ and $b=1$, we obtain
\[
J_i\cap[0,i-1]^2\subset J_{i,i-1}.
\]
Hence $J_i\subset\hat J_{i,i-1}$. In a similar way, (\ref{5.8B}) follows from (\ref{5-8}) with
$a=1$, $b=0$; (\ref{5.8C}) follows from (\ref{5-9}) with $a=0$, $b=\beta_i$.
\vskip2mm

{\it Proof of} ``(\ref{5-8}) and (\ref{5-9}) $\Leftarrow$  (\ref{5.8A}) -- (\ref{5.8C})''.
First assume $i<p$. In this case, the partial order $\prec$ in $[0,i]^3$ is the cartesian product of linear orders. 
Recall that (\ref{5-8}) and (\ref{5-9}) together are equivalent to (\ref{5-2}) and note that (\ref{5-2}) is equivalent to
\[
(J_i\times\{i\}+\Delta)\cap[0,i-1]^3\subset \bigcup_{j=0}^{i-1}(J_{i,j}\times\{j\}).
\]
Thus it suffices to show that
\begin{equation}\label{5.8D}
J_i\cap[0,i-1]^2\subset J_{i,j}\quad\text{for all}\ 0\le j<i.
\end{equation}
Since $\bigcup_{j=0}^{i-1}(J_{i,j}\times\{j\})$ is an ideal of $[0,i-1]^3$, we have $J_{i,j}\subset J_{i,j-1}$ for
all $0\le j<i$. By (\ref{5.8A}), we also have $J_i\cap[0,i-1]^2\subset J_{i,i-1}$. Hence (\ref{5.8D}) holds.

Now assume $i\ge p$. Since $\bigcup_{j=0}^{i-1}(J_{i,j}\times\{j\})$ is an ideal of $[0,i-1]^3$, by Lemma~\ref{L4.1} (i), we have 
\[
\begin{cases}
[J_{i,j}+D+(a,0)]\cap[0,i-1]^2\subset J_{i,j-ap-b}\cr
[J_{i,j}+D+(a+1,-p^2+bp)]\cap[0,i-1]^2\subset J_{i,j-ap-b}\cr
\end{cases}
\]
for $a\ge 0,\ 0\le b\le p-1$ and $ap+b\le j<i$. By Lemma~\ref{L5.7}, we have
\begin{equation}\label{5.8E}
\begin{cases}
[\hat J_{i,j}+D+(a,0)]\cap[0,i]^2\subset \hat J_{i,j-ap-b}\cr
[\hat J_{i,j}+D+(a+1,-p^2+bp)]\cap[0,i]^2\subset \hat J_{i,j-ap-b}\cr
\end{cases}
\end{equation}
for $a\ge 0,\ 0\le b\le p-1$ and $ap+b\le j<i$. Note that $\beta_i$ is also the largest integer such that $1\le \beta_i\le p-1$, $i-\beta_i\ge 0$ and $\hat J_{i,i-\beta_i}\ne[0,i]^2$. By (\ref{5.8B}), (\ref{5.8C}), (\ref{5.8E})
and the proof of Theorem~\ref{T4.6}, we have
\begin{equation}\label{5.8F}
\begin{cases}
[J_i+D+(a,0)]\cap[0,i]^2\subset \hat J_{i,i-ap-b}\cr
[J_i+D+(a+1,-p^2+bp)]\cap[0,i]^2\subset\hat J_{i,i-ap-b}\cr
\end{cases}
\end{equation}
for $a\ge 0,\ 0\le b\le p-1$ and $ap+b\le j<i$. Conditions (\ref{5-8}) and (\ref{5-9}) immediately follow form
(\ref{5.8F}).
\end{proof}

\noindent{\bf Remark}. In Lemma~\ref{L5.9}, we always have
\[
S_i\le T_i.
\]
In fact, by Lemma~\ref{L2.1} (i), there is at least one $J_i$ satisfying 
all the conditions in Lemma~\ref{L5.1}. Thus there exists at least one walk $W_i$ in $[0,i]^2$ such that $S_i\le W_i\le T_i$.

\begin{defn}\label{D5.1}
Let $0\le i\le n$ and let $J_i$ be an ideal of $[0,i]^2$. We call $J_i$ of 
\begin{enumerate}
\item[]type I \quad if $J_i\cap([i-1,i]\times[0,i])
=\emptyset$;

\item[]type II \quad if $J_i\cap([i-1,i]\times[0,i])\ne\emptyset$ but $J_i\cap(\{i\}\times[0,i])
=\emptyset$; 

\item[]type III \quad if $J_i\cap(\{i\}\times[0,i])\ne\emptyset$.
\end{enumerate}
\end{defn}

\begin{thm}\label{T5.10}
Let $1\le i\le n$ and let $J_j$ ($0\le j\le i$) be an ideal $[0,j]^2$. Assume that
$J_0,\dots,J_{i-1}$ is a consistent sequence of ideals and write
\[
\bigcup_{j=0}^{i-1}\bigl[(J_j\times\{j\})\cup(J_j\times\{j\})^A\cup(J_j\times\{j\})^{A^2}\bigr]
=\bigcup_{j=0}^{i-1}(J_{i,j}\times\{j\}),
\]
where $J_{i,j}$ is an ideal of $[0,i-1]^2$. Let $W_{i,j}=\omega(J_{i,j})$ ($0\le j <i$) and $W_i=\omega(J_i)$
and let $S_i$ and $T_i$ be as in Lemma~\ref{L5.9}.
\begin{enumerate}

\item[(i)] $J_i$ is of type I and consistent with $J_0,\dots,J_{i-1}$ if and only if 
\begin{equation}\label{5.10.1}
(0,i)\notin\iota(S_i),\quad (i-1,0)\notin\iota(S_i)
\end{equation}
and
\[
S_i\le W_i\le T'_i,
\]
where
\[
T'_i=T_i\wedge A_i\wedge B_i,
\]
$A_i$ is the highest walk in $[0,i]^2$ starting from $(0,i-1)$ and $B_i$ is the highest walk in $[0,i]^2$ ending at $(i-2,0)$. 

\item[(ii)] $J_i$ is of type II and consistent with $J_0,\dots,J_{i-1}$ if and only if
\begin{equation}\label{5.10.2}
(0,i)\notin\iota(S_i),\quad (i,0)\notin\iota(S_i)
\end{equation}
and 
\[
\begin{cases}
W_i|_{[i-1,i]\times[0,i]}=\bigl((i-1,v),\,(i-1,0)\bigr)\cr
\Gamma_i\le W_i|_{[0,i-1]\times[0,i]}\le\Lambda_i\cr
\end{cases}
\]
for some integer $v$ satisfying
\begin{equation}\label{5.10.3}
\begin{cases}
0\le v<\min\{p^2,\frac{p-1}pi+\frac 1p\}\cr
(i-1,v)\in\iota(T_i),\ (i-1,v+i)\notin\iota(X_i)\cr
(v,i-p)\in \iota(T_i)\quad \text{if}\ i\ge p\cr
\end{cases}
\end{equation} 
and for the walks $\Gamma_i$ and $\Lambda_i$ defined as follows.
\[
\Gamma_i=(S_i\vee E_{i,v})|_{[0,i-1]\times[0,i]}\vee C_{i,v},
\]
\[
\Lambda_i=(T_i\wedge A_i)|_{[0,i-1]\times[0,i]}\wedge D_{i,v},
\]
where $C_{i,v}$ is the lowest walk in $[0,i-1]\times[0,i]$ ending at $(i-1,v)$,
$D_{i,v}$ is the highest walk in $[0,i-1]\times[0,i]$ ending at $(i-1,v)$, and
\[
E_{i,v}=
\begin{cases}
\text{the lowest walk in $[0,i]^2$ passing through $(v,i-p)$},&\text{if}\ i\ge p,\cr
\emptyset,&\text{if}\ i<p.\cr
\end{cases}
\]

\item[(iii)] $J_i$ is of type III and consistent with $J_0,\dots,J_{i-1}$ if and only if
\[
\Phi_i\le W_i\le \Psi_i
\]
for some integer $u$ satisfying
\begin{equation}\label{5.10.6}
\begin{cases}
0\le u\le i\cr
(i,u)\in\iota(T_i),\ (u,i)\in\iota(T_i)\cr
(i,u+1)\notin\iota(S_i),\ (u+1,i)\notin\iota(S_i)\cr
\end{cases}
\end{equation}
and for the walks $\Phi_i$ and $\Psi_i$ defined as follows.
\[
\Phi_i=S_i\vee F_{i,u}\vee M_{i,u},
\]
\[
\Psi_i=T_i\wedge G_{i,u}\wedge N_{i,u},
\] 
where $F_{i,u}$ is the lowest walk in $[0,i]^2$ starting from $(u,i)$,
$G_{i,u}$ is the highest walk in $[0,i]^2$ starting from $(u,i)$,
$M_{i,u}$ is the lowest walk in $[0,i]^2$ ending at $(i,u)$,
$N_{i,u}$ is the highest walk in $[0,i]^2$ ending at $(i,u)$.

\end{enumerate}
\end{thm}

\begin{proof}
{\it Necessity}.
We first show the necessity in cases (i) -- (iii). By Lemma~\ref{L5.9},
we have $S_i\le W_i\le T_i$.

(i) Since $J_i$ is of type I, $(i-1,0)\notin J_i$. By (\ref{5-1}), 
$(0,i)\notin J_i$. Thus $(0,i)\notin\iota(S_i)$,  
$(i-1,0)\notin\iota(S_i)$ and $W_i\le A_i\wedge B_i$. Hence $W_i\le T'_i$.

(ii) Since $J_i$ is of type II, we have 
\[
W_i|_{[i-1,i]\times[0,i]}=\bigl((i-1,v),\,(i-1,0)\bigr)
\]
for some $0\le v\le i$.  Since $(i,0)\notin J_i$, by (\ref{5-1}), $(0,i)\notin J_i$. Thus $(i-1,v)\in J_i$ implies that $v<p^2$ and 
$v+(i-1)\frac 1p<i$, i.e.,
\[
v<\min\bigl\{p^2,\ \frac{p-1}pi+\frac 1p\bigr\}.
\]
Clearly, $(0,i)\notin\iota(S_i)$, $(i,0)\notin\iota(S_i)$,
$(i-1,v)\in\iota(T_i)$ and $(i-1,v+1)\notin\iota(S_i)$. By (\ref{5-13}),
$(v,i-p)\in J_i\subset\iota(T_i)$ if $i\ge p$. 

Since $W_i|_{[0,i-1]\times[0,i]}$ ends at $(i-1,v)$, we have
\begin{equation}\label{5.10.9}
C_{i,v}\le W_i|_{[0,i-1]\times[0,i]}\le D_{i,v}.
\end{equation}
Since $(0,i)\notin J_i$, we have $W\le A_i$. In case $i\ge p$, Lemma~\ref{L5.3} implies $(v,i-p)\in J_i$. Thus, whether $i\ge p$ or not, we always have $W\ge E_{i,v}$. It follows that
\begin{equation}\label{SEW}
S_i\vee E_{i,v}\le W_i\le T_i\wedge A_i.
\end{equation}
Combining (\ref{5.10.9}) and (\ref{SEW}), we get
\[
\Gamma_i\le W_i|_{[0,i-1]\times[0,i]}\le \Lambda_i.
\]

(iii) Assume that $W_i$ ends at $(i,u)$. By (\ref{5-1}), $W_i$ starts with
$(u,i)$. Thus $F_{i,u}\le W_i\le G_{i,u}$ and $M_{i,u}\le W_i\le N_{i,u}$.
It follows that $\Phi_i\le W_i\le \Psi_i$. Condition (\ref{5.10.6}) is obvious.

{\it Sufficiency}. For the sufficiency in cases (i) -- (iii), we only give the proof for case (iii). The proofs for cases (i) and (ii) are similar.

By Lemmas~\ref{L5.1} and \ref{L5.2}, it suffice to show that conditions 
(\ref{5-1}) -- (\ref{5-3}) and (\ref{5-5}) are satisfied. Since 
\[
F_{i,u}\vee M_{i,u}\le W_i\le G_{i,u}\wedge N_{i,u},
\]
$W_i$ must start from $(u,i)$ and end at $(i,u)$. Hence (\ref{5-1}) holds.
Since $S_i\le W_i\le T_i$, by Lemma~\ref{L5.9}, (\ref{5-2}) and (\ref{5-3})
follow. Let $v=\max\{y:(i-1,y)\in J_i\}$. Then $v-u\le p^2$. Thus 
$(v,i-p)\prec (u+p^2,i-p)\prec (u,i)\in J_i$. Hence $(v,i-p)\in J_i$ and 
consequently, (\ref{5-13}) holds. By Lemma~\ref{L5.3}, (\ref{5-5}) follows. 
\end{proof}

\begin{lem}\label{L5.11}
In case {\rm (i)} of Theorem~\ref{T5.10}, condition {\rm (\ref{5.10.1})} implies
\begin{equation}\label{5.10.10}
S_i\le T'_i.
\end{equation}
In case {\rm (ii)}, conditions {\rm (\ref{5.10.2})} and {\rm (\ref{5.10.3})} imply
\begin{equation}\label{5.10.11}
\Gamma_i\le \Lambda_i.
\end{equation}
In case {\rm (iii)}, condition {\rm (\ref{5.10.6})} implies
\begin{equation}\label{5.10.12}
\Phi_i\le \Psi_i.
\end{equation}
\end{lem}

\noindent{\bf Remark}. Lemma~\ref{L5.11} assures the existence of $W_i$ in Theorem~\ref{T5.10} provided condition (\ref{5.10.1}) in case (i), or
conditions (\ref{5.10.2}) and (\ref{5.10.3}) in case (ii), or 
condition (\ref{5.10.6}) in case (iii) are satisfied.

\begin{proof}[Proof of Lemma~\ref{L5.11}] It is obvious that (\ref{5.10.1})
implies (\ref{5.10.10}) and that (\ref{5.10.6}) implies (\ref{5.10.12}). 
We only prove that (\ref{5.10.2}) and (\ref{5.10.3}) imply (\ref{5.10.11}).

We show that each of the walks $S_i|_{[0,i-1]\times[0,i]}$,
$E_{i,v}|_{[0,i-1]\times[0,i]}$ and $C_{i,v}$ is $\le$ each of the
walks $T_i|_{[0,i-1]\times[0,i]}$, $A_i|_{[0,i-1]\times[0,i]}$ and
$D_{i,v}$. Most of these relations are obvious. The only ones that need
proofs are 
\begin{equation}\label{5.10.13}
E_{i,v}|_{[0,i-1]\times[0,i]}\le A_i|_{[0,i-1]\times[0,i]},
\end{equation}
\begin{equation}\label{5.10.14}
E_{i,v}|_{[0,i-1]\times[0,i]}\le D_{i,v},
\end{equation}
\begin{equation}\label{5.10.15}
C_{i,v}\le A_i|_{[0,i-1]\times[0,i]}.
\end{equation}
 
To prove (\ref{5.10.13}), we may assume $i\ge p$. It suffices to show that
$(0,i)\notin\iota(E_{i,v})$, i.e., $i-p+\frac 1p v<i$. (See Figure 27(a).)
This is true since $v<p^2$.

To prove (\ref{5.10.14}), we may again assume $i\ge p$. It suffices to show that $i-p\le v+(i-1-v)p^2$, i.e., $v\le i-\frac p{p+1}$. (See Figure 27(b).)
This follows from the inequality $v<\frac{p-1}p i+\frac 1p$ in (\ref{5.10.3}).

To prove (\ref{5.10.15}), it suffices to have $v+\frac 1p(i-1)<i$.
(See Figure 27(c).) This is given by (\ref{5.10.3}).
\end{proof}

%%%%%%%%%%%% Figure  %%%%%%%%%%%%%%%%%%%%%%%%%%%%%%%
\vskip5mm
\setlength{\unitlength}{4mm}
\[
\begin{picture}(7,7)
\put(0,0){\vector(1,0){7}}
\put(0,0){\vector(0,1){7}}
\put(0,6){\line(1,0){6}}
\put(6,0){\line(0,1){6}}
\put(7,2.67){\line(-3,1){7.5}}
\put(0,5){\makebox(0,0){$\scriptstyle \bullet$}}
\put(4,3.67){\makebox(0,0){$\scriptstyle \bullet$}}
\put(0.1,5){\makebox(0,0)[bl]{$\scriptstyle i-p+\frac 1p v$}}
\put(3.5,4){\makebox(0,0)[bl]{$\scriptstyle (v,i-p)$}}
\put(5,3){\makebox(0,0)[tr]{$\scriptstyle \text{slope}=-\frac 1p$}}
\put(6,-0.2){\makebox(0,0)[t]{$\scriptstyle i$}}
\put(-0.1,6){\makebox(0,0)[r]{$\scriptstyle i$}}
\put(0,-1){\makebox(0,0)[l]{$\scriptstyle \text{(a) Proof of (\ref{5.10.13})}$}}
\end{picture}
\kern1.7cm
\begin{picture}(7,7)
\put(0,0){\vector(1,0){7}}
\put(0,0){\vector(0,1){7}}
\put(0,6){\line(1,0){6}}
\put(6,0){\line(0,1){6}}
\put(4,1.8){\line(0,1){6.2}}
\put(5.25,2){\line(-1,4){1.5}}
\put(4,7){\makebox(0,0){$\scriptstyle \bullet$}}
\put(4,3){\makebox(0,0){$\scriptstyle \bullet$}}
\put(5.25,2){\makebox(0,0){$\scriptstyle \bullet$}}
\put(3.8,3){\makebox(0,0)[r]{$\scriptstyle (v,i-p)$}}
\put(4.3,7){\makebox(0,0)[l]{$\scriptstyle (v,\,v+(i-1-v)p^2)$}}
\put(4.8,1.8){\makebox(0,0)[t]{$\scriptstyle (i-1,v)$}}
\put(3.5,7.5){\makebox(0,0)[r]{$\scriptstyle \text{slope}=-p^2$}}
\put(6,-0.2){\makebox(0,0)[t]{$\scriptstyle i$}}
\put(-0.1,6){\makebox(0,0)[r]{$\scriptstyle i$}}
\put(0,-1){\makebox(0,0)[l]{$\scriptstyle \text{(b) Proof of (\ref{5.10.14})}$}}
\end{picture}
\kern1.7cm
\begin{picture}(7,7)
\put(0,0){\vector(1,0){7}}
\put(0,0){\vector(0,1){7}}
\put(0,6){\line(1,0){6}}
\put(6,0){\line(0,1){6}}
\put(-1,5){\line(3,-1){8}}
\put(0,4.67){\makebox(0,0){$\scriptstyle \bullet$}}
\put(5,3){\makebox(0,0){$\scriptstyle \bullet$}}
\put(0.2,4.7){\makebox(0,0)[bl]{$\scriptstyle v+\frac 1p (i-1)$}}
\put(5,2.9){\makebox(0,0)[tr]{$\scriptstyle (i-1,v)$}}
\put(2.8,3.8){\makebox(0,0)[bl]{$\scriptstyle \text{slope}=-\frac 1p$}}
\put(-0.2,6){\makebox(0,0)[r]{$\scriptstyle i$}}
\put(6,-0.2){\makebox(0,0)[t]{$\scriptstyle i$}}
\put(0,-1){\makebox(0,0)[l]{$\scriptstyle \text{(c) Proof of (\ref{5.10.15})}$}}
\end{picture}
\]
\vskip3mm
\[
\text{Figure 27. Proofs of (\ref{5.10.13}) -- (\ref{5.10.15}) }
\]
%%%%%%%%%%%%%%%% End Figure %%%%%%%%%%%%%%%%%%%%%%%%%%%%%%%%%%%%%%
\vskip3mm

\begin{exmp}\label{E5.1}
\rm
Let $p=3$ and $m=9$ ($n=\frac m3(p-1)=6$). In this example, we exhibit a consistent sequence of ideals
$J_0,\dots,J_6$ using Theorem~\ref{T5.10}. Figure 28 gives the boundaries $W_i=\omega(J_i)$ ($0\le i\le 6$) and the walks $S_i$ and $T_i$ which are needed for choosing $W_i$. The resulting $A$-invariant ideal of $[0,6]^3$,
\[
I=\bigcup_{i=0}^6\Bigl[(J_i\times\{j\})\cup(J_i\times\{j\})^A\cup(J_i\times\{j\})^{A^2}\Bigr],
\]
is depicted in Figure 30. The cross sections of $I$ on the parallels of the $xy$-planes, i.e., $J_{6,0},\dots,J_{6,5},
J_6$ are given in Figure 29 in terms of their boundaries $W_{6,0},\dots,W_{6,5},W_6$.
The $A$-symmetry of $I$ is clearly visible in Figure 30. However, the fact that $I$ is an ideal in 
$({\mathcal U},\prec)$ is not obvious from Figure 30.     
\end{exmp}

\vfill\eject

%%%%%%%%%%%% Figure 28 %%%%%%%%%%%%%%%%%%%%%%%%%%%%%%%
\vskip5mm
\setlength{\unitlength}{5mm}
\[
\begin{picture}(2,2)
\put(0,0){\vector(1,0){2}}
\put(0,0){\vector(0,1){2}}
\put(0,0){\makebox(0,0){$\scriptstyle \bullet$}}
\put(1,-1){\makebox(0,0){$\scriptstyle W_0$}}
\end{picture}
\kern2cm
\begin{picture}(2,2)
\put(0,0){\vector(1,0){2}}
\put(0,0){\vector(0,1){2}}
\put(0,1){\line(1,0){1}}
\put(1,0){\line(0,1){1}}
\put(1,1){\makebox(0,0){$\scriptstyle \bullet$}}
\put(1,-1){\makebox(0,0){$\scriptstyle W_1$}}
\put(0,-2){\makebox(0,0)[l]{$\scriptstyle S_1=\emptyset$}}
\put(0,-3){\makebox(0,0)[l]{$\scriptstyle T_1=(1,1)$}}
\end{picture}
\kern2cm
\begin{picture}(3,3)
\put(0,0){\vector(1,0){3}}
\put(0,0){\vector(0,1){3}}
\put(0,2){\line(1,0){2}}
\put(2,0){\line(0,1){2}}
\put(1,0){\line(0,1){0.1}}
\put(0,1){\line(1,0){0.1}}
\put(2,2){\makebox(0,0){$\scriptstyle \bullet$}}
\put(1.5,-1){\makebox(0,0){$\scriptstyle W_2$}}
\put(0,-2){\makebox(0,0)[l]{$\scriptstyle S_2=\emptyset$}}
\put(0,-3){\makebox(0,0)[l]{$\scriptstyle T_2=(2,2)$}}
\end{picture}
\kern2cm
\begin{picture}(4,4)
\put(0,0){\vector(1,0){4}}
\put(0,0){\vector(0,1){4}}
\put(0,3){\line(1,0){3}}
\put(3,0){\line(0,1){3}}
\put(1,3){\line(0,-1){1}}
\put(1,2){\line(1,0){1}}
\put(2,2){\line(0,-1){1}}
\put(2,1){\line(1,0){1}}
\multiput(1,0)(1,0){2}{\line(0,1){0.1}}
\multiput(0,1)(0,1){2}{\line(1,0){0.1}}
\put(1,3){\makebox(0,0){$\scriptstyle \bullet$}}
\put(1,2){\makebox(0,0){$\scriptstyle \bullet$}}
\put(2,2){\makebox(0,0){$\scriptstyle \bullet$}}
\put(2,1){\makebox(0,0){$\scriptstyle \bullet$}}
\put(3,1){\makebox(0,0){$\scriptstyle \bullet$}}
\put(2,-1){\makebox(0,0){$\scriptstyle W_3$}}
\put(0,-2){\makebox(0,0)[l]{$\scriptstyle S_3=\emptyset$}}
\put(0,-3){\makebox(0,0)[l]{$\scriptstyle T_3=(3,3)$}}
\end{picture}
\]

\vskip 1cm

\[
\begin{picture}(5,5)
\put(0,0){\vector(1,0){5}}
\put(0,0){\vector(0,1){5}}
\put(0,4){\line(1,0){4}}
\put(4,0){\line(0,1){4}}
\put(1,4){\line(0,-1){2}}
\put(1,2){\line(1,0){1}}
\put(2,2){\line(0,-1){1}}
\put(2,1){\line(1,0){2}}
\multiput(2,0)(1,0){2}{\line(0,1){0.1}}
\multiput(0,1)(0,1){3}{\line(1,0){0.1}}

\put(1,4){\makebox(0,0){$\scriptstyle \circ$}}
\put(1,2){\makebox(0,0){$\scriptstyle \circ$}}
\put(2,2){\makebox(0,0){$\scriptstyle \circ$}}
\put(2,1){\makebox(0,0){$\scriptstyle \circ$}}
\put(4,1){\makebox(0,0){$\scriptstyle \circ$}}
\put(0,0){\makebox(0,0){$\scriptstyle \times$}}
\put(1,0){\makebox(0,0){$\scriptstyle \times$}}
\put(1.5,-1){\makebox(0,0)[r]{$\scriptstyle S_4:$}}
\put(2,-1){\makebox(0,0){$\scriptstyle \times$}}
\put(3,-1){\makebox(0,0){$\scriptstyle \times$}}
\put(2,-1){\line(1,0){1}}
\put(1.5,-2){\makebox(0,0)[r]{$\scriptstyle T_4:$}}
\put(2,-2){\makebox(0,0){$\scriptstyle \circ$}}
\put(3,-2){\makebox(0,0){$\scriptstyle \circ$}}
\put(2,-2){\line(1,0){1}}
\end{picture}
\kern3cm
\begin{picture}(5,5)
\put(0,0){\vector(1,0){5}}
\put(0,0){\vector(0,1){5}}
\put(0,4){\line(1,0){4}}
\put(4,0){\line(0,1){4}}
\put(0,2){\line(1,0){2}}
\put(2,2){\line(0,-1){1}}
\put(2,1){\line(1,0){1}}
\put(3,1){\line(0,-1){1}}
\multiput(1,0)(1,0){2}{\line(0,1){0.1}}
\multiput(0,1)(0,2){2}{\line(1,0){0.1}}

\put(0,4){\makebox(0,0){$\scriptstyle \bullet$}}
\put(0,2){\makebox(0,0){$\scriptstyle \bullet$}}
\put(2,2){\makebox(0,0){$\scriptstyle \bullet$}}
\put(2,1){\makebox(0,0){$\scriptstyle \bullet$}}
\put(3,1){\makebox(0,0){$\scriptstyle \bullet$}}
\put(3,0){\makebox(0,0){$\scriptstyle \bullet$}}
\put(4,0){\makebox(0,0){$\scriptstyle \bullet$}}
\put(2.5,-1){\makebox(0,0){$\scriptstyle W_4$}}
\end{picture}
\]

\vskip 1cm

\[
\begin{picture}(6,6)
\put(0,0){\vector(1,0){6}}
\put(0,0){\vector(0,1){6}}
\put(0,5){\line(1,0){5}}
\put(5,0){\line(0,1){5}}
\put(0,2){\line(1,0){2}}
\put(2,2){\line(0,-1){1}}
\put(2,1){\line(1,0){1}}
\put(3,1){\line(0,-1){1}}
\multiput(1,0)(1,0){2}{\line(0,1){0.1}}
\put(4,0){\line(0,1){0.1}}
\multiput(0,3)(0,1){2}{\line(1,0){0.1}}

\put(0,5){\makebox(0,0){$\scriptstyle \circ$}}
\put(0,2){\makebox(0,0){$\scriptstyle \circ$}}
\put(2,2){\makebox(0,0){$\scriptstyle \circ$}}
\put(2,1){\makebox(0,0){$\scriptstyle \circ$}}
\put(3,1){\makebox(0,0){$\scriptstyle \circ$}}
\put(3,0){\makebox(0,0){$\scriptstyle \circ$}}
\put(5,0){\makebox(0,0){$\scriptstyle \circ$}}
\put(0,0){\makebox(0,0){$\scriptstyle \times$}}
\put(0,1){\makebox(0,0){$\scriptstyle \times$}}
\put(1.5,-1){\makebox(0,0)[r]{$\scriptstyle S_5:$}}
\put(2,-1){\makebox(0,0){$\scriptstyle \times$}}
\put(3,-1){\makebox(0,0){$\scriptstyle \times$}}
\put(2,-1){\line(1,0){1}}
\put(1.5,-2){\makebox(0,0)[r]{$\scriptstyle T_5:$}}
\put(2,-2){\makebox(0,0){$\scriptstyle \circ$}}
\put(3,-2){\makebox(0,0){$\scriptstyle \circ$}}
\put(2,-2){\line(1,0){1}}
\end{picture}
\kern3cm
\begin{picture}(6,6)
\put(0,0){\vector(1,0){6}}
\put(0,0){\vector(0,1){6}}
\put(0,5){\line(1,0){5}}
\put(5,0){\line(0,1){5}}
\put(0,2){\line(1,0){1}}
\put(1,2){\line(0,-1){1}}
\put(1,1){\line(1,0){2}}
\put(3,1){\line(0,-1){1}}
\multiput(1,0)(1,0){2}{\line(0,1){0.1}}
\put(0,1){\line(1,0){0.1}}
\multiput(0,3)(0,1){2}{\line(1,0){0.1}}

\put(0,2){\makebox(0,0){$\scriptstyle \bullet$}}
\put(1,2){\makebox(0,0){$\scriptstyle \bullet$}}
\put(1,1){\makebox(0,0){$\scriptstyle \bullet$}}
\put(3,1){\makebox(0,0){$\scriptstyle \bullet$}}
\put(3,0){\makebox(0,0){$\scriptstyle \bullet$}}
\put(4,0){\makebox(0,0){$\scriptstyle \bullet$}}
\put(3,-1){\makebox(0,0){$\scriptstyle W_5$}}
\end{picture}
\]

\vskip 1cm

\[
\begin{picture}(7,7)
\put(0,0){\vector(1,0){7}}
\put(0,0){\vector(0,1){7}}
\put(0,6){\line(1,0){6}}
\put(6,0){\line(0,1){6}}
\put(0,2){\line(1,0){1}}
\put(1,2){\line(0,-1){1}}
\put(1,1){\line(1,0){1}}
\put(2,1){\line(0,-1){1}}
\multiput(4,0)(1,0){2}{\line(0,1){0.1}}
\put(1,0){\line(0,1){0.1}}
\multiput(0,3)(0,1){3}{\line(1,0){0.1}}
\put(0,1){\line(1,0){0.1}}

\put(0,2){\makebox(0,0){$\scriptstyle \circ$}}
\put(1,2){\makebox(0,0){$\scriptstyle \circ$}}
\put(1,1){\makebox(0,0){$\scriptstyle \circ$}}
\put(2,1){\makebox(0,0){$\scriptstyle \circ$}}
\put(2,0){\makebox(0,0){$\scriptstyle \circ$}}
\put(3,0){\makebox(0,0){$\scriptstyle \circ$}}
\put(2,-1){\makebox(0,0)[l]{$\scriptstyle S_6=\emptyset$}}
\put(2,-2){\makebox(0,0)[l]{$\scriptstyle T_6:$}}
\put(3,-2){\makebox(0,0){$\scriptstyle \circ$}}
\put(4,-2){\makebox(0,0){$\scriptstyle \circ$}}
\put(3,-2){\line(1,0){1}}
\end{picture}
\kern3cm
\begin{picture}(7,7)
\put(0,0){\vector(1,0){7}}
\put(0,0){\vector(0,1){7}}
\put(0,6){\line(1,0){6}}
\put(6,0){\line(0,1){6}}
\put(0,1){\line(1,0){1}}
\put(1,1){\line(0,-1){1}}
\multiput(2,0)(1,0){4}{\line(0,1){0.1}}
\multiput(0,2)(0,1){4}{\line(1,0){0.1}}

\put(0,1){\makebox(0,0){$\scriptstyle \bullet$}}
\put(1,1){\makebox(0,0){$\scriptstyle \bullet$}}
\put(1,0){\makebox(0,0){$\scriptstyle \bullet$}}
\put(3.5,-1){\makebox(0,0){$\scriptstyle W_6$}}
\end{picture}
\]
\vskip5mm
\[
\text{Figure 28. Example~\ref{E5.1}, the walks $S_i$, $T_i$ and $W_i$}
\]
%%%%%%%%%%%%%%%% End Figure %%%%%%%%%%%%%%%%%%%%%%%%%%%%%%%%%%%%%%

%%%%%%%%%%%% Figure 29 %%%%%%%%%%%%%%%%%%%%%%%%%%%%%%%
\vskip5mm
\setlength{\unitlength}{4mm}
\[
\begin{picture}(6,6)
\put(0,0){\line(1,0){6}}
\put(0,6){\line(1,0){6}}
\put(0,0){\line(0,1){6}}
\put(6,0){\line(0,1){6}}
\multiput(1,0)(1,0){5}{\line(0,1){0.1}}
\multiput(0,1)(0,1){5}{\line(1,0){0.1}}

\put(1,6){\makebox(0,0){$\scriptstyle \bullet$}}
\put(1,5){\makebox(0,0){$\scriptstyle \bullet$}}
\put(2,5){\makebox(0,0){$\scriptstyle \bullet$}}
\put(2,4){\makebox(0,0){$\scriptstyle \bullet$}}
\put(5,4){\makebox(0,0){$\scriptstyle \bullet$}}
\put(5,1){\makebox(0,0){$\scriptstyle \bullet$}}
\put(6,1){\makebox(0,0){$\scriptstyle \bullet$}}

\put(1,6){\line(0,-1){1}}
\put(1,5){\line(1,0){1}}
\put(2,5){\line(0,-1){1}}
\put(2,4){\line(1,0){3}}
\put(5,4){\line(0,-1){3}}
\put(5,1){\line(1,0){1}}
\put(3,-1){\makebox(0,0)[t]{$\scriptstyle W_{6,0}$}}

\end{picture}
\kern1cm
\begin{picture}(6,6)
\put(0,0){\line(1,0){6}}
\put(0,6){\line(1,0){6}}
\put(0,0){\line(0,1){6}}
\put(6,0){\line(0,1){6}}
\multiput(1,0)(1,0){5}{\line(0,1){0.1}}
\multiput(0,1)(0,1){5}{\line(1,0){0.1}}

\put(1,6){\makebox(0,0){$\scriptstyle \bullet$}}
\put(1,5){\makebox(0,0){$\scriptstyle \bullet$}}
\put(2,5){\makebox(0,0){$\scriptstyle \bullet$}}
\put(2,3){\makebox(0,0){$\scriptstyle \bullet$}}
\put(5,3){\makebox(0,0){$\scriptstyle \bullet$}}
\put(5,1){\makebox(0,0){$\scriptstyle \bullet$}}
\put(6,1){\makebox(0,0){$\scriptstyle \bullet$}}

\put(1,6){\line(0,-1){1}}
\put(1,5){\line(1,0){1}}
\put(2,5){\line(0,-1){2}}
\put(2,3){\line(1,0){3}}
\put(5,3){\line(0,-1){2}}
\put(5,1){\line(1,0){1}}
\put(3,-1){\makebox(0,0)[t]{$\scriptstyle W_{6,1}$}}

\end{picture}
\kern1cm
\begin{picture}(6,6)
\put(0,0){\line(1,0){6}}
\put(0,6){\line(1,0){6}}
\put(0,0){\line(0,1){6}}
\put(6,0){\line(0,1){6}}
\multiput(1,0)(1,0){5}{\line(0,1){0.1}}
\multiput(0,1)(0,1){5}{\line(1,0){0.1}}

\put(0,5){\makebox(0,0){$\scriptstyle \bullet$}}
\put(1,5){\makebox(0,0){$\scriptstyle \bullet$}}
\put(1,4){\makebox(0,0){$\scriptstyle \bullet$}}
\put(2,4){\makebox(0,0){$\scriptstyle \bullet$}}
\put(2,2){\makebox(0,0){$\scriptstyle \bullet$}}
\put(4,2){\makebox(0,0){$\scriptstyle \bullet$}}
\put(4,1){\makebox(0,0){$\scriptstyle \bullet$}}
\put(5,1){\makebox(0,0){$\scriptstyle \bullet$}}
\put(5,0){\makebox(0,0){$\scriptstyle \bullet$}}

\put(0,5){\line(1,0){1}}
\put(1,5){\line(0,-1){1}}
\put(1,4){\line(1,0){1}}
\put(2,4){\line(0,-1){2}}
\put(2,2){\line(1,0){2}}
\put(4,2){\line(0,-1){1}}
\put(4,1){\line(1,0){1}}
\put(5,1){\line(0,-1){1}}
\put(3,-1){\makebox(0,0)[t]{$\scriptstyle W_{6,2}$}}

\end{picture}
\kern1cm
\begin{picture}(6,6)
\put(0,0){\line(1,0){6}}
\put(0,6){\line(1,0){6}}
\put(0,0){\line(0,1){6}}
\put(6,0){\line(0,1){6}}
\multiput(1,0)(1,0){5}{\line(0,1){0.1}}
\multiput(0,1)(0,1){5}{\line(1,0){0.1}}

\put(0,5){\makebox(0,0){$\scriptstyle \bullet$}}
\put(1,5){\makebox(0,0){$\scriptstyle \bullet$}}
\put(1,2){\makebox(0,0){$\scriptstyle \bullet$}}
\put(2,2){\makebox(0,0){$\scriptstyle \bullet$}}
\put(2,1){\makebox(0,0){$\scriptstyle \bullet$}}
\put(3,1){\makebox(0,0){$\scriptstyle \bullet$}}
\put(3,0){\makebox(0,0){$\scriptstyle \bullet$}}
\put(4,0){\makebox(0,0){$\scriptstyle \bullet$}}

\put(0,5){\line(1,0){1}}
\put(1,5){\line(0,-1){3}}
\put(1,2){\line(1,0){1}}
\put(2,2){\line(0,-1){1}}
\put(2,1){\line(1,0){1}}
\put(3,1){\line(0,-1){1}}
\put(3,-1){\makebox(0,0)[t]{$\scriptstyle W_{6,3}$}}

\end{picture}
\]

\[
\begin{picture}(6,6)
\put(0,0){\line(1,0){6}}
\put(0,6){\line(1,0){6}}
\put(0,0){\line(0,1){6}}
\put(6,0){\line(0,1){6}}
\multiput(1,0)(1,0){5}{\line(0,1){0.1}}
\multiput(0,1)(0,1){5}{\line(1,0){0.1}}

\put(0,5){\makebox(0,0){$\scriptstyle \bullet$}}
\put(0,2){\makebox(0,0){$\scriptstyle \bullet$}}
\put(2,2){\makebox(0,0){$\scriptstyle \bullet$}}
\put(2,1){\makebox(0,0){$\scriptstyle \bullet$}}
\put(3,1){\makebox(0,0){$\scriptstyle \bullet$}}
\put(3,0){\makebox(0,0){$\scriptstyle \bullet$}}
\put(4,0){\makebox(0,0){$\scriptstyle \bullet$}}

\put(0,2){\line(1,0){2}}
\put(2,2){\line(0,-1){1}}
\put(2,1){\line(1,0){1}}
\put(3,1){\line(0,-1){1}}
\put(3,-1){\makebox(0,0)[t]{$\scriptstyle W_{6,4}$}}

\end{picture}
\kern1cm
\begin{picture}(6,6)
\put(0,0){\line(1,0){6}}
\put(0,6){\line(1,0){6}}
\put(0,0){\line(0,1){6}}
\put(6,0){\line(0,1){6}}
\multiput(1,0)(1,0){5}{\line(0,1){0.1}}
\multiput(0,1)(0,1){5}{\line(1,0){0.1}}

\put(0,2){\makebox(0,0){$\scriptstyle \bullet$}}
\put(1,2){\makebox(0,0){$\scriptstyle \bullet$}}
\put(1,1){\makebox(0,0){$\scriptstyle \bullet$}}
\put(3,1){\makebox(0,0){$\scriptstyle \bullet$}}
\put(3,0){\makebox(0,0){$\scriptstyle \bullet$}}
\put(4,0){\makebox(0,0){$\scriptstyle \bullet$}}

\put(0,2){\line(1,0){1}}
\put(1,2){\line(0,-1){1}}
\put(1,1){\line(1,0){2}}
\put(3,1){\line(0,-1){1}}
\put(3,-1){\makebox(0,0)[t]{$\scriptstyle W_{6,5}$}}

\end{picture}
\kern1cm
\begin{picture}(6,6)
\put(0,0){\line(1,0){6}}
\put(0,6){\line(1,0){6}}
\put(0,0){\line(0,1){6}}
\put(6,0){\line(0,1){6}}
\multiput(1,0)(1,0){5}{\line(0,1){0.1}}
\multiput(0,1)(0,1){5}{\line(1,0){0.1}}

\put(0,1){\makebox(0,0){$\scriptstyle \bullet$}}
\put(1,1){\makebox(0,0){$\scriptstyle \bullet$}}
\put(1,0){\makebox(0,0){$\scriptstyle \bullet$}}

\put(0,1){\line(1,0){1}}
\put(1,1){\line(0,-1){1}}
\put(3,-1){\makebox(0,0)[t]{$\scriptstyle W_6$}}

\end{picture}
\kern1cm
\begin{picture}(6,6)
\end{picture}
\]
\vskip5mm

\[
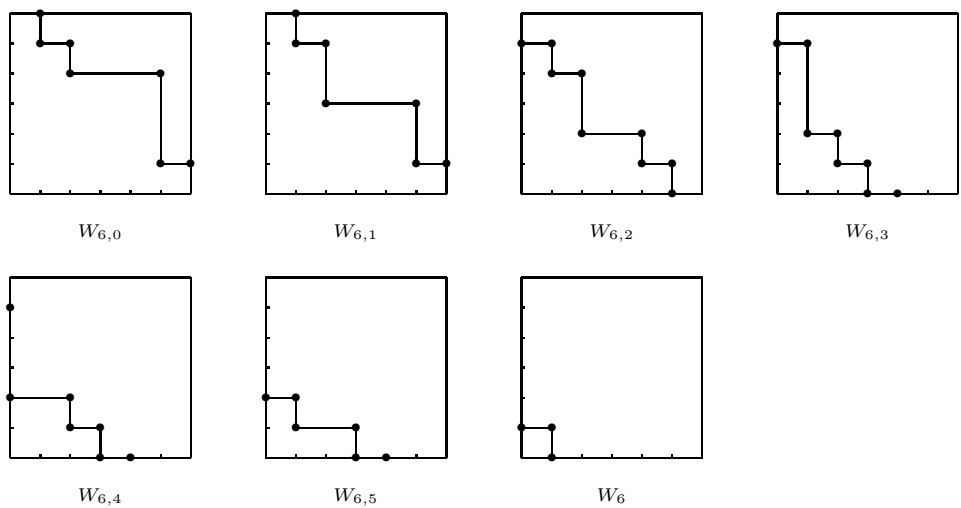
\text{Figure 29. Example~\ref{E5.1}, boundaries of the cross sections of $I$}
\]
%%%%%%%%%%%%%%%% End Figure %%%%%%%%%%%%%%%%%%%%%%%%%%%%%%%%%%%%%%

%%%%%%%%%%%% Figure 30 %%%%%%%%%%%%%%%%%%%%%%%%%%%%%%%
\vskip5mm
\setlength{\unitlength}{8mm}
\[
\begin{picture}(11,11)
\put(10,4){\vector(1,0){1}}
\put(4,10){\vector(0,1){1}}
\put(1,1){\vector(-1,-1){1}}
\put(11,3.8){\makebox(0,0)[t]{$\scriptstyle x$}}
\put(4.1,11){\makebox(0,0)[l]{$\scriptstyle y$}}
\put(0,0.2){\makebox(0,0)[b]{$\scriptstyle z$}}
\put(10,3.8){\makebox(0,0)[t]{$\scriptstyle 6$}}
\put(4.1,10.2){\makebox(0,0)[bl]{$\scriptstyle 6$}}
\put(0.9,1){\makebox(0,0)[br]{$\scriptstyle 6$}}

\put(4,10){\line(1,0){1}}
\put(5,9){\line(1,0){1}}
\put(6,8){\line(1,0){3}}
\put(6,7){\line(1,0){3}}
\put(5.5,6.5){\line(1,0){3}}
\put(5.5,5.5){\line(1,0){2}}
\put(5,5){\line(1,0){2}}
\put(7,4){\line(1,0){1}}
\put(3.5,9.5){\line(1,0){1}}
\put(3.5,8.5){\line(1,0){2}}
\put(2.5,7.5){\line(1,0){1}}
\put(2.5,4.5){\line(1,0){1}}
\put(4,5){\line(1,0){3}}
\put(1.5,3.5){\line(1,0){1}}
\put(1.5,2.5){\line(1,0){3}}
\put(1,2){\line(1,0){1}}
\put(1,1){\line(1,0){1}}
\put(9,5){\line(1,0){1}}
\put(7.5,4.5){\line(1,0){2}}
\put(3,4){\line(1,0){1}}
\put(2.5,1.5){\line(1,0){3}}
\put(3,3){\line(1,0){1}}
\put(6,3){\line(1,0){2}}
\put(8.5,3.5){\line(1,0){1}}
\put(4.5,7.5){\line(1,0){1}}
\put(4,7){\line(1,0){1}}
\put(5,4){\line(1,0){1}}

\put(5,10){\line(0,-1){1}}
\put(6,9){\line(0,-1){2}}
\put(9,8){\line(0,-1){3}}
\put(10,5){\line(0,-1){1}}
\put(9.5,4.5){\line(0,-1){1}}
\put(8,4){\line(0,-1){1}}
\put(6,4){\line(0,-1){1}}
\put(4.5,2.5){\line(0,-1){1}}
\put(2.5,3.5){\line(0,-1){2}}
\put(1,2){\line(0,-1){1}}
\put(2,2){\line(0,-1){1}}
\put(1.5,3.5){\line(0,-1){1}}
\put(3.5,9.5){\line(0,-1){1}}
\put(4.5,9.5){\line(0,-1){2}}
\put(5.5,8.5){\line(0,-1){1}}
\put(3.5,7.5){\line(0,-1){3}}
\put(2.5,7.5){\line(0,-1){3}}
\put(2,7){\line(0,-1){3}}
\put(4,7){\line(0,-1){2}}
\put(5,7){\line(0,-1){3}}
\put(7,5){\line(0,-1){1}}
\put(7.5,5.5){\line(0,-1){1}}
\put(8.5,6.5){\line(0,-1){3}}
\put(5.5,6.5){\line(0,-1){1}}
\put(4,4){\line(0,-1){1}}
\put(3,4){\line(0,-1){1}}

\put(4,10){\line(-1,-1){0.5}}
\put(5,10){\line(-1,-1){0.5}}
\put(5,9){\line(-1,-1){1.5}}
\put(6,9){\line(-1,-1){0.5}}
\put(6,7){\line(-1,-1){0.5}}
\put(9,7){\line(-1,-1){0.5}}
\put(9,5){\line(-1,-1){1}}
\put(10,5){\line(-1,-1){0.5}}
\put(10,4){\line(-1,-1){0.5}}
\put(3.5,8.5){\line(-1,-1){1.5}}
\put(2.5,4.5){\line(-1,-1){1}}
\put(5.5,5.5){\line(-1,-1){1.5}}
\put(7.5,5.5){\line(-1,-1){0.5}}
\put(7.5,4.5){\line(-1,-1){0.5}}
\put(6,4){\line(-1,-1){1.5}}
\put(6,3){\line(-1,-1){1.5}}
\put(7,3){\line(-1,-1){1.5}}
\put(1.5,2.5){\line(-1,-1){0.5}}
\put(2.5,1.5){\line(-1,-1){0.5}}
\put(3,4){\line(-1,-1){0.5}}
\put(3,3){\line(-1,-1){1}}
\put(4,5){\line(-1,-1){0.5}}
\put(4.5,7.5){\line(-1,-1){0.5}}
\put(5.5,7.5){\line(-1,-1){0.5}}
\put(5,4){\line(-1,-1){1}}
\put(8.5,3.5){\line(-1,-1){0.5}}

\end{picture}
\]

\vskip5mm
\[
\text{Figure 30. Example~\ref{E5.1}, the $A$-invariant ideal $I$}
\]
%%%%%%%%%%%%%%%% End Figure %%%%%%%%%%%%%%%%%%%%%%%%%%%%%%%%%%%%%%

%%%%%%%%%%%%%%%%%%%%%%%%%%%%%%%%%%%%%%%%%%%%
%    References
%%%%%%%%%%%%%%%%%%%%%%%%%%%%%%%%%%%%%%%%%%%%

\end{document}